\documentclass[11pt,reqno]{amsart}

\usepackage[english]{babel}
\usepackage{amsmath, amsthm, amssymb}
\usepackage{amsmath}
\usepackage{amsfonts}
\usepackage{amsthm}
\usepackage{mathrsfs}
\usepackage{amsrefs}
\usepackage{hyperref}
\usepackage{xcolor}
\usepackage{enumitem}
\usepackage{dsfont}

\theoremstyle{plain}
\newtheorem{theorem}{Theorem}[section]
\newtheorem*{theorem*}{Claim}

\newtheorem{lemma}[theorem]{Lemma}
\newtheorem{proposition}[theorem]{Proposition}

\theoremstyle{definition}
\newtheorem{example}{Example}

\newtheorem{remark}[theorem]{Remark}

\numberwithin{equation}{section}

\newcommand{\R}{\mathbb{R}}
\newcommand{\N}{\mathbb{N}}
\newcommand{\jacobi}{J_u}
\newcommand{\jacobinew}{\widetilde{J_u}}
\newcommand{\cc}{\textrm{\bf c}}
\newcommand{\bb}{\widehat{b}}
\newcommand{\vv}{b}
\newcommand{\dd}{d}
\newcommand{\id}{I}
\renewcommand{\div}{{\rm div}}
\newcommand{\tr}{{\rm tr}}
\renewcommand{\d}{\mathop{}\!\mathrm{d}}
\newcommand{\supp}{\, {\rm supp} \,}
\newcommand{\elliptic}{\lambda}
\newcommand{\bounded}{\Lambda}
\newcommand{\ariem}{\mathcal{A}_{\rm Riem}}
\newcommand{\nriem}{\textrm{\bf n}_{A}}
\newcommand{\ariemdos}{\mathcal{A}_{x}}
\newcommand{\hess}{Hu}
\newcommand{\lb}{\Delta_{\rm LB}}
\newcommand{\ricci}{{\rm Ric}}
\newcommand{\grad}{\rm \nabla_g}
\newcommand{\lcal}{\mathcal{L}}
\newcommand{\anew}{\mathcal{A}}
\newcommand{\nn}{\textrm{\bf n}}
\newcommand{\azero}{\mathcal{A}_{0}}
\newcommand{\ee}{\textrm{\bf e}}

\begin{document}
	
\title[Stable solutions to equations with variable coefficients]{
Stable solutions to semilinear elliptic equations for operators with variable coefficients
}
\author{I\~{n}igo U. Erneta}
\address{I.U. Erneta\textsuperscript{1,2} ---
\textsuperscript{1}Universitat Polit\`{e}cnica de Catalunya, Departament de Matem\`{a}tiques, Diagonal 647, 08028 Barcelona, Spain \&
\textsuperscript{2}Centre de Recerca Matem\`{a}tica, Edifici C, Campus Bellaterra, 08193 Bellaterra, Spain}
%\address{I.U. Erneta \textsuperscript{1,2},
%	\newline
%	\textsuperscript{1} 
%	Universitat Polit\`{e}cnica de Catalunya, Departament de Matem\`{a}tiques, Diagonal 647, 08028 Barcelona, Spain
%%	Departament de Matem\`{a}tiques, Universitat Polit\`{e}cnica de Catalunya (UPC), Pau Gargallo 14, 08028 Barcelona, Spain
%	\newline
%\textsuperscript{2} 
%Centre de Recerca Matem\`{a}tica, Edifici C, Campus Bellaterra, 08193 Bellaterra, Spain}
\email{inigo.urtiaga@upc.edu}

\thanks{
The author acknowledges financial support from the Spanish Ministry of Economy and Competitiveness (MINECO),
through the Mar\'{i}a de Maeztu Program for Units of Excellence in R\&D (MDM-2014-0445-18-1)
and is supported by grant MTM2017-84214-C2-1-P funded by MCIN/AEI/10.13039/501100011033 and by ``ERDF A way of making Europe''. 
He is member of the GenCat (Catalonia) research group 2017SGR1392 and of 
Centre de Recerca Matem\`{a}tica (CRM)}

\begin{abstract}
In this paper we extend the interior regularity results for stable solutions 
in [Cabr\'{e}, Figalli, Ros-Oton, and Serra, Acta Math. 224 (2020)]
to operators with variable coefficients.
We show that stable solutions to the semilinear elliptic equation 
$a_{ij}(x)u_{ij} + \vv_i(x) u_i + f(u) = 0$
are H\"{o}lder continuous in the optimal range of dimensions $n \leq 9$. 
Our bounds are independent of the nonlinearity $f \in C^1$, which we assume to be non-negative.

The main achievement of our work is to make the constants in our estimates depend on the $C^1$ norm of $a_{ij}$ and the $C^0$ norm of $\vv_i$, instead of their $C^2$ and $C^1$ norms, respectively, which arise in a first approach to the computations.
\end{abstract}

\maketitle

\section{Introduction}

Let $\Omega \subset \R^n$ be a bounded domain and $f \colon \R \to \R$ a $C^1$ function.
We consider stable solutions
$u\colon \overline{\Omega} \to \R$
to the semilinear equation
\begin{equation}\label{eq:omega}
- L u = f(u) \quad \text{ in } \Omega,
\end{equation}
where $L$ is a 
second order linear elliptic differential operator
of the form
\begin{equation}
\label{def:op}
L u = a_{ij}(x) u_{ij} + \vv_i(x) u_i
, \quad a_{ij}(x) = a_{ji}(x).
\end{equation}

The purpose of this paper is to 
extend the 
recent results 
of Cabr\'{e}, Figalli, Ros-Oton, and Serra
in \cite{CabreFigalliRosSerra} and of Cabr\'{e} in \cite{CabreQuant}
for the Laplacian
to the above operators with variable coefficients.
In \cite{CabreFigalliRosSerra},
the authors solved a 
long-standing conjecture concerning the regularity of stable solutions to semilinear problems.
They showed that stable solutions are bounded (and hence smooth)
in dimension $n \leq 9$.
This result
is optimal, since there are examples of singular stable solutions for $n \geq 10$.

In the papers \cite{CabreFigalliRosSerra} and \cite{CabreQuant},
the authors obtain universal a priori estimates that do not depend on the nonlinearity $f$.
They prove interior regularity bounds, assuming $f \geq 0$,
and boundary regularity estimates on $C^3$ domains, assuming $f \geq 0$, $f' \geq 0$, and $f'' \geq 0$.
The boundary result applies only to solutions vanishing on the boundary.

The first main interest in extending these results to operators with variable coefficients
(besides possible future applications to nonlinear problems)
is to simplify the boundary regularity arguments, even for the Laplacian.
Indeed, starting from a curved boundary, the proof of regularity in \cite{CabreFigalliRosSerra}
requires a delicate blow-up and Liouville theorem argument,
which is needed in order to apply a result of theirs only available on a flat boundary.
In addition, this proof is by contradiction-compactness and does not allow to quantify the constants in the estimates.
Recently in \cite{CabreQuant}, Cabr\'{e} has given a quantitative proof of 
this result for the Laplacian in the case of a flat boundary.
On the other hand, a curved boundary can be flattened out by a change of variables. 
Note now that, in the new coordinates, the Laplacian is written as an operator of the form \eqref{def:op}.
It is therefore natural to establish quantitative a priori estimates for our family of equations in the half-ball.
By extending the techniques in \cite{CabreQuant} to operators with variable coefficients,
we will avoid the 
intricate blow-up and Liouville theorem result form \cite{CabreFigalliRosSerra}
as well as the compactness part.

An important feature of our estimates is that they depend only on the ellipticity constants and on the norms $\|a_{ij}\|_{C^1}$ and $\|\vv_i\|_{C^0}$ of the coefficients.
The main difficulty in our proofs will be to 
obtain this dependence instead of on $\|a_{ij}\|_{C^2}$ and $\|\vv_i\|_{C^1}$,
which are the norms that appear naturally 
in a first approach to the computations.
This will be especially relevant for boundary regularity,
since it will allow us to relax the $C^3$ regularity requirement of the domain in \cite{CabreFigalliRosSerra}. 
The key point here is that,
as mentioned before,
flattening the boundary
transforms the Laplacian into an operator of the form \eqref{def:op},
with coefficients 
$a_{ij}$ and $\vv_i$
involving first and second derivatives of the boundary surface, respectively.
Thus, a $C^1$, $C^0$ bound of the coefficients $a_{ij}$, $\vv_i$ would correspond to a $C^2$ bound of the boundary.
The actual regularity of the boundary that is needed is under current investigation.

Moreover, our methods will also be useful in a future work
where we treat the case 
of the Laplacian
with non-homogeneous boundary conditions.
Recall that the previous papers \cite{CabreFigalliRosSerra} and \cite{CabreQuant}
require strongly that the solutions vanish on the boundary.
Flattening the boundary, we will be able to reduce the 
problem to 
an equation on the half-space for an operator of the form \eqref{def:op}
with zero boundary conditions and an additional source term.

The study of the regularity of stable solutions 
was initiated 
in the seventies
by Crandall and Rabinowitz in~\cite{CrandallRabinowitz}.
There,
they showed the boundedness of stable solutions
when $n \leq 9$ for exponential and power-type nonlinearities.
Their work was motivated by problems in combustion~\cite{Gelfand}, commonly known as ``Gelfand-type problems'';
for more information on these problems, we refer the reader to the monograph of Dupaigne~\cite{Dupaigne}.
Later, in the mid-nineties, Brezis~\cite{Brezis-IFT} 
asked for an extension of this regularity result to a larger class of nonlinearities.
The boundedness of stable solutions
was proven by Nedev~\cite{Nedev} for $n \leq 3$,
and by Cabr\'{e}~\cite{Cabre-Dim4} for $n = 4$.
The optimal dimension
$n \leq 9$ 
remained open
until 
it was finally reached by 
Cabr\'{e}, Figalli, Ros-Oton, and Serra 
in~\cite{CabreFigalliRosSerra}.

\subsection{The setting. Stability.}
We are interested in the class of stable solutions to the semilinear equation \eqref{eq:omega}.
We say that $u$ is a \emph{stable} solution of \eqref{eq:omega} if there exists a function 
$\varphi \in C^2(\Omega)\cap C^0(\overline{\Omega})$
such that
\begin{equation}
\label{stable:point}
\left\{
\begin{array}{rlll}
-L \varphi & \geq & f'(u) \varphi &\text{ in } \Omega,\\
\varphi & >& 0 &\text{ in } \Omega,\\
 \varphi & =& 0  &\text{ on } \partial\Omega.
\end{array}
\right.
\end{equation}
Equivalently, a solution is stable when 
the principal Dirichlet eigenvalue of the linearized equation is nonnegative.
We denote the linearized equation at $u$, the \emph{Jacobi operator}, by
\begin{equation}
\label{def:jacobi}
\jacobi \varphi = L \varphi + f'(u) \varphi.
\end{equation}

Throughout the paper we assume the coefficients of $L$ to be smooth.
Thus, our operator~\eqref{def:op} can be written in divergence form as
\begin{equation}
\label{def:op:div}
\lcal u = 
\partial_i \left(a_{ij}(x) \partial_j u \right)+\dd_i(x)\partial_i u,
\end{equation}
for certain appropriate coefficients $\dd_i$.
Now, recall that our bounds for non-divergence operators depend on the norms $\|a_{ij}\|_{C^1}$ and $\|\vv_i\|_{C^0}$ of the coefficients.
As a consequence, our results continue to hold for every divergence-form operator $\lcal$ as in \eqref{def:op:div},
with constants depending on $\|a_{ij}\|_{C^1}$ and $\|\dd_i\|_{C^0}$ instead.

We assume that the symmetric coefficient matrix $A(x) = \left(a_{ij}(x) \right)$ is 
uniformly elliptic
in $\Omega$, i.e.,
there are 
constants $\elliptic$ and $\bounded$ such that 
\begin{equation}
\label{elliptic}
0 < \elliptic \leq a_{ij}(x) p_i p_j \leq \bounded \quad 
\text{ for all } 
x \in \Omega \text{ and } p \in \R^n \text{ such that } |p |= 1.
\end{equation}
This condition will be written as $\elliptic \leq A(x) \leq \bounded$.
In particular, the matrix $A(x)$ is positive definite and defines a norm 
\begin{equation}
\label{def:norm}
|p|_{A(x)} := \left( a_{ij}(x) p_i p_j\right)^{1/2} 
\end{equation}
on vectors $p \in \R^n$.

For variational equations $- \partial_i \left( a_{ij}(x) \partial_j u \right) = f(u)$
stability is equivalent to the nonnegativity of the second variation of the associated energy functional.
This provides the useful integral inequality
\begin{equation}
\label{var:ineq}
\int_{\Omega} f'(u) \xi^2 \d x \leq \int_{\Omega} |\nabla \xi|_{A(x)}^2 \d x,
\end{equation}
satisfied by all test functions $\xi \in C^{1}_c(\Omega)$.
A key 
strategy to derive a priori estimates 
in that setting
is to choose appropriate test functions in \eqref{var:ineq}.
When chosen correctly in terms of the Jacobi operator,
the test functions allow to get rid of the nonlinearity within the proofs.
This is what is done for the Laplacian in \cite{CabreFigalliRosSerra}.

Since our operator $L$ does not have variational structure, \eqref{var:ineq} is not available.
Nevertheless, we are able to exploit the pointwise stability condition \eqref{stable:point}
for $\varphi$ to obtain a convenient
integral inequality which does not involve the function $\varphi$.
We will use it as replacement of \eqref{var:ineq} in our non-variational setting.
To derive the integral inequality,
we first write the operator $L$ in divergence form as in \eqref{def:op:div} with $\dd_i(x) = \bb_i(x)$,
where $\bb$ is the vector field given by
\begin{equation}
\label{def:b}
\bb_i(x) = \vv_i(x) - \partial_k a_{ki}(x),
\end{equation}
hence $L u = \partial_i \left( a_{ij}(x) u_j\right) + \bb_i(x) u_i$.
Now, for a test function $\xi \in C^1_c(\Omega)$,
multiply \eqref{stable:point} by $\xi^2/\varphi$ and integrate by parts
to obtain
\[
\begin{split}
\int_{\Omega} f'(u) \xi^2 \d x &\leq \int_{\Omega} \left( A(x) \nabla \varphi \cdot \nabla \left( \frac{\xi^2}{\varphi}\right) - \bb(x) \cdot \frac{\xi^2}{\varphi} \nabla \varphi \right) \d x\\
& = \int_{\Omega} \left(
-|\xi \nabla \log{\varphi}|_{A(x)}^2
+ 2 A(x) \xi \, \nabla \log{\varphi} \cdot \nabla \xi
- \xi \, \bb(x) \cdot \xi \, \nabla \log{\varphi} 
\right) \d x.
\end{split}
\]
Using that 
\[
2 A(x) \xi \, \nabla \log{\varphi} \cdot \left( \nabla \xi-{\textstyle\frac{1}{2}} \xi A^{-1}(x) \bb(x)\right)
-| \xi \nabla \log{\varphi}|_{A(x)}^2
\leq \left|\nabla \xi-{\textstyle\frac{1}{2}} \xi A^{-1}(x) \bb(x)\right|^2_{A(x)},
\]
we deduce
\begin{equation}\label{stable:int}
\int_{\Omega} f'(u) \xi^2 \d x 
\leq \int_{\Omega} 
\left|\nabla \xi-{\textstyle\frac{1}{2}} \xi A^{-1}(x) \bb(x)\right|^2_{A(x)}
\d x \quad \text{ for all } \xi \in C^1_c(\Omega).
\end{equation}

We remark that, in general,~\eqref{stable:int} is not equivalent to the stability condition~\eqref{stable:point}.
The main reason is that our equation $-L u = f(u)$ is not variational due to the presence of $\bb$ when written in 
the divergence form \eqref{def:op:div}.
In Appendix~\ref{app:equivalence},
we give an example of the non-equivalence,
and, at the same time, we characterize the drifts $\bb$ for which the equivalence holds.

\subsection{Main results}
This paper concerns the interior regularity of stable solutions.
Boundary regularity results will be treated in a forthcoming work.
Therefore, it suffices to consider stable solutions 
to $- L u = f(u)$ in the unit ball $B_1$.
A constant depending only on $n$, $\elliptic$, and $\bounded$ 
will be called \emph{universal},
a terminology that we use throughout the paper.

The following is our main result,
which provides interior a priori estimates for stable solutions:
a H\"{o}lder bound when $n \leq 9$,
and a $W^{1, 2 + \gamma}$ estimate in every dimension.
The only requirement for the nonlinearity is $f \geq 0$, as in \cite{CabreFigalliRosSerra}.
An important accomplishment in our estimates is that they involve the norms 
$\|A\|_{C^1}$ and $\|\vv\|_{C^0}$,
while a first approach to the problem leads to computations including 
second derivatives of $A$ and first derivatives of $\vv$.
On the other hand, our bounds are independent of $f$.

\begin{theorem}
\label{thm:main}
Let $u \in C^{\infty}(\overline{B}_{1})$ be a stable solution of $-L u = f(u)$ in $B_1 \subset \R^n$,
for some nonnegative function $f \in C^1(\R)$.

Then
\begin{equation}
\label{eq:high:int}
\|\nabla u\|_{L^{2 + \gamma}(B_{1/2})} \leq C \|u\|_{L^1(B_1)},
\end{equation}
where $\gamma = \gamma(n) > 0$
and $C = C(n, \elliptic, \|A\|_{C^1(\overline{B}_1)}, \|\vv\|_{C^0(\overline{B}_1)})$.
In addition,
\begin{equation}
\label{eq:holder}
\|u\|_{C^{\alpha}(\overline{B}_{1/2})} \leq C \|u\|_{L^1(B_1)} \quad \text{ if } n \leq 9,
\end{equation}
where 
$\alpha = \alpha(n, \elliptic, \bounded) > 0$
and  $C = C(n, \elliptic, \|A\|_{C^1(\overline{B}_1)}, \|\vv\|_{C^0(\overline{B}_1)})$.
\end{theorem}

For applications, it may be useful to point out that
the result only needs $A$ to be Lipschitz and $\vv$ to be bounded.
Our direct computations within the proofs assume 
$A \in C^1$ and $\vv \in C^1$
in order to evaluate
certain identities pointwise.
However, we only need these to be meaningful in a weaker sense;
see Remarks~\ref{sz:lessreg} and~\ref{rad:lessreg}.
Similarly,
we only need 
$u$ to be $C^2$ and to have weak third derivatives.
These last conditions seem to require more regularity of the drift $\vv$; see Remark~\ref{sz:lessreg}.

The proof of Theorem~\ref{thm:main} will rely
on our second main result, 
Theorem~\ref{thm:sz} below, and its consequences.
It consists of 
two types of Hessian estimates.
The first one, \eqref{ineq:sz},
is an extension of 
the geometric stability condition due to Sternberg and Zumbrun~\cite{SternbergZumbrun1}
to operators with variable coefficients.
The second one, \eqref{ineq:hessgrad1}-\eqref{ineq:hessgrad2},
controls the $L^1$ norm of the ``Hessian times the gradient'', $|D^2 u| |\nabla u|$, 
in balls and annuli
by the $L^2$ norm squared of the gradient 
whenever the lower order coefficients are small.

\begin{theorem}\label{thm:sz}
Let $u \in C^{\infty}\big(\overline{B}_{1}\big)$ be a stable solution of 
$-L u = f(u)$ in $B_1$,
for some function $f\in C^1(\R)$.
Assume that 
\[
\| D A \|_{C^0(\overline{B}_1)} + \| \vv \|_{C^0(\overline{B}_1)} \leq \varepsilon
\]
for some $\varepsilon > 0$.

Then 
\begin{equation}
\label{ineq:sz}
\begin{split}
\int_{B_1}  \anew^{2}  \eta^2 \d x & \leq 
\int_{B_1}  |\nabla u|_{A(0)}^2\left( |\nabla \eta|^2_{A(x)} + C \varepsilon |\nabla (\eta^2)| + C \varepsilon^2 \eta^2 \right)\d x \\
& \quad \quad \quad + C \varepsilon \int_{B_1} |D^2 u| |\nabla u| \eta^2  \d x
\end{split}
\end{equation}
for all $\eta \in C^{\infty}_c(B_1)$,
where $C$ is a universal constant
and
\begin{equation}
\label{def:a}
\anew := 
\left\{
\begin{array}{ll}
\Big(\tr(A(x) D^2 u A(0) D^2 u) -|\nabla u|_{A(0)}^{-2}  |D^2 u A(0) \nabla u|^2_{A(x)} \Big)^{1/2} & \text{ if } \nabla u \neq 0 \\
0 & \text{ if } \nabla u = 0.
\end{array}
\right.
\end{equation}

Assume in addition that $f \geq 0$.
Then there is a universal constant $\varepsilon_0 > 0$ with the following property:
if $\varepsilon \leq \varepsilon_0$,
then
\begin{equation}
\label{ineq:hessgrad1}
\| \,|\nabla u| \, D^2 u \, \|_{L^1(B_{3/4})} \leq C \|\nabla u\|^2_{L^2(B_1)}
\end{equation}
and
\begin{equation}
\label{ineq:hessgrad2}
\| \,|\nabla u| \, D^2 u \, \|_{L^1(B_{1/2} \setminus B_{1/4})} \leq C \|\nabla u\|^2_{L^2(B_1 \setminus B_{1/8})},
\end{equation}
where $C$ is a universal constant.
\end{theorem}

Our first inequality \eqref{ineq:sz} 
generalizes the Sternberg-Zumbrun estimate for the Laplacian,
which corresponds to the case $\varepsilon = 0$ and $A(0) = \id$.
The peculiar form of the function $\anew$ in \eqref{def:a}
(the coefficients of $A$ are evaluated both at $x$ and $0$)
will guarantee that the direct computations within our proofs give dependence on the norm $\|A\|_{C^1}$, instead of $\|A\|_{C^2}$
for other choices of $\anew$.
In this direction,
it is worth noting that
the classical Sternberg-Zumbrun result and the function $\anew$
have a Riemannian analogue
(found by Farina, Sire, and Valdinoci~\cite{FarinaSireValdinoci})
which can be related
to our Euclidean setting with variable coefficients.
The estimate from the Riemannian framework leads to bounds depending
on $\|A\|_{C^2}$.
We elaborate on these topics further in Remarks~\ref{remark:const} and \ref{remark:riem}.

The ``Hessian times the gradient'' estimates \eqref{ineq:hessgrad1}-\eqref{ineq:hessgrad2}
rely on the inequality \eqref{ineq:sz}
with sufficiently small errors $\varepsilon$,
and will require the assumption  $f \geq 0$.
While the bound on annuli \eqref{ineq:hessgrad2} can be deduced from the one in balls \eqref{ineq:hessgrad1}
by a standard scaling and covering argument,
we include it in the statement since it will be crucial in the proof of the H\"{o}lder estimate in Theorem~\ref{thm:main}.

\subsection{Structure of the proof}

By a scaling and covering argument, it suffices to obtain the a priori estimates from Theorem~\ref{thm:main} in small balls.
There, the problem can be written as an equation in the unit ball
involving an operator $L$ that is close to the Laplacian, i.e.,
whose coefficients satisfy $A(0) = \id$ and $\|D A \|_{C^0(\overline{B}_1)} + \|\vv\|_{C^0(\overline{B}_1)} \leq \varepsilon$, with $\varepsilon$ small.
We explain this in more detail in Section~\ref{section:prelim} below.

The key estimates leading to Theorem~\ref{thm:main}
are contained in
Propositions~\ref{prop:l2l1}, \ref{prop:rad}, and \ref{prop:l1rad} below.
Our proofs are all quantitative as in the paper~\cite{CabreQuant}
and avoid the compactness argument from the previous work~\cite{CabreFigalliRosSerra}.
The proofs of the first two propositions 
use the Hessian estimates of Theorem~\ref{thm:sz} above.
In particular, this forces us to prove the Sternberg-Zumbrun inequality
before the crucial weighted $L^2$ estimate for the radial derivative (Proposition~\ref{prop:rad}).
It is worth noting that,
for the Laplacian, these two results are independent from each other
(and  hence can be obtained in any order, as in the works \cite{CabreFigalliRosSerra} and \cite{CabreQuant}),
while this is no longer the case for operators with variable coefficients.

In the first proposition,
we control the $L^2$ norm of the gradient by the $L^1$ norm of the solution
when the error $\varepsilon$ of the coefficients is sufficiently small.
This is a direct consequence of Theorem \ref{thm:sz}
and the interpolation inequalities of Cabr\'{e} in~\cite{CabreQuant}.
We recall these inequalities in Appendix~\ref{app:interpolation}.
\begin{proposition}
\label{prop:l2l1}
Let $u \in C^{\infty}\big(\overline{B}_{1}\big)$ be a stable solution of 
$-L u = f(u)$ in $B_1$,
for some nonnegative function $f\in C^1(\R)$.
Assume that 
\[
\| D A \|_{C^0(\overline{B}_1)} + \| \vv \|_{C^0(\overline{B}_1)} \leq \varepsilon
\]
for some 
$\varepsilon > 0$.

Then there exists a universal $\varepsilon_0 > 0$ 
with the following property:
if $\varepsilon \leq \varepsilon_0$, then 
\begin{equation}\label{eq:grad2by1}
 \|\nabla u\|_{L^{2}(B_{1/2})}  \le C \| u\|_{L^{1}(B_{1})},
\end{equation}
where $C$ is a universal constant.
\end{proposition}

The second proposition
is a weighted $L^2$ estimate for the radial derivative in a ball
by 
the full gradient in an annulus.
It is here that we need $n \leq 9$.
Again, we will assume that the coefficient error $\varepsilon$ is small
and that the nonlinearity is nonnegative $f \geq 0$.
Here and throughout the paper we use the notation 
\[
r = |x|, \quad u_r = \dfrac{x}{|x|} \cdot \nabla u
\]
for the radial derivative.

\begin{proposition}
\label{prop:rad}
Let $u \in C^{\infty}\big(\overline{B}_{1}\big)$ 
be a stable solution of 
$- L u = f(u)$ in $B_1$,
for some nonnegative
function $f \in C^1(\R)$.
Assume that 
\[
\| D A \|_{C^0(\overline{B}_1)} + \| \vv \|_{C^0(\overline{B}_1)} \leq \varepsilon
\]
for some $\varepsilon > 0$.

Then there is a universal $\varepsilon_0 > 0$ with the following property:
if $\varepsilon \leq \varepsilon_0$ and $3 \leq n \leq 9$, then
\begin{equation}
\label{ineq:rad:decay}
\begin{split}
\int_{B_{\rho}} r^{2-n} u_r^2 \d x &\leq C \int_{B_{2\rho} \setminus B_{\rho}} r^{2-n} |\nabla u|^2 \d x + C \varepsilon \int_{B_{4\rho}} r^{3-n} |\nabla u|^2 \d x
\end{split}
\end{equation}
for all $\rho \leq 1/4$, where $C$ is a universal constant.
\end{proposition}
Notice that this result requires $n \geq 3$.
However, adding superfluous variables to the solution,
we can also use it when $n \leq 2$.

Our inequality \eqref{ineq:rad:decay} in Proposition~\ref{prop:rad}
is an analogue of Lemma 2.1 in \cite{CabreFigalliRosSerra}, where the authors obtain 
a similar 
bound for the Laplacian ($\varepsilon = 0$ and $A(0) = \id$)
without the nonnegativity assumption on $f$.
Recall that this assumption is needed in the Hessian estimates~\eqref{ineq:hessgrad1}-\eqref{ineq:hessgrad2} in Theorem~\ref{thm:sz} above,
which will allow us to treat a weighted $|D^2 u| |\nabla u|$ error term 
which does not appear for the Laplacian.
We will be able to control this error by writing it as an infinite sum on dyadic annuli,
pulling the weight out of the integral in each annulus,
and applying the bound \eqref{ineq:hessgrad2}.

Finally in the third proposition
we show that,
under the assumption that $A(0) = \id$,
(generalized) superharmonic functions
are controlled by the radial derivative
plus an error involving the full gradient in $L^1$.
This is an extension of Lemma~4.1 in Cabr\'{e}~\cite{CabreQuant} 
to operators with variable coefficients.

\begin{proposition}
\label{prop:l1rad}
Let $u \in C^{\infty}(\overline{B}_{1})$ be 
superharmonic in the sense that
$L u \leq 0$ in $B_1$.
Assume that 
\[
A(0) = \id
\quad \text{ and } \quad 
\|D A\|_{C^0(\overline{B}_1)} + \| \vv\|_{C^0(\overline{B}_1)} \leq \varepsilon
\]
for some $\varepsilon > 0$.

Then
there exists a constant $t$, which depends on $u$,
such that
\[
\|u - t\|_{L^{1}( B_1 \setminus B_{1/8})} \leq C \| u_{r} \|_{L^1( B_1 \setminus B_{1/8})} + C \varepsilon \|\nabla u \|_{L^{1}(B_1)},
\]
where $C$ is a constant depending only on $n$ and $\elliptic$.
\end{proposition}

Our proof of Proposition~\ref{prop:l1rad} 
is by comparison with harmonic functions.
The required $L^1$ estimates for harmonic functions
will follow by duality from the $L^{\infty}$ bounds of a Neumann problem.
Such $L^{\infty}$ estimates are the most technical part of the argument,
where we use a Moser iteration in the spirit of Winkert~\cite{Winkert} to deduce the uniform bounds.

The $W^{1, 2+\gamma}$ 
result \eqref{eq:high:int} in Theorem~\ref{thm:main}
will follow from the Hessian estimate on balls \eqref{ineq:hessgrad1} in Theorem~\ref{thm:sz} 
together with Proposition~\ref{prop:l2l1}.
To show it, 
we first control the $L^2$ norm of the gradient uniformly on level sets by 
the Dirichlet integral in a ball, and hence by the $L^1$ norm of the solution.
A device from \cite{CabreFigalliRosSerra} will then allow us to deduce the higher integrability.

To prove the H\"{o}lder estimate \eqref{eq:holder} in Theorem~\ref{thm:main},
we will show that the scale-invariant weighted integral $\int_{B_{\rho}} r^{2-n} |\nabla u|^2 $ decays algebraically.
In the previous works \cite{CabreFigalliRosSerra} and \cite{CabreQuant},
the authors proved the decay of the weighted radial derivative instead.
They could later deduce the $C^{\alpha}$ estimate
by either averaging or applying a version of Morrey's embedding for radial derivatives.
Here we will obtain the decay of the full gradient directly for the first time.
For this,
combining Propositions~\ref{prop:l2l1} and~\ref{prop:l1rad},
we are able to bound the full gradient by the radial derivative on annuli in $L^2$.
This, together with the dyadic decomposition explained above, allows us to control the weighted integral of the gradient by that of the radial derivative (up to gradient errors).
Now, Proposition~\ref{prop:rad} 
will
yield a control of the weighted integral of the gradient in the ball by the same quantity on an annulus.
Finally, a standard iteration will lead to decay.

Our integral stability inequality \eqref{stable:int}
will be crucial in the proofs of both Theorem~\ref{thm:sz} and Proposition~\ref{prop:rad}.
These will follow from \eqref{stable:int}, with $\Omega = B_1$,
by choosing appropriate test functions in terms of the Jacobi operator,
as we explain next.
Taking a test function of the form $\xi = \cc \eta$, where $\cc$ and $\eta$ are smooth and $\supp \eta \subset B_1$, 
the integrand on the right-hand side of \eqref{stable:int} becomes
\begin{equation}
\label{comp1}
\begin{split}
&\left|\nabla \xi-{\textstyle\frac{1}{2}} \xi A^{-1}(x) \bb(x)\right|^2_{A(x)}
=
\left| \eta \nabla \cc + \cc \nabla \eta -{\textstyle\frac{1}{2}} \cc \eta A^{-1}(x) \bb(x)\right|^2_{A(x)} \\
& \quad = 
\left| \eta \nabla \cc \right|_{A(x)}^2 
+ 2 A(x) \eta \nabla \cc \cdot \cc \nabla \eta 
- \eta^2 \cc \, \bb(x) \cdot \nabla \cc 
+ \cc^2 \left| \nabla \eta -{\textstyle\frac{1}{2}} \eta A^{-1}(x) \bb(x)\right|^2_{A(x)}.
\end{split}
\end{equation}
Integrating in $B_1$, the first term $\int_{B_1} |\eta \nabla \cc|^2_{A(x)} \d x$ in \eqref{comp1} can be integrated by parts as
\[
\int_{B_1} \left| \eta \nabla \cc \right|_{A(x)}^2 \d x 
= \int_{B_1} \Big( - \div\left( A(x) \nabla \cc \right) \, \cc \eta^2 - 2 A(x) \eta \nabla \cc \cdot \cc \nabla \eta \Big) \d x
\]
and hence, 
by \eqref{comp1}, \eqref{stable:int}, and rearranging terms,
it follows that
\begin{equation}\label{stable:jacobi}
\int_{B_1} (\jacobi \cc) \, \cc \eta^2 \d x 
\leq \int_{B_1} 
\cc^2 \left|\nabla \eta-{\textstyle\frac{1}{2}} \eta A^{-1}(x) \bb(x)\right|^2_{A(x)}
\d x.
\end{equation}

The key idea now is to choose the $\cc$ function
in a way such that
$\jacobi \cc$ becomes independent of the nonlinearity.
This will yield universal a priori estimates for stable solutions.
Since $u$ solves the equation $-L u = f(u)$,
taking a derivative, we have that 
$f'(u) \nabla u = - \nabla L u$
and hence 
$\jacobi \nabla u = L \nabla u - \nabla L u$ 
no longer involves $f$.
This computation suggests that we choose $\cc$ as 
a function of the gradient of $u$.
Thus,
to prove the estimate for $\anew$ in Theorem~\ref{thm:sz}, we will make the choice
\[
\cc(x) = |\nabla u(x)|_{A(0)} = \big( a_{ij}(0) u_i u_j\big)^{1/2}.
\]
On the other hand, the weighted $L^2$ bound
in Proposition~\ref{prop:rad}
will also require to choose the auxiliary function $\eta$ above carefully.
The test functions leading to this estimate are
\[
\cc(x) = x \cdot \nabla u = r u_r \quad \text{ and } \quad \eta(x) = |x|_{A^{-1}(0)}^{\frac{2-n}{2}} \zeta,
\]
where $\zeta \in C^\infty_c(B_1)$ is a cut-off.

We note that 
our
test functions 
above
are the 
ones used in the paper \cite{CabreFigalliRosSerra}
under the linear transformation $x \mapsto A^{1/2}(0) x$,
where $A^{1/2}(0)$ is the positive square root of the matrix $A(0)$.
These seem to be
the simplest functions
leading to a priori estimates 
in the variable coefficients framework.
Moreover,
thanks to the particular form of these functions,
all direct computations within our proofs will only involve first derivatives of the coefficients $A$ and $\vv$,
while other choices of functions require two derivatives of $A$.
A suitable integration by parts will yield bounds in terms of the norms 
$\|A\|_{C^1}$ and $\|\vv\|_{C^0}$, as in the results mentioned above.

\subsection{Outline of the article}
In Section~\ref{section:prelim} we briefly comment on the invariance of stability under affine transformations.
Section~\ref{section:hessian} is devoted to proving Theorem~\ref{thm:sz} and Proposition~\ref{prop:l2l1}.
In Section~\ref{section:higher:int} we prove the $W^{1, 2+\gamma}$ bound \eqref{eq:high:int} from Theorem~\ref{thm:main}.
Section~\ref{section:radial} contains the proof of Proposition~\ref{prop:rad}.
In Section~\ref{section:radl1} we prove Proposition~\ref{prop:l1rad}.
Finally, in Section~\ref{section:holder} we prove the H\"{o}lder bound \eqref{eq:holder} in Theorem~\ref{thm:main}.
In Appendix~\ref{app:equivalence} we show that the stability condition~\eqref{stable:point} is not equivalent to the integral inequality~\eqref{stable:int}.
Appendix~\ref{app:trace:sobolev}
contains an elementary proof of the Sobolev trace inequality in the ball needed in the Moser iteration leading to Proposition~\ref{prop:l1rad}.
In Appendix~\ref{app:interpolation}
we recall the interpolation inequalities of Cabr\'{e}~\cite{CabreQuant}.
In Appendix~\ref{app:simon} we recall Simon's lemma~\cite{Simon} for absorbing the errors in larger balls.

\section{Preliminaries: Invariance under affine transformations}
\label{section:prelim}

To prove Theorem~\ref{thm:main}, we will analyze the semilinear equation $-L u = f(u)$
in small balls.
Since the class of stable solutions is invariant under affine transformations,
the question reduces to studying an equation in the unit ball involving
an operator that is close to the Laplacian.
After proving 
the necessary estimates in this setting,
the theorem will follow from a scaling and covering argument.
It is worth mentioning that the nonnegativity of the nonlinearity,
which is required in our main results,
is preserved under these transformations.

We now explain this invariance with more detail in different particular situations.
First we study the equation under translations and scalings.
These simple yet important transformations will be used several times throughout the paper.
Secondly,
we consider the equation under general linear transformations.
These allow us to reduce ourselves to the case where the coefficient matrix is the identity at the origin.
Notice that this is only required in 
Proposition~\ref{prop:l1rad}, but will be crucial in the proof of the $C^{\alpha}$ bound \eqref{eq:holder} from Theorem~\ref{thm:main} given in Section~\ref{section:holder}.

As mentioned in the Introduction, the bounds in our a priori estimates depend only on the ellipticity constants $\elliptic$ and $\bounded$
and on the quantity
\begin{equation}
\label{norms}
\|D A\|_{C^0(\overline{B}_1)} +  \|\vv\|_{C^0(\overline{B}_1)}
\end{equation}
concerning the coefficients.
As we will see now, the two norms 
in \eqref{norms} 
have the same scaling.
It is therefore natural to state our results in terms of this quantity.

\begin{enumerate}
[label= (\roman*)]
\item \label{scaling} \textbf{Translation and scale invariance}.
If $u$ is a stable solution of $-L u = f(u)$ in a ball $B_{\rho}(y)$,
then the function $u^{y, \rho} := u(y + \rho \, \cdot)$
is a solution of 
$-L^{y, \rho} u^{y, \rho} = \rho^2 f(u^{y, \rho})$ in $B_{1}$,
where $L^{y, \rho}$ is the operator 
\[
L^{y, \rho} v = \tr\left(A^{y, \rho}(x) D^2 v\right) + \vv^{y, \rho}(x) \cdot \nabla v,
\]
with coefficients
\[
A^{y, \rho}(x) = A(y + \rho \, x) \quad \text{ and } \quad \vv^{y, \rho}(x) = \rho \, \vv(y + \rho \, x).
\]
The stability condition~\eqref{stable:point} in $B_{\rho}(y)$
becomes $-L^{y, \rho} \varphi^{y, \rho} \leq \rho^2 f'(u^{y, \rho}) \varphi^{y,\rho}$ in $B_{1}$,
where $\varphi^{y, \rho} = \varphi(y + \rho \, \cdot)$
satisfies the assumptions in \eqref{stable:point},
and hence $u^{y, \rho}$ is stable.

Since the coefficients satisfy the bounds
\[
\|D A^{y, \rho}\|_{C^0(\overline{B}_1)} \leq \rho \| D A\|_{C^0(\overline{B}_{\rho}(y))}
\quad \text{ and } \quad 
\| \vv^{y, \rho} \|_{C^0(\overline{B}_1)} \leq \rho \| \vv\|_{C^0(\overline{B}_{\rho}(y))},
\]
whenever $B_{\rho}(y) \subset B_R$ for some $R > 0$
and $L$ is defined in this larger ball, we have
\begin{equation}
\label{unif:old}
\|D A^{y, \rho}\|_{C^0(\overline{B}_1)} + \| \vv^{y, \rho} \|_{C^0(\overline{B}_1)} \leq \rho \left( \| D A\|_{C^0(\overline{B}_{R})} + \| \vv \|_{C^0(\overline{B}_{R})}\right),
\end{equation}
which can be made small for $\rho$ small.
In particular, 
assuming that $L$ is close to the Laplacian as in the statements of our propositions and
by \eqref{unif:old}, we deduce the following property:
if $\| D A\|_{C^0(\overline{B}_{R})} + \| \vv \|_{C^0(\overline{B}_{R})} \leq \varepsilon$, then
\[
\|D A^{y, \rho}\|_{C^0(\overline{B}_1)} + \| \vv^{y, \rho} \|_{C^0(\overline{B}_1)} \leq \rho \varepsilon \quad \text{ for all } B_{\rho}(y) \subset B_R.
\]
This elementary observation will be used throughout the paper.

\item \label{invariance} \textbf{Invariance under linear transformations}.
Given a symmetric positive definite matrix $M \in \R^{n\times n}$,
if $u$ is a stable solution of $-L u = f(u)$ in the unit ball $B_{1}$,
then the function 
$u^{M} := u(M \, \cdot)$
is a solution of 
$-L^{M} u^{M} = f(u^{M})$ in $M^{-1}(B_{1})$,
where $L^M$ is the operator 
\[
L^M v = \tr\big(A^M(x) D^2 v\big) + \vv^M(x) \cdot \nabla v,
\]
with coefficients 
\[
A^{M}(x) = M^{-1} A(M x) M^{-1} \quad \text{ and } \quad \vv^{M}(x) = M^{-1} \vv(M x).
\]
As above, the stability condition \eqref{stable:point} in $B_1$ becomes $-L^{M} \varphi^{M} \leq f'(u^{M}) \varphi^{M}$ in $M^{-1}(B_1)$, where $\varphi^{M} = \varphi(M \cdot)$,
and hence $u^{M}$ is a stable solution.

If $M$ satisfies $\sqrt{\elliptic} \leq M \leq \sqrt{\bounded}$, then the new coefficients $A^{M}$ are uniformly elliptic with 
$\elliptic/\bounded \leq A^{M}(x) \leq \bounded/ \elliptic$.
Moreover, we have the bounds
\[
|D A^{M}(x)|
\leq \frac{\sqrt{\bounded}}{\elliptic} |D A(M x)|
\quad \text{ and } \quad
|\vv^{M}(x)| 
\leq \frac{1}{\sqrt{\elliptic}} |\vv(M x)|,
\]
and taking the supremum in $x \in M^{-1} (B_1)$, 
using that 
$B_{1/\sqrt{\bounded}} \subset M^{-1}(B_1)$, we deduce
\begin{equation}
\label{unif:new}
\begin{split}
\|D A^{M}\|_{C^0(\overline{B}_{1/\sqrt{\bounded}})} + \|\vv^{M}\|_{C^0(\overline{B}_{1/\sqrt{\bounded}})} 
& \leq \frac{\sqrt{\bounded}}{\elliptic}\left(\|D A\|_{C^0(\overline{B}_1)}
+ \|\vv\|_{C^0(\overline{B}_1)}\right).
\end{split}
\end{equation}
In particular,
taking $M = A^{1/2}(0)$ as the unique positive square root of $A(0)$,
we see that $u^{A^{1/2}(0)}$ solves an elliptic equation in the ball $B_{1/\sqrt{\bounded}}$
with coefficients $(A^{A^{1/2}(0)})(x)$ satisfying $(A^{A^{1/2}(0)})(0) = \id$, i.e.,
equal to the identity at the origin.
By the monotonicity of the principal eigenvalue with respect to the domain, 
it follows that $u^{A^{1/2}(0)}$ is also a stable solution in this ball.

It is now easy to combine 
these transformations with the ones given in the first part \ref{scaling}.
For each ball $B_{\rho}(y) \subset B_1$,
the function 
$\widetilde{u} = u\big(y + \frac{\rho}{\sqrt{\bounded}} A^{1/2}(y) \, \cdot\big)$
is a stable solution of an elliptic equation 
$- \widetilde{L} \widetilde{u} = \widetilde{f}(\widetilde{u})$ 
in $B_1$.
Here, $\widetilde{f}$ is the nonlinearity
$\widetilde{f} = \frac{\rho^2}{\bounded} f$,
while
$\widetilde{L}$ is an operator of the form \eqref{def:op} with coefficients
\[
\textstyle
\widetilde{A}(x) = A^{-1/2}(y) A\big(y + \frac{\rho}{\sqrt{\bounded}} A^{1/2}(y) x \big)A^{-1/2}(y)
\]
and
\[
\textstyle
\widetilde{\vv}(x) = \frac{\rho}{\sqrt{\bounded}}A^{-1/2}(y) \vv\big(y + \frac{\rho}{\sqrt{\bounded}} A^{1/2}(y) x \big).
\]
Notice that the matrix $\widetilde{A}(x)$ 
is uniformly elliptic with $\elliptic/\bounded \leq \widetilde{A}(x) \leq \bounded/ \elliptic$
and is equal to the identity at the origin.
Furthermore,
combining \eqref{unif:new} and \eqref{unif:old},
the coefficients can be bounded by
\[
\|D \widetilde{A}\|_{C^0(\overline{B}_1)} + \|\widetilde{\vv}\|_{C^0(\overline{B}_1)}
 \leq 
 \frac{\rho}{\elliptic}\left(\|D A\|_{C^0(\overline{B}_1)}
+ \|\vv\|_{C^0(\overline{B}_1)}\right).
\]
As mentioned above,
this observation will be important in the proof of the H\"{o}lder estimate~\ref{eq:holder} 
in Section~\ref{section:holder} below.
\end{enumerate}

\section{Hessian and $W^{1,2}$ estimates}
\label{section:hessian}

The goal of this section is to prove
Theorem~\ref{thm:sz} and the energy estimate Proposition~\ref{prop:l2l1}.

Recall the function $\anew \colon \overline{B}_1 \to \R$ introduced in \eqref{def:a} in the statement of Theorem~\ref{thm:sz}.
This function can also be written as
\begin{equation}
\label{def:anew}
\anew =
\left\{\begin{array}{ll}
\Big( \|A^{1/2}(x)D^2 u A^{1/2}(0)\|_{\rm HS}^2 
- 
|A^{1/2}(x) D^2 u A^{1/2}(0) \nn(x)|^2
\Big)^{1/2} & \text{if } \nabla u \neq 0\\
        0 & \text{if } \nabla u = 0,
\end{array}\right.
\end{equation}
where $\|\cdot \|_{\rm HS}$ denotes the Euclidean Hilbert-Schmidt norm for matrices\footnote{Recall that, for a matrix $M \in \R^{n\times n}$, this norm squared is $\|M\|^2_{\rm HS} = \tr (M^{T} M) = \sum_{i, j = 1}^{n} M_{ij}^2$.}
and $\nn(x)$ is the unit vector field
$\nn \colon \overline{B}_1 \cap \{ \nabla u \neq 0\} \to \R$
given by
\begin{equation}
\label{n:vector}
\nn(x) := |\nabla u|^{-1}_{A(0)} A^{1/2}(0) \nabla u(x).
\end{equation}
The equivalence between the expressions \eqref{def:a} and \eqref{def:anew} follows from the identities
\[
\|A^{1/2}(x)D^2 u A^{1/2}(0)\|_{\rm HS}^2 = \tr\left(A(x) D^2 u A(0) D^2 u\right)
\]
and
\[
|A^{1/2}(x) D^2 u A^{1/2}(0) \nn(x)|^2 = |\nabla u|_{A(0)}^{-2} |D^2 u A(0) \nabla u|_{A(x)}^2,
\]
which are easy to check.

We start by proving the bound \eqref{ineq:sz} in Theorem~\ref{thm:sz},
which is a generalization of 
the geometric stability inequality due to Sternberg and Zumbrun~\cite{SternbergZumbrun1}
for stable solutions to $- \Delta u = f(u)$.
For this, we will test the integral stability inequality \eqref{stable:jacobi} with the function
\[
\cc(x) = |\nabla u|_{A(0)}
\]
and a cut-off $\eta$.
The proof of the remaining estimates in 
Theorem~\ref{thm:sz}
will rely on this preliminary inequality.

Two comments are in order.
First, with this choice of $\cc$,
our result requires an appropriate integration by parts to allow dependence of the bounds on only $\|\vv\|_{C^{0}}$.
Secondly,
after the proof,
in Remarks \ref{remark:const} and \ref{remark:riem} we will comment on alternative choices of $\cc$ and of the function $\anew$.

We originally took $|\nabla u|_{A(x)}$ as our $\cc$ function,
a choice that required the regularity $A \in C^2$ and $\vv \in C^1$ 
when computing $\jacobi \cc$ in the stability inequality \eqref{stable:jacobi}.
With that choice,
a further integration by parts was needed to obtain bounds depending only on $\|A\|_{C^1}$.
Instead,
our function $|\nabla u|_{A(0)}$ only needs $A \in C^1$ and $\vv \in C^1$.
Moreover, 
the proof with our choice
is easier and we only need an integration by parts to get rid of the first derivatives of $\vv$.
Note that 
the function $|\nabla u|_{A(x)}$ is motivated by geometric considerations
and had already appeared in the Riemannian analogue of the Sternberg-Zumbrun estimates, as explained in Remark~\ref{remark:riem}.

\begin{proof}[Proof of \eqref{ineq:sz} in Theorem~\ref{thm:sz}]
Since $|\nabla u|_{A(0)}$ is not necessarily smooth when $\nabla u = 0$,
we consider the smooth function
\[
\cc_{\delta} 
:= \sqrt{|\nabla u|_{A(0)}^2 + \delta^2}
\]
instead.
We will apply the stability inequality \eqref{stable:jacobi} with $\cc = \cc_{\delta}$.
In the end we will let $\delta \to 0$, which will yield the claim.
Throughout this proof, the letter $C$ denotes a generic universal constant.

By the stability inequality 
\eqref{stable:jacobi}, we have the upper bound
\begin{equation}
\label{au6new}
\begin{split}
\int_{B_1} \cc_{\delta} \, \jacobi \cc_{\delta} \, \eta^2\d x &
\leq \int_{B_1} \left(|\nabla u|_{A(0)}^2 + \delta^2\right) \left|\nabla \eta-{\textstyle\frac{1}{2}} \eta A^{-1}(x) \bb(x)\right|^2_{A(x)} \d x \\
& \leq \int_{B_1} \left(|\nabla u|_{A(0)}^2 + \delta^2\right) \left(
|\nabla \eta|^2_{A(x)} + C \varepsilon |\nabla (\eta^2)| + C \varepsilon^2 \eta^2
\right)
\d x,
\end{split}
\end{equation}
where in the last line we have expanded the quadratic expression and applied the bounds of the coefficients.

We will bound the 
expression
$\cc_{\delta} \, \jacobi \cc_{\delta} = \cc_{\delta} \, L \cc_{\delta} + f'(u) \cc_{\delta}^2$
from below.
Since
\[
\begin{split}
\partial_i \cc_{\delta} 
& =  \cc_{\delta}^{-1}\,  u_{ik } a_{kl}(0) u_{l}
\end{split}
\]
and
\[
\begin{split}
\partial^2_{ij} \cc_{\delta}
&= 
\cc_{\delta}^{-1} \, u_{ijk} a_{kl}(0) u_{l} + \cc_{\delta}^{-1} u_{ik} a_{kl}(0) u_{j l } -\cc_{\delta}^{-3} \, u_{ik} a_{kl}(0) u_{l} \,  u_{jp} a_{pq}(0) u_{q},
\end{split}
\]
we deduce
\begin{equation}
\label{au2new}
\begin{split}
\cc_{\delta} \, a_{ij}(x) \partial^2_{ij} \cc_{\delta} 
&= a_{kl}(0)u_{l} \, a_{ij}(x) u_{ijk} \\
& \quad \quad \quad + \|A^{1/2}(x) D^2 u A^{1/2}(0)\|^2_{\rm HS} - \textstyle \frac{|\nabla u|^2_{A(0)}}{|\nabla u|^2_{A(0)} + \delta^2 }|A^{1/2}(x) D^2 u A^{1/2}(0) \nn|^2\\
&\geq a_{kl}(0)u_{l} \, a_{ij}(x) u_{ijk} + \anew^2
\end{split}
\end{equation}
and
\begin{equation}
\label{au3new}
\begin{split}
\cc_{\delta} \, \vv_{i}(x) \partial_i \cc_{\delta} =
a_{kl}(0)u_l \, \vv_{i}(x) u_{ik}.
\end{split}
\end{equation}
Adding \eqref{au2new} and \eqref{au3new}, we obtain
\begin{equation}
\label{au4new}
\begin{split}
\cc_{\delta}\, \jacobi \cc_{\delta}  
&\geq a_{kl}(0) u_l \, L u_k + \anew^2 + f'(u) \cc_{\delta}^2.
\end{split}
\end{equation}

Differentiating the equation $-L u = f(u)$
in the direction of $A(0)\nabla u$, 
we have
\begin{equation}
\label{au1new}
A(0) \nabla (Lu) \cdot \nabla u = - f'(u) |\nabla u|_{A(0)}^2.
\end{equation}
The first term on the right-hand side of \eqref{au4new} can be written in terms of this derivative as
\[
a_{kl}(0) u_l \, L u_k = A(0) \nabla (Lu) \cdot \nabla u - a_{kl}(0) u_l \partial_k a_{ij}(x) u_{ij} - a_{kl}(0) u_l \partial_k \vv_{i}(x) u_i,
\]
hence, by \eqref{au1new} and 
the coefficient estimates,
%the bounds of the coefficients, 
we can bound this expression from below as
\begin{equation}
\label{au45}
a_{kl}(0) u_l \, L u_k \geq - f'(u) |\nabla u|_{A(0)}^2 - C \varepsilon |D^2 u | |\nabla u| - a_{kl}(0) u_l \partial_k \vv_{i}(x) u_i.
\end{equation}
Applying \eqref{au45} in \eqref{au4new},
since $\cc_{\delta}^2 - |\nabla u|_{A(0)}^2 = \delta^2$, we obtain
\begin{equation}
\label{au5new}
\begin{split}
\cc_{\delta} \, \jacobi \cc_{\delta} 
&\geq \anew^2 + \delta^2 f'(u) - C \varepsilon |D^2 u| |\nabla u| - a_{kl}(0) u_l \partial_k \vv_{i}(x) u_i .
\end{split}
\end{equation}
Multiplying \eqref{au5new} by $\eta^2$ and integrating, 
the last term $- \int_{B_1} a_{kl}(0) u_l \partial_k \vv_{i}(x) u_i  \eta^2 \d x$
can be integrated by parts as
\begin{equation}
\label{bparts}
\begin{split}
\left| - \int_{B_1}a_{kl}(0) u_l \partial_k \vv_{i}(x) u_i \, \eta^2 \d x\right| &= \left| \int_{B_1} \vv_{i}(x)  \partial_k (a_{kl}(0) u_l  u_i \, \eta^2 )\d x\right|\\
&\leq C \varepsilon \int_{B_1} |D^2 u | |\nabla u| \eta^2 \d x + C \varepsilon \int_{B_1} |\nabla u|^2 |\nabla (\eta^2)| \d x.
\end{split}
\end{equation}

Combining \eqref{au5new}, \eqref{bparts}, and \eqref{au6new}, rearranging terms, we obtain
\[
\begin{split}
&\int_{B_1} \left( \anew^2 + \delta f'(u) \right) \eta^2 \d x \\
&\quad \quad \leq 
\int_{B_1} \left(|\nabla u|_{A(0)}^2 + \delta^2\right) \left(
|\nabla \eta|^2_{A(x)} + C \varepsilon |\nabla (\eta^2)| + C \varepsilon^2 \eta^2
\right)
\d x
+ C \varepsilon \int_{B_1}  |D^2 u| |\nabla u| \eta^2 \d x,
\end{split}
\]
and letting $\delta \to 0$ yields the claim.
\end{proof}

Several remarks are in order:

\begin{remark}
\label{sz:lessreg}
In \eqref{au1new} above we took a derivative of the equation in the direction $A(0) \nabla u$ 
to get rid of the dependence on the nonlinearity.
Instead, we could have multiplied the equation 
by the test function $\xi = \div \left( A(0) \nabla u \, \eta^2 \right)$
and integrated by parts.
Notice that this avoids having to take any derivatives of $\vv$, 
since the term involving it can be bounded directly.

In the argument above, we need $u$ to have three (weak) derivatives,
otherwise we cannot compute 
$L \cc$ (or rather $L \cc_{\delta}$).
In \cite{CabreFigalliRosSerra}, the authors only need to assume $u \in C^2(B_1)$
to deduce the analogue estimate for the Laplacian,
since this already gives $u \in W^{3, p}_{\rm loc}$ for all $p < \infty$.
Indeed, differentiating the equation,
$- \Delta u_i = f'(u) u_i \in L^p_{\rm loc}$ 
and by $L^p$ estimates they deduce
$u_i \in W^{2, p}_{\rm loc}$, hence $u \in W^{3, p}_{\rm loc}$ for all $p < \infty$.
This fact allows them to carry out a similar argument to the one explained above.\footnote{In fact, they are able to deduce the estimate without computing $L \cc$ directly, but they still need to have three derivatives of the solution; see the proof of Lemma 2.1 in \cite{CabreFigalliRosSerra}.}

For an operator with variable coefficients $L$,
the regularity of the solution depends on that of the coefficients.
Assuming $a_{ij} \in C^{0, 1}(B_1)$,
$\vv_i \in L^{\infty}(B_1)$,
and $u$ bounded,
applying $L^p$ estimates to the equation $- L u = f(u) \in L^p_{\rm loc}$,
we deduce $u \in W^{2, p}_{\rm loc}$ for all $p < \infty$ (and hence in $C^{1, \alpha}$ for all  $0 < \alpha< 1$).
Now, for $u$ to be in $C^2 \cap W_{\rm loc}^{3, p}$ 
we need
more regularity on the drift $\vv$.
To see this, taking a derivative of the equation we have
$- L u_k = \partial_k a_{ij}(x) u_{ij} + \partial_k \vv_{i}(x) u_{i} + f'(u) u_k$
and the right hand side is in $L^p$ for 
$\partial_k a_{ij} \in L^{\infty}$ (i.e., $a_{ij}$ Lipschitz) and $\partial_k \vv_{i} \in L^p$.
In particular, if $\vv \in W^{1, p}(B_1)$ with $p > n$, we deduce $u \in C^{2}(B_1)\cap W^{3, p}_{\rm loc}(B_1)$.
This is somewhat surprising, 
since our estimates
do not involve any derivatives of $\vv$.
\end{remark}

\begin{remark}
\label{remark:const}
The following comments concern the form of the function $\anew$ in our a priori estimate \eqref{ineq:sz} in Theorem~\ref{thm:sz}.
Recall that $\anew$ quantifies a part of the ``mixed'', non-symmetric matrix $A^{1/2}(x) D^2 u A^{1/2}(0)$, 
which includes both the variable coefficients $A^{1/2}(x)$ and the constants $A^{1/2}(0)$.
We are led naturally to it from the choice of test function $\cc = |\nabla u|_{A(0)}$ in the integral stability inequality,
which is the function used by Cabr\'{e}, Figalli, Ros-Oton, and Serra in \cite{CabreFigalliRosSerra} after a linear transformation.

We could have also given an estimate for 
the function
\begin{equation}
\label{def:ariemdos}
\ariemdos := \Big(
\|A^{1/2}(x) D^2 u A^{1/2}(x)\|_{\rm HS}^2 - \big|A^{1/2}(x) D^2 u A^{1/2}(x) \nriem(x)\big|^2
 \Big)^{1/2}
\end{equation}
involving the symmetric matrix $A^{1/2}(x) D^2 u A^{1/2}(x)$,
where $\nriem(x) := |\nabla u|_{A(x)}^{-1} A^{1/2}(x) \nabla u$,
by choosing the test function $\cc = |\nabla u|_{A(x)}$ instead.
However,
in this case,
the proof of the analogue of Theorem~\ref{thm:sz} is more involved.
This is why we prefer our choice of $\anew$.
On the other hand, the choice $\cc = |\nabla u|_{A(x)}$
is related to an existing Riemannian version of the Sternberg and Zumbrun inequality,
which we explain next in Remark~\ref{remark:riem}.

This discussion leads to the question of whether a similar estimate exists for 
$\azero$, the natural part of the simpler symmetric matrix $A^{1/2}(0)D^2 u A^{1/2}(0)$, 
which only involves the constant coefficients $A^{1/2}(0)$.
There does not seem to be a direct way to derive such an estimate from the stability inequality,
since it is not clear which $\cc$ function could lead to it.
Nevertheless, 
when the parameter $\varepsilon$ is small,
thanks to \eqref{ultimate} below, it can be shown that $\azero$ is comparable to $\anew$.
Hence, for $\varepsilon$ small,
we can deduce the desired bound for $\azero$ from our result~\eqref{ineq:sz} for $\anew$.
We will need this fact in the proof of the ``Hessian times the gradient'' estimates in Theorem~\ref{thm:sz}, as explained below.
\end{remark}

\begin{remark}
\label{remark:riem}
Our result \eqref{ineq:sz} is related to a Riemannian analogue of the Sternberg and Zumbrun estimate
found by Farina, Sire, and Valdinoci in \cite{FarinaSireValdinoci}.
It states that stable solutions to the equation $- \lb u = f(u)$ in a Riemannian manifold $(M, g)$,
where $\lb$ is the Laplace-Beltrami operator, satisfy the inequality
\begin{equation}
\label{szriem}
\int_{M} \ariem^2 \,\eta^2 + \int_{M} \ricci(\grad u, \grad u) \, \eta^2 \leq \int_{M} |\grad u|_g^2 |\grad \eta|_g^2.
\end{equation}
Here,
$\ariem$  (given by \eqref{def:ariem}) is a Riemannian analogue of the function $\anew$ in Theorem~\ref{thm:sz},
$\ricci$ denotes the Ricci tensor,
and all the norms, gradients, and integrals are intrinsic to the metric $g$.

When expressed in coordinates, these Riemannian quantities 
fit within our Euclidean setting with variable coefficients.
For instance, the operator $\lb$ can be written in coordinates as
$L u = \div\left(A(x) \nabla u \right) + \bb(x) \cdot \nabla u$.
Here $A(x) = \left(a_{ij}(x)\right) = \left(g^{ij}(x)\right)$ is the inverse of the metric and $\bb_i(x) = \frac{1}{2} g^{ij}(x) \partial_j \log | g|$
involves the volume density $|g| = \det(g_{ij}(x))^{1/2}$.
Moreover, with our notation for matrices, the function $\ariem$ in \eqref{szriem} can be written locally in $\{\grad u \neq 0\}$ as
\begin{equation}
\label{def:ariem}
\ariem =
\Big(
\|A^{1/2}(x) \hess A^{1/2}(x)\|_{\rm HS}^2 - \big|A^{1/2}(x) \hess A^{1/2}(x) \nriem(x)\big|^2
 \Big)^{1/2},
\end{equation}
where
$\nriem(x) = |\nabla u|_{A^(x)}^{-1} A^{1/2}(x) \nabla u$ 
has appeared in the definition \eqref{def:ariemdos} of $\ariemdos$ in Remark~\ref{remark:const}
and
$\hess = \left((\hess )_{ij}\right)$ 
is the Riemannian Hessian matrix given by $(\hess)_{ij} = u_{ij} - \Gamma_{ij}^{k} u_k$, where $\Gamma_{ij}^{k}$ are the Christoffel symbols of the metric.

By this identification of $\lb$ with $L$, applying the Riemannian result in \cite{FarinaSireValdinoci},
collecting all lower order terms, 
and estimating the derivatives of the metric,
we are led to an a priori bound 
for the function $\ariemdos$ in \eqref{def:ariemdos}
which involves
errors of the same type as in~\eqref{ineq:sz}.
Due to the presence of the Ricci tensor in \eqref{szriem},
this estimate derived from the Riemannian inequality \eqref{szriem} 
depends on the norm $\|A\|_{C^2(\overline{B}_1)}$, i.e., 
it requires two derivatives of the metric.
Nevertheless, integrating the unwanted coefficient derivatives by parts as we did in our proof of \eqref{ineq:sz},
we could deduce an estimate depending only on $\|A\|_{C^1(\overline{B}_1)}$.

The authors in \cite{FarinaSireValdinoci} obtain \eqref{szriem}
by choosing the test function $\cc = |\grad u|_{g}$ in their stability inequality.
In our coordinates, this function reads as $\cc(x) = |\nabla u|_{A(x)}$.
As explained in Remark~\ref{remark:const},
this choice of $\cc$ and our integral stability inequality~\eqref{stable:jacobi}
lead to a similar estimate for $\ariemdos$ by using the ideas from the proof of \eqref{ineq:sz} above.

We emphasize that both approaches (the Riemannian one and ours)
give an estimate for $\ariemdos$ which contains an error term involving the product $|D^2 u | |\nabla u|$.
This error arises from the interaction between the second and first order terms in the Riemannian Hessian $\hess$
when squaring $\ariem$, and thus squaring $\hess = D^2 u - \Gamma \nabla u = D^2 u + O(\varepsilon |\nabla u|)$.
\end{remark}

Next, we prove the ``Hessian times the gradient'' estimates
\eqref{ineq:hessgrad1} and \eqref{ineq:hessgrad2} in Theorem~\ref{thm:sz}.
For this,
we will need to consider the auxiliary function
\begin{equation}
\label{def:alul}
\azero := 
\left\{\begin{array}{ll}
\Big( \|A^{1/2}(0)D^2 u A^{1/2}(0)\|_{\rm HS}^2 
- 
|A^{1/2}(0) D^2 u A^{1/2}(0) \nn(x)|^2
\Big)^{1/2} & \text{if } \nabla u \neq 0\\
        0 & \text{if } \nabla u = 0,
\end{array}\right.
\end{equation}
where $\nn(x) = |\nabla u|_{A^{-1}(0)}^{-1} A^{1/2}(0)\nabla u$ is again the 
vector field in the definition of $\anew$ in \eqref{def:anew}.
Notice that \eqref{def:alul} is precisely the definition of $\anew$ in \eqref{def:anew} with the matrix $A^{1/2}(x)$ replaced by $A^{1/2}(0)$;
see Remark~\ref{remark:const}.
The greatest advantage of the function $\azero$ over $\anew$ is the symmetry of 
the matrix $A^{1/2}(0) D^2 u A^{1/2}(0)$ in the definition above.
This will allow us to bound the Hessian of the solution by $\azero$, with the exception of the $\nn \otimes \nn$ component, which can be treated separately thanks to the nonnegativity assumption on the nonlinearity.

We will also need the a priori estimate \eqref{ineq:sz} proved above,
which gives a bound for the $L^2$ norm of the function $\anew$.
In the proof below, for $\|D A\|_{C^0(\overline{B}_1)} \leq \varepsilon$, 
we will see that 
\begin{equation}
\label{ultimate}
|\anew^2 - \azero^2| \leq C \varepsilon |x| \azero^2 \quad \text{ in } B_1,
\end{equation}
where $C$ is a universal constant.
In particular, for $\varepsilon$ small, the functions are comparable
and \eqref{ineq:sz} allows us to bound the $L^2$ norm of $\azero$ as well.

\begin{proof}[Proof of \eqref{ineq:hessgrad1} and \eqref{ineq:hessgrad2} in Theorem~\ref{thm:sz}]
Throughout the proof, $C$ denotes a generic universal constant.
The proof is divided into four steps.

\vspace{3mm}\noindent
\textbf{Step 1:}\textit{ We prove that
\begin{equation}
\label{key:hessian}
|D^2 u| \leq  - C\tr\big(A(0) D^2 u \big) +  C\azero + C \varepsilon |x| |D^2 u | + C \varepsilon |\nabla u|  \quad \text{ a.e. in } B_1,
\end{equation}
where $C > 0$ is universal.
} 

First we bound the full Hessian of $u$ almost everywhere
by the function $\azero$
and the $\nn \otimes \nn$ component of the matrix $A^{1/2}(0) D^2 u A^{1/2}(0)$.
If $\nabla u(x) \neq 0$,
then, extending $\nn(x)$ to an orthonormal basis of $\R^n$, it is easy to see\footnote{This follows immediately from the fact that, for any symmetric matrix $M \in \R^{n\times n}$, we have $\|M\|_{\rm HS}^2 = \sum_{i, j =1}^{n-1} M_{ij}^2 + 2 \sum_{i= 1}^{n-1} M_{in}^2 + M_{nn}^2$ and $\|M\|_{\rm HS}^2 - |M e_n|^2 = \sum_{i, j = 1}^{n-1} M_{ij}^2 + \sum_{i= 1}^{n-1} M_{in}^2$.} that
\begin{equation}
\label{meaning:alul}
\|A^{1/2}(0) D^2 u A^{1/2}(0)\|^2_{\rm HS} \leq 2 \azero^2 + \big|(A^{1/2}(0) D^2 u A^{1/2}(0))\nn(x) \cdot \nn(x)\big|^2.
\end{equation}
Moreover, by Stampacchia's result, $|D^2 u| = 0$ a.e. in $\nabla u = 0$  
(see \cite{LiebLoss}*{Theorem~6.19}),
and the inequality \eqref{meaning:alul} holds almost everywhere in $B_1$.
By uniform ellipticity we also have
$|D^2 u| \leq C |A^{1/2}(0) D^2 u A^{1/2}(0)| \leq C \|A^{1/2}(0) D^2 u A^{1/2}(0)\|_{\rm HS}$
and hence
\begin{equation}
\label{ran0}
|D^2 u| \leq C \azero + C \big|(A^{1/2}(0)D^2 u A^{1/2}(0)) \nn(x)  \cdot \nn(x) \big| \quad   \text{ a.e. in } B_1.
\end{equation}

Next we use
that the nonlinearity is nonnegative to bound the $\nn \otimes \nn$ component $(A^{1/2}(0)D^2 u A^{1/2}(0)) \nn(x) \cdot \nn(x)$
in \eqref{ran0}
in terms of the equation, the function $\azero$, and lower order terms. 

Since $ 0 \geq - f(u) = L u = \tr(A(x) D^2 u) + \vv(x) \cdot \nabla u$,
we have 
\begin{equation}
\label{signtrick}
\begin{split}
\left| \tr(A(x) D^2 u) \right| &= |L u - \vv(x) \cdot \nabla u| \\
& \leq -L u + |\vv(x) \cdot \nabla u| = -\tr(A(x) D^2 u) - \vv(x) \cdot \nabla u + |\vv(x) \cdot \nabla u|\\
& \leq -\tr(A(x) D^2 u) + 2 \varepsilon |\nabla u|.
\end{split}
\end{equation}
By the mean value theorem we have $|A(x) - A(0)| \leq \varepsilon |x|$, and hence by \eqref{signtrick}
\begin{equation}
\label{signtrick2}
\begin{split}
\left| \tr(A(0) D^2 u) \right| &\leq \left|\tr(A(x) D^2 u)\right| + C \varepsilon |x| |D^2 u| \\
&\leq  -\tr(A(x) D^2 u) + C \varepsilon |x| |D^2 u|  + C \varepsilon |\nabla u|\\
&\leq  -\tr(A(0) D^2 u) + C \varepsilon |x| |D^2 u|  + C \varepsilon |\nabla u|.
\end{split}
\end{equation}

By the same argument to deduce \eqref{meaning:alul} above,
it is easy to see that\footnote{Follows from the fact that, for any symmetric matrix $M \in \R^{n\times n}$, we have $|M| \leq n \|M\|_{\rm HS}$
and $\|M - M_{nn} e_n \otimes e_n\|^2_{\rm HS} 
%= \sum_{i, j = 1}^{n-1} M_{ij}^2 + 2 \sum_{i = 1}^{n-1} |M_{in}|^2 
\leq 2 \left(\|M \|^2_{\rm HS} - |M e_n|^2\right).$
}
\begin{equation}
\label{aux:reason}
\big|A^{1/2}(0) D^2 u A^{1/2}(0) - \big[ (A^{1/2}(0) D^2 u A^{1/2}(0)) \nn \cdot \nn \big] \nn \otimes \nn \big|
\leq C \azero \quad \text{ a.e. in } B_1.
\end{equation}
Writing the $\nn \otimes \nn$ component of $A^{1/2}(0) D^2 u A^{1/2}(0)$ as
\[
\begin{split}
&(A^{1/2}(0)D^2 u A^{1/2}(0)) \nn \cdot \nn \\
& \quad \quad =   \tr \, \Big(A^{1/2}(0)D^2 u A^{1/2}(0)\Big) \\
& \quad \quad\quad \quad - \tr \, \Big( A^{1/2}(0) D^2 u A^{1/2}(0) - \big[ (A^{1/2}(0) D^2 u A^{1/2}(0)) \nn \cdot \nn \big] \nn \otimes \nn \Big),
\end{split}
\]
from \eqref{signtrick2} and \eqref{aux:reason}, 
it follows that
\begin{equation}
\label{whon}
\begin{split}
\big|(A^{1/2}(0)D^2 u A^{1/2}(0)) \nn \cdot \nn \big|  &\leq \big|\tr \, \big(A^{1/2}(0)D^2 u A^{1/2}(0)\big)\big| + C\azero \\
&\leq -\tr(A(0) D^2 u) + C \azero + C \varepsilon |x| |D^2 u| + C \varepsilon |\nabla u|
\end{split}
\quad \text{ a.e. in }B_1.
\end{equation}
Combining \eqref{whon} and \eqref{ran0} yields the claimed inequality \eqref{key:hessian}.

\vspace{3mm}\noindent
\textbf{Step 2:}
\textit{We prove that
there is a universal $\varepsilon_0 > 0$ such that, if $\varepsilon \leq \varepsilon_0$, then
\[
\int_{B_1}  |D^2 u| |\nabla u|  \eta^2 \d x \leq 
C\int_{B_1} \anew |\nabla u| \, \eta^2 \d x
+ C \int_{B_1}  |\nabla u|^2 \big( |\nabla (\eta^2)| + \varepsilon \eta^2 \big) \d x 
\]
for all $\eta \in C^{\infty}_c(B_1)$,
where $C$ is universal.
} 

By uniform ellipticity, it suffices to estimate the integral $\int_{B_1} |D^2 u| |\nabla u|_{A(0)} \eta^2 \d x$.
Multiplying \eqref{key:hessian} in Step $1$ by $|\nabla u|_{A(0)} \eta^2$ and integrating in $B_1$, 
by uniform ellipticity
we have
\begin{equation}
\label{ran:1new}
\begin{split}
\int_{B_1}  |D^2 u| |\nabla u|_{A(0)} \eta^2 \d x &\leq - C \int_{B_1} |\nabla u|_{A(0)} \tr\big( A(0) D^2 u\big) \, \eta^2 \d x 
+C \int_{B_1} \azero |\nabla u|_{A(0)}\, \eta^2 \d x\\
&\quad \quad \quad  + C \varepsilon \int_{B_1} |x| |D^2 u| |\nabla u|_{A(0)} \eta^2 \d x+C \varepsilon \int_{B_1}  |\nabla u|^2  \, \eta^2\d x.
\end{split}
\end{equation}
The only delicate term in the right-hand side of \eqref{ran:1new} is 
the first one,
which can be treated as follows.

We write the product $-|\nabla u|_{A(0)} \tr(A(0) D^2 u )$ 
in $\{\nabla u \neq 0\}$ as
\begin{equation}
\label{more1}
\begin{split}
-2|\nabla u|_{A(0)} \tr\big(A(0) D^2 u\big) &= - |\nabla u|_{A(0)} \tr(A(0) D^2 u) - \div \left( |\nabla u|_{A(0)} A(0) \nabla u \right)\\
&\quad \quad  \quad  \quad + \nabla |\nabla u|_{A(0)} \cdot A(0) \nabla u.
\end{split}
\end{equation}
Since
\[
\begin{split}
\nabla |\nabla u|_{A(0)} \cdot A(0) \nabla u &= |\nabla u|_{A(0)}^{-1} \,   D^2 u A(0) \nabla u \cdot A(0) \nabla u \\
&=  |\nabla u|_{A(0) } (A^{1/2}(0) D^2 u A^{1/2}(0)) \nn \cdot \nn,
\end{split}
\]
by \eqref{more1} and using that $\nn$ is unitary, it follows that
\begin{equation}\label{more3}
\begin{split}
&-2|\nabla u|_{A(0)} \tr\big(A(0)D^2 u\big)\\
& = -|\nabla u|_{A(0)} \tr \Big(A^{1/2}(0)D^2 u A^{1/2}(0) - \big[\big(A^{1/2}(0)D^2 u A^{1/2}(0)\big)\nn \cdot \nn\big] \nn \otimes \nn \Big) \\
& \quad \quad \quad  - \div \big(|\nabla u|_{A(0)} A(0)\nabla u \big)
\end{split}
\end{equation}
a.e. in $B_1$.
By the bound \eqref{aux:reason} in the proof of Step $1$ above, it follows that
\[
\left| \tr \Big(A^{1/2}(0)D^2 u A^{1/2}(0) - \big[\big(A^{1/2}(0)D^2 u A^{1/2}(0)\big)\nn \cdot \nn\big] \nn \otimes \nn \Big)\right|
\leq C \azero \quad \text{ a.e. in } B_1,
\]
and hence from \eqref{more3} we deduce
\begin{equation}\label{more4}
\begin{split}
-2|\nabla u|_{A(0)} \tr\big(A(0) D^2 u\big) &\leq  - \div \big(|\nabla u|_{A(0)} A(0)\nabla u \big) +C  \azero |\nabla u|_{A(0)} \\
\end{split}
\quad \text{ a.e. in } B_1.
\end{equation}

Substituting \eqref{more4} in \eqref{ran:1new} leads to
\[
\begin{split}
\int_{B_1}  |D^2 u| |\nabla u|_{A(0)} \eta^2 \d x &\leq 
- C\int_{B_1}\div \big(|\nabla u|_{A(0)} A(0)\nabla u \big) \,\eta^2 \d x 
+C \int_{B_1} \azero |\nabla u|_{A(0)}\, \eta^2 \d x\\
&\quad \quad \quad + C \varepsilon \int_{B_1} |x| |D^2 u| |\nabla u|_{A(0)}  \, \eta^2\d x +C \varepsilon \int_{B_1}  |\nabla u|^2  \, \eta^2\d x,
\end{split}
\]
and integrating by parts the divergence term, we obtain the inequality
\begin{equation}
\label{somelul}
\begin{split}
\int_{B_1}  |D^2 u| |\nabla u|_{A(0)} \eta^2 \d x &\leq  C\int_{B_1}|\nabla u|^2 \, \left(|\nabla \eta^2| + \varepsilon \eta^2\right)\d x 
+C \int_{B_1} \azero |\nabla u|_{A(0)}\, \eta^2 \d x\\
&\quad \quad \quad  + C \varepsilon \int_{B_1} |x| |D^2 u| |\nabla u|_{A(0)}  \, \eta^2\d x.
\end{split}
\end{equation}

Since $|x| \leq 1$ in $B_1$, choosing $\varepsilon_0 > 0$ universal small such that $C \varepsilon_0 = 1/2$, we can absorb the ``Hessian times the gradient'' error in \eqref{somelul} into the left-hand side to obtain
\begin{equation}
\label{alle}
\begin{split}
\int_{B_1}  |D^2 u| |\nabla u|_{A(0)} \eta^2 \d x &\leq  C\int_{B_1}|\nabla u|^2 \, \left(|\nabla \eta^2| + \varepsilon \eta^2\right)\d x 
+C \int_{B_1} \azero |\nabla u|_{A(0)}\, \eta^2 \d x.
\end{split}
\end{equation}

To conclude the argument, let us show that $\anew$ and $\azero$ are comparable for $\varepsilon$ small.
Letting $E(x) = A(x) - A(0)$ and $M(x) = D^2 u(x) A^{1/2}(0)$, it is easy to check that
\[
\anew^2 = \azero^2 + \tr\left( M(x)^{T}E(x) M(x) \right) - (M(x)^{T} E(x) M(x)) \nn \cdot \nn \quad \text{ in } \{\nabla u \neq 0\},
\]
and for $x \in \{\nabla u \neq 0\}$, extending $\nn = \nn(x)$ to an orthonormal basis $\ee_1$, $\ldots$, $\ee_n = \nn$ of $\R^n$,
we can rewrite this identity as
\begin{equation}
\label{here}
\anew^2 = \azero^2 +  \sum_{i = 1}^{n-1} E(x) M(x) \ee_i \cdot M(x) \ee_i \quad \text{ in } \{\nabla u \neq 0\}.
\end{equation}
By the mean value theorem we can bound the error by $|E(x)| \leq \varepsilon |x|$, 
and hence, by uniform ellipticity,
\begin{equation}
\label{alle2}
\left| \sum_{i = 1}^{n-1} E(x) M(x) \ee_i \cdot M(x) \ee_i. \right| \leq \varepsilon |x| \sum_{i = 1}^{n-1} |M(x) \ee_i|^2 \leq C \varepsilon |x| \sum_{i = 1}^{n-1} |M(x) \ee_i|_{A(0)}^2.
\end{equation}
Since $|M(x) \ee_i|_{A(0)}^2 = |A^{1/2}(0) D^2 u(x) A^{1/2}(0) \ee_i|$,
by \eqref{aux:reason} above,
the sum in right-hand side of  \eqref{alle2} can be further bounded by
\begin{equation}
\label{alle3}
\sum_{i = 1}^{n-1} |M(x) \ee_i|_{A(0)}^2 \leq C \azero \quad \text{ a.e. in } B_1.
\end{equation}
Combining \eqref{alle2} and \eqref{alle3},
from \eqref{here} we conclude that
\[
\left(1 - C \varepsilon |x| \right) \azero^2 \leq \anew^2 \leq \left(1 + C \varepsilon |x|\right) \azero^2 \quad \text{ in } B_1,
\]
which was the inequality \eqref{ultimate} mentioned before the proof.
Choosing $\varepsilon_0$ smaller if necessary, we may assume that $\azero \leq 2 \anew$,
which applied in \eqref{somelul} yields the claim.

\vspace{3mm}\noindent
\textbf{Step 3:}
\textit{We prove that, 
if $\varepsilon \leq \varepsilon_0$, with $\varepsilon_0 > 0$ as in Step 1, then 
\[
\begin{split}
\displaystyle \int_{B_1}  \anew^{2}  \eta^2 \d x & \leq 
C \int_{B_1}  |\nabla u|^2 \big( |\nabla \eta|^2 + \varepsilon^2 \eta^2\big) \d x
\end{split}
\]
for all $\eta \in C^{\infty}_c(B_1)$,
where $C$ is a universal constant.
}

Combining 
\eqref{ineq:sz} in Theorem~\ref{thm:sz} and Step $2$, we have
\begin{equation}
\label{au8new}
\begin{split}
\int_{B_1}  \anew^2 \eta^2 \d x & \leq 
C \varepsilon \int_{B_1} \anew |\nabla u| \eta^2 \d x
+ \int_{B_1} |\nabla u|_{A(0)}^2 |\nabla \eta|_{A(x)}^2 \d x\\
&\quad \quad \quad + C \varepsilon \int_{B_1}  |\nabla u|^2 \big( |\nabla (\eta^2)| + \varepsilon \eta^2 \big) \d x.
\end{split}
\end{equation}
By Young's inequality,
the first term on the right-hand side of \eqref{au8new} can be bounded by
\[
C\varepsilon \int_{B_1} \anew |\nabla u| \eta^2 \d x \leq \dfrac{1}{2}\int_{B_1} \anew^{2} \eta^2 \d x + C \varepsilon^2 \int_{B_1}  |\nabla u|^{2} \eta^2 \d x,
\]
and the $\anew^2 \eta^2$ integral can be absorbed into the left-hand side.
By uniform ellipticity and 
the bound $\varepsilon |\nabla(\eta^2)| \leq |\nabla \eta|^2 + \varepsilon^2 \eta^2$, we deduce the claim.

\vspace{3mm}\noindent
\textbf{Step 4:}
\textit{Conclusion.}

Combining Steps $2$ and $3$, for $\varepsilon \leq \varepsilon_0$ as above and by Cauchy-Schwarz, we obtain
\begin{equation}
\label{ineq:hessgrad}
\begin{split}
&\int_{B_1} |D^2 u| |\nabla u|  \eta^2 \d x \\
&\quad \quad  \leq 
C \left(\int_{B_1} \anew^2 \eta^2\d x \right)^{1/2}\left( \int_{B_1} |\nabla u|^2 \eta^2 \d x\right)^{1/2}
+ C \int_{B_1} |\nabla u|^2 \big( |\nabla (\eta^2)| + \varepsilon \eta^2 \big) \d x \\
&\quad \quad  \leq 
C \left(\int_{B_1}|\nabla u|^2  \big( |\nabla \eta|^2 + \varepsilon^2 \eta^2 \big) \d x \right)^{1/2}\left( \int_{B_1} |\nabla u|^2 \eta^2 \d x\right)^{1/2}\\
& \quad \quad \quad \quad \quad  + C \int_{B_1} |\nabla u|^2 \big( |\nabla (\eta^2)| + \varepsilon \eta^2 \big) \d x.
\end{split}
\end{equation}

The inequalities \eqref{ineq:hessgrad1} and \eqref{ineq:hessgrad2} in Theorem~\ref{thm:sz} follow easily from \eqref{ineq:hessgrad} by choosing appropriate cut-off functions
and using that $\varepsilon$ is bounded by a universal constant $\varepsilon_0$.
Choosing
$\eta \in C^{\infty}_c(B_{1})$ such that $\eta= 1$ in $B_{3/4}$ and $0 \leq \eta \leq 1$ in $B_{1}$
leads to the estimate in balls \eqref{ineq:hessgrad1}.
The second estimate in annuli \eqref{ineq:hessgrad2}
follows by choosing
$\eta \in C^{\infty}_c(B_{1} \setminus \overline{B}_{1/8})$ with $\eta= 1$ in $B_{1/2} \setminus B_{1/4}$ and $0 \leq \eta \leq 1$ in $B_1$.
\end{proof}

\begin{remark}
\label{remark:l1hess}
By the proof above, 
we can also deduce an interior a priori estimate for the $L^1$ norm of the Hessian.
For this, assuming $\varepsilon$ to be small,
recalling that $\azero \leq C \anew$ by \eqref{ultimate},
and absorbing the Hessian term in Step 1, we obtain
\[
|D^2 u| \leq -C \div\left(A(0) \nabla u \right) + C \anew + C |\nabla u| \quad \text{a.e. in } B_1.
\]
Multiplying this inequality by a cut-off function and integrating by parts,
using the bound for $\anew$ in Step $3$ and applying Cauchy-Schwarz,
we deduce an estimate 
for the $L^1$ norm of the Hessian in terms of the $L^2$ norm of the gradient in a larger ball.
\end{remark}

We conclude this section by proving Proposition~\ref{prop:l2l1}.
To show that the $L^2$ norm of the gradient is controlled by the $L^1$ norm of the function in a larger ball,
we use the interpolation inequalities of Cabr\'{e}~\cite{CabreQuant}
combined with the Hessian estimates from Theorem~\ref{thm:sz}.
The errors in larger balls can then be absorbed thanks to a celebrated lemma of Simon~\cite{Simon}.
We recall the interpolation inequalities of Cabr\'{e} in Appendix~\ref{app:interpolation}
and Simon's lemma in Appendix~\ref{app:simon}.

\begin{proof}[Proof of Proposition~\ref{prop:l2l1}]
We cover $B_{1/2}$ 
(except for a set of measure zero) 
with a family of disjoint open cubes $Q_j$
of the same side-length and small enough so that $Q_j\subset B_{3/4}$.
The side-length and the number of cubes depend only on $n$.
Combining the interpolation inequalities of Propositions \ref{interpol} and \ref{nash}, rescaled from the unit cube to $Q_j$, 
with 
$\tilde\delta= \delta^{3/2}$ for a given $\delta\in(0,1)$,
we obtain
\begin{equation*}
\int_{Q_j}|\nabla u|^{2}dx \leq C\delta \int_{Q_j}\lvert D^2u\rvert |\nabla u| \,dx+ C\delta \int_{Q_j}|\nabla u|^2dx+C\delta^{-2-\frac{3n}{2}}\left( \int_{Q_j}|u|\,dx\right)^2.
\end{equation*} 
Since $Q_j\subset B_{3/4}$,
applying \eqref{ineq:hessgrad1} from Theorem~\ref{thm:sz}, for $\varepsilon \leq \varepsilon_0$  we have
\begin{equation*}
\int_{Q_j}|\nabla u|^{2}dx \leq C \delta \int_{B_1}|\nabla u|^2dx+C\delta^{-2-\frac{3n}{2}}\left( \int_{B_1}|u|\,dx\right)^2.
\end{equation*} 
Adding up these inequalities, we obtain
\begin{equation}\label{rofnew}
\|\nabla u\|_{L^{2}(B_{1/2})}^2 \le C \delta  \|\nabla u\|_{L^{2}(B_{1})}^2 + C\delta^{-2-\frac{3n}{2}} \|u\|_{L^{1}(B_{1})}^2 \quad \text{ for } \delta \in (0, 1) \text{ and } \varepsilon \leq \varepsilon_0.
\end{equation}

As explained in Section~\ref{section:prelim},
for $B_{\rho}(y) \subset B_1$, 
the function
$u^{y,\rho} :=u(y+\rho \, \cdot)$ 
is a stable solution to a semilinear equation with coefficients $A^{y,\rho} = A(y + \rho \, \cdot)$ and $\vv^{y,\rho} = \rho\, \vv(y + \rho \, \cdot)$.
In particular, since $\rho \leq 1$, for $\varepsilon \leq \varepsilon_0$ we have that
\[
\| D A^{y,\rho} \|_{C^0(\overline{B}_1)} + \| \vv^{y,\rho} \|_{C^0(\overline{B}_1)} \leq \rho \varepsilon \leq \varepsilon_0,
\]
and we can apply \eqref{rofnew} to $u^{y,\rho}$, which yields
\[
\begin{split}
\rho^{n+2}\int_{B_{\rho/2}(y)}|\nabla u|^2\,dx&\leq 
C\delta \rho^{n+2}\int_{B_{\rho}(y)}|\nabla u|^2\,dx
+ C\delta^{-2-\frac{3n}{2}}\left(\int_{B_{\rho}(y)}|u|\,dx\right)^2\\
& \leq C \delta \rho^{n+2}\int_{B_{\rho}(y)}|\nabla u|^2\,dx
+C\delta^{-2-\frac{3n}{2}}\left(\int_{B_{1}}|u|\,dx \right)^2.
\end{split}
\]
By Lemma \ref{lemma:simon} with
$\sigma(B):=\|\nabla u\|_{L^2(B)}^2$, 
the claim follows.
\end{proof}

\section{The $W^{1, 2 + \gamma}$ estimate}
\label{section:higher:int}

Here we prove the higher integrability estimate \eqref{eq:high:int} in Theorem~\ref{thm:main}.
The strategy of proof is the same as for the Laplacian in \cite{CabreFigalliRosSerra}.
First we bound the $L^{2+ \gamma}$ norm in terms of the $L^2$ norm of the gradient when the coefficients are small.
This will follow from a uniform estimate of the Dirichlet norm on level sets,
which relies on the Hessian estimates in Theorem~\ref{thm:sz}.

\begin{lemma}\label{lemma:higher:int}
Let $u \in C^{\infty}\big(\overline{B}_{1}\big)$ be a stable solution of 
$-L u = f(u)$ in $B_1$,
for some nonnegative function $f\in C^1(\R)$.
Assume that 
\[
\| D A \|_{C^0(\overline{B}_1)} + \| \vv \|_{C^0(\overline{B}_1)} \leq \varepsilon
\]
for some $\varepsilon > 0$.

Then there exists a universal constant $\varepsilon_0 > 0$ 
with the following property:
if $\varepsilon \leq \varepsilon_0$, then 
\[
\|\nabla u\|_{L^{2+\gamma}(B_{1/2})} \leq C \|\nabla u\|_{L^2(B_1)},
\]
where $\gamma > 0$ 
depends only on $n$,
and
$C$ is a universal constant.
\end{lemma}
\begin{proof}
The proof is divided in two steps.

\vspace{3mm}\noindent
\textbf{Step 1:}
{\it We prove that, if $\varepsilon \leq \varepsilon_0$, then for a.e. $t \in \R$ we have}
\[
\int_{\{u = t\} \cap B_{1/2}} |\nabla u|^2 \d \mathcal{H}^{n-1} \leq  
C  \|\nabla u\|_{L^2(B_{1})}^2,
\]
{\it where $\varepsilon_0 > 0$ and $C$ are universal.}

Since $\left|\div\big( |\nabla u| \nabla u\big) \right| \leq C |D^2 u|  |\nabla u|$,
by \eqref{ineq:hessgrad1} in Theorem~\ref{thm:sz}, for $\varepsilon \leq \varepsilon_0$
we have
\begin{equation}
\label{new:stepnew}
\begin{split}
\left\|\div\big( |\nabla u| \nabla u\big)\right\|_{L^1(B_{3/4})} 
\leq C  \|\nabla u\|_{L^2(B_{1})}^2.
\end{split}
\end{equation}
Consider a cut-off function $\eta\in C^{\infty}_c(B_{3/4})$ with $\eta = 1$ in $B_{1/2}$ and $0 \leq \eta \leq 1$.
By the divergence theorem, for a.e. $t \in \R$ we have
\[\begin{split}
\int_{\{u = t\} \cap B_{1/2}} |\nabla u|^2 \d \mathcal{H}^{n-1} &\leq 
\int_{\{u = t\} \cap B_{1} \cap \{\nabla u \neq 0\}}  |\nabla u|^2 \eta^2 \d \mathcal{H}^{n-1}\\
&= - \int_{\{u > t\} \cap B_{1} \cap \{\nabla u \neq 0\}} \div\big(|\nabla u| \nabla u \, \eta^2 \big) \d x \\
&\leq \int_{B_1} |\nabla u|^2 |\nabla (\eta^2)| \d x + \int_{B_1} \big|\div\big(|\nabla u| \nabla u\big)\big| \eta^2 \d x
\end{split}
\]
and applying \eqref{new:stepnew} we obtain the claim.

\vspace{3mm}\noindent
\textbf{Step 2:}
{\it Conclusion.}

Let
\[
v := \dfrac{u - (u)_{B_1}}{\|\nabla u\|_{L^2(B_1)}},
\]
where
$(u)_{B_1} := \displaystyle\frac{1}{|B_1|}\int_{B_1} u \d x$.
In particular $\|\nabla v\|_{L^2(B_1)} = 1$
and by the Sobolev-Poincaré inequality, for some dimensional $p > 2,$ we have
\begin{equation}\label{sobolev:poincarenew}
\left( \int_{B_1} |v|^p \d x \right)^{\frac{1}{p}} \leq C \left( \int_{B_1} |\nabla v|^2 \d x \right)^{\frac{1}{2}} = C.
\end{equation}
By the coarea formula and \eqref{sobolev:poincarenew}, we have
\begin{equation}\label{coarea:applicationnew}
\begin{split}
&\int_{\R}\d t \int_{\{v = t\}\cap \{|\nabla v| \neq 0\}} |t|^{p}|\nabla v|^{-1} \d \mathcal{H}^{n-1} = \int_{B_1 \cap \{|\nabla v| \neq 0\}} |v|^{p}  \d x \leq C.
\end{split}
\end{equation}
Since $p > 2,$ we can choose dimensional constants $q > 1$ and $\theta \in (0,1/3)$ such that $p/q = (1-\theta)/\theta.$
We define
\[h(t) := \max\{1, |t|\}.\]
Using the coarea formula and the H\"{o}lder inequality (note that $p\theta - q (1-\theta) = 0$), we obtain
\[
\begin{split}
\int_{B_{1/2}} |\nabla v|^{3 - 3\theta} \d x &= \int_{\R}\d t \int_{\{v = t\}\cap B_{1/2} \cap \{|\nabla v| \neq 0\}} h(t)^{p\theta - q(1-\theta)}|\nabla v|^{-\theta + 2(1-\theta)} \d x \\
&\leq \left(\int_{\R}\d t \int_{\{v = t\}\cap B_{1} \cap \{|\nabla v| \neq 0\}} h(t)^{p}|\nabla v|^{-1} \d \mathcal{H}^{n-1} \right )^{\theta} \cdot \\
& \quad \quad \quad \quad \quad \quad \quad  \cdot \left(\int_{\R}\d t \int_{\{v = t\}\cap B_{1/2}} h(t)^{-q}|\nabla v|^{2} \d \mathcal{H}^{n-1} \right )^{1-\theta}.
\end{split}
\]
Thanks to \eqref{coarea:applicationnew} and the definition of $h(t),$ we have
\[
\begin{split}
&\int_{\R}\d t \int_{\{v = t\}\cap B_{1} \cap \{|\nabla v| \neq 0\}} h(t)^{p}|\nabla v|^{-1} \d \mathcal{H}^{n-1}\\
&\leq \int_{-1}^{1} \d t \int_{\{v = t\}\cap B_{1} \cap \{|\nabla v| \neq 0\}} |\nabla v|^{-1} \d \mathcal{H}^{n-1} + C \leq |B_1| + C \leq C.
\end{split}
\]
Since $q > 1,$
it follows that $\int_{\R}h(t)^{-q} \d t$ is finite
and by Step 1, for $\varepsilon \leq \varepsilon_0$, we have
\[
\begin{split}
&\int_{\R}\d t \, h(t)^{-q} \int_{\{v = t\}\cap B_{1/2}} |\nabla v|^2 \d \mathcal{H}^{n-1} \leq C.
\end{split}
\]
Finally, we obtain
\[
\begin{split}
&\int_{B_{1/2}} |\nabla v|^{3-3\theta} \d x \leq C
\end{split}
\]
which gives the claim, since $\nabla v \equiv \nabla u/ \|\nabla u\|_{L^2(B_1)}$.
\end{proof}

To deduce the $L^{2+\gamma}$ estimate \eqref{eq:high:int} in Theorem~\ref{thm:main}, we will combine Proposition~\ref{prop:l2l1} with Lemma~\ref{lemma:higher:int}, 
and apply a scaling and covering argument.

\begin{proof}[Proof of \eqref{eq:high:int} in Theorem~\ref{thm:main}]
Combining Proposition~\ref{prop:l2l1}
and Lemma~\ref{lemma:higher:int}, applied to $u(\cdot/2)$,
we deduce that there is a universal $\varepsilon_0 > 0$ such that, if $\varepsilon \leq \varepsilon_0$, then 
\begin{equation}
\label{asd}
\|\nabla u\|_{L^{2 + \gamma}(B_{1/4})} \leq C \|u\|_{L^1(B_1)},
\end{equation}
where $\gamma > 0$ depends only on $n$, and $C$ is universal.

Now \eqref{eq:high:int} will follow easily from \eqref{asd} by a scaling and covering argument.
Let $\rho \in (0, 1)$ to be chosen later.
We cover the ball $B_{1/2}$ by a finite number of balls $B_{\rho/4}(y_j)$
with
$B_{\rho}(y_j) \subset B_1$.
The number balls depends only on $n$ and $\rho$.
As explained in Section~\ref{section:prelim},
the functions $u^{y_j, \rho} = u(y_j + \rho \, \cdot)$
are stable solutions to a semilinear equation with coefficients 
$A^{y_j, \rho} = A(y_j + \rho \, \cdot)$ and  $\vv^{y_j, \rho} = \rho \, \vv(y_j + \rho \, \cdot)$.
Choosing $\rho$ small enough so that 
$\rho \big(\|D A\|_{C^0(\overline{B}_1)} +\|\vv\|_{C^0(\overline{B}_1)}\big) \leq \varepsilon_0$,
it follows that
\[
\|D A^{y_j, \rho}\|_{C^0(\overline{B}_1)} +\|\vv^{y_j, \rho}\|_{C^0(\overline{B}_1)} \leq  \varepsilon_0,
\]
and we can apply \eqref{asd}
to each $u^{y_j, \rho}$,
which yields
\[
\|\nabla u\|_{L^{2+\gamma}(B_{1/2})} 
\leq \sum_{j} \|\nabla u\|_{L^{2+\gamma}(B_{\rho/4}(y_{j}))} 
\leq C \sum_{j}\|u\|_{L^1(B_{\rho}(y_j))}
\leq  C \|u\|_{L^1(B_1)},
\]
for some $C$ depending only on 
$n$, $\elliptic$, $\bounded$, and $\rho$.
Since $\rho$ depends only on 
$\|D A\|_{C^0(\overline{B}_1)}$, $\|\vv\|_{C^0(\overline{B}_1)}$,
and $\varepsilon_0$, which is universal,
this concludes the proof.
\end{proof}

\section{The weighted $L^2$ estimate for radial derivatives}
\label{section:radial}

Our goal in this section is to prove Proposition~\ref{prop:rad}, 
where we bound the weighted $L^2$ norm of the radial derivative in balls
by the $L^2$ norm of the full gradient in annuli.
To prove the estimate, 
we will first apply the integral stability inequality with the test functions 
\[
\cc(x) = x \cdot \nabla u \quad \text{ and } \quad \eta = |x|_{A^{-1}(0)}^{\frac{2-n}{2}} \zeta,
\]
where $\zeta$ is a cut-off.
Our choice will yield the desired bound with an additional error term involving a weighted integral of the ``Hessian times the gradient'',
which we will be able to treat thanks to the a priori estimates on annuli from Theorem~\ref{thm:sz}.

We start by choosing $\cc = x \cdot \nabla u$ and a generic test function $\eta$
in the integral stability inequality:

\begin{lemma}\label{lemma:int:radial}
Let $u \in C^{\infty}\big(\overline{B}_{1}\big)$ be a stable solution of 
$- L u = f(u)$ in $B_1$,
for some function $f \in C^1(\R)$.
Assume that 
\[
\| D A \|_{C^0(\overline{B}_1)} + \| \vv \|_{C^0(\overline{B}_1)} \leq \varepsilon
\]
for some $\varepsilon > 0$.

Then
\[
\begin{split}
& \int_{B_1} |\nabla u|_{A(x)}^2 \Big((n-2)\eta^2 +  x\cdot \nabla (\eta^2)\Big) \d x \\
&+ \int_{B_1} \Big( -2 (x\cdot \nabla u) A(x) \nabla u \cdot \nabla (\eta^2)-|x \cdot \nabla u|^2 |\nabla \eta|_{A(x)}^2 \Big) \d x\\
& \quad \quad \quad  \quad \quad
\leq C \varepsilon \int_{B_1} |D^2 u| |\nabla u| |x|^2 \eta^2 \d x\\
& \quad \quad \quad  \quad \quad \quad \quad \quad 
+ C \varepsilon \int_{B_1} |\nabla u|^2 \Big( |x|^2 |\nabla(\eta^2)| + \left(|x| + |x|^2 \varepsilon \right) \eta^2 \Big) \d x
\end{split}
\]
for all $\eta \in C_c^{\infty}(B_1)$,
where $C$ is a universal constant.
\end{lemma}
\begin{proof}
Throughout the proof, $C$ denotes a generic universal constant.
Testing the integral stability inequality \eqref{stable:jacobi} with $\eta$ and $\cc = x \cdot \nabla u$,
we deduce
\begin{equation}
\label{bat}
\begin{split}
\int_{B_1}
(x\cdot \nabla u)  \jacobi (x \cdot \nabla u) \, 
\eta^2
\d x 
&\leq \int_{B_1} 
|x \cdot \nabla u|^2 \left|\nabla \eta-{\textstyle\frac{1}{2}} \eta A^{-1}(x) \bb(x)\right|^2_{A(x)}
\d x.
\end{split}
\end{equation}
The quadratic term on the right-hand side of \eqref{bat} can be bounded by
\[
\left|\nabla \eta-{\textstyle\frac{1}{2}} \eta A^{-1}(x) \bb(x)\right|^2_{A(x)} 
\leq |\nabla \eta|^2_{A(x)} +C\varepsilon |\nabla (\eta^2)| + C \varepsilon^2 \eta^2,
\]
and hence
\begin{equation}\label{int:rad:1}
\begin{split}
\int_{B_1} (x\cdot \nabla u) 
\jacobi (x \cdot \nabla u) \,  \eta^2
\d x 
&\leq \int_{B_1} 
|x \cdot \nabla u|^2 
|\nabla \eta|_{A(x)}^2 \d x\\
& \quad \quad + C \varepsilon \int_{B_1} 
|\nabla u|^2 |x|^2
\big( |\nabla (\eta^2)| + \varepsilon \eta^2 \big) 
\d x.
\end{split}
\end{equation}

To compute the Jacobi operator $\jacobi (x \cdot \nabla u) = L (x \cdot \nabla u) + f'(u) (x \cdot \nabla u)$,
we differentiate the equation $-L u = f(u)$ in the direction of $x$,
which yields
\begin{equation}
\label{x:der}
- x \cdot \nabla ( Lu ) = f'(u) \left (x \cdot \nabla u\right ),
\end{equation}
and hence
\begin{equation}
\label{bef}
\begin{split}
\jacobi(x \cdot \nabla u) 
&= L \left (x \cdot \nabla u\right ) - x \cdot \nabla \left (L u\right ) \\
&= 2a_{ij}(x) u_{ij} - x_k \partial_k a_{ij}(x) u_{ij} + \vv_i(x) u_i - x_j \partial_j \vv_i(x) u_i.
\end{split}
\end{equation}
From \eqref{bef}, by the coefficient bounds, it follows that
\begin{equation}
\label{bef2}
\begin{split}
(x \cdot \nabla u)\jacobi(x \cdot \nabla u)  &\geq 2 x_k u_k a_{ij}(x) u_{ij}  - x_k u_k x_j \partial_j \vv_i(x) u_i \\
& \quad \quad \quad - C \varepsilon |x|^2 |D^2 u | |\nabla u| - C \varepsilon |x| |\nabla u|^2.
\end{split}
\end{equation}

The idea now is to integrate by parts to get rid of the highest order terms on the left-hand side of \eqref{int:rad:1}.
For this we must rewrite the term $2 x_k u_k \, a_{ij}(x) u_{ij}$ in \eqref{bef2} in divergence form.
By the chain rule, we have
\begin{equation}
\label{ab1}
x_{k} u_{k} \, a_{ij}(x) u_{ij}= \partial_i\big(x_{k} u_{k}  \,  a_{ij}(x) u_{j}\big) - x_{k} u_{k} \partial_i a_{ij}(x) u_{j} - a_{ij}(x)u_{i}u_{j} - x_{k}a_{ij}(x)u_{ik} u_{j}.
\end{equation}
Using that $a_{ij}(x) = a_{ji}(x)$,
the last term in \eqref{ab1} can be written as
\[
x_{k}a_{ij}(x)u_{ik} u_{j} = \frac{1}{2} \partial_k \big(a_{ij}(x)u_i u_j x_{k}\big) - \frac{n}{2} a_{ij}(x) u_i u_j - \frac{1}{2} x_{k}\partial_k a_{ij}(x) u_i u_j,
\]
and hence
\begin{equation}\label{a:expression}
\begin{split}
2 x_{k}u_k \, a_{ij}(x)u_{ij}&
= \partial_i \Big(2 x_{k} u_{k} \, a_{ij}(x) u_{j} - a_{jk}(x)u_j u_k x_{i}\Big) 
+ (n-2) a_{ij}(x)u_{i}u_{j} \\
&\quad \quad  - 2 x_{k} u_{k} \partial_i a_{ij}(x) u_{j} 
+ x_{k}\partial_k a_{ij}(x) u_i u_j \\
& \geq \div\left(2(x\cdot \nabla u) A(x)\nabla u - |\nabla u|^2_{A(x)} x\right) + (n-2) |\nabla u|^2_{A(x)} \\
& \quad \quad - C \varepsilon |x| |\nabla u|^2,
\end{split}
\end{equation}
where in the last inequality we have used the estimates for the coefficients.
Combining \eqref{a:expression} and \eqref{bef2},
we obtain
\begin{equation}\label{rad:1}
\begin{split}
& (x \cdot \nabla u) \jacobi (x \cdot \nabla u)\\
&\quad \quad  \geq \div\left(2(x\cdot \nabla u) A(x)\nabla u - |\nabla u|^2_{A(x)} x\right) + (n-2) |\nabla u|^2_{A(x)} \\
& \quad \quad \quad \quad \quad  - x_k u_k x_j \partial_j \vv_i(x) u_i - C \varepsilon |x|^2 |D^2 u | |\nabla u| - C \varepsilon |x| |\nabla u|^2.
\end{split}
\end{equation}

Multiplying \eqref{rad:1} by $\eta^2$ and integrating, the third term $-\int_{B_1} x_k u_k x_j \partial_j \vv_i(x) u_i \eta^2 \d x$ can be integrated by parts
and estimated by
\begin{equation}
\label{errorbd3}
\begin{split}
\left|-\int_{B_1} x_k u_k x_j \partial_j \vv_i(x) u_i \eta^2 \d x \right| &=  \left|\int_{B_1}  \vv_i(x) \partial_j (x_k u_k x_j u_i  \eta^2)  \d x\right|\\
& \leq C \varepsilon \int_{B_1}  |D^2 u | |\nabla u| |x|^2 \eta^2 \d x \\
&\quad\quad\quad + C \varepsilon \int_{B_1}  |\nabla u|^2 \left(  |x|^2 |\nabla (\eta^2)| + |x| \eta^2 \right)\d x.
\end{split}
\end{equation}
Substituting \eqref{rad:1} in the inequality \eqref{int:rad:1},
rearranging terms and
by the error bound \eqref{errorbd3},
it follows that
\begin{equation}\label{int:rad:2}
\begin{split}
&\int_{B_1} \left((n-2) |\nabla u|^2_{A(x)} \eta^2 - |x \cdot \nabla u|^2 |\nabla \eta|^2_{A(x)} \right)\d x\\
&+ \int_{B_1} \div\Big(2(x\cdot \nabla u) A(x)\nabla u - |\nabla u|^2_{A(x)} x\Big) \eta^2 \d x \\
& \quad \leq 
C \varepsilon \int_{B_1} |D^2 u | |\nabla u| |x|^2 \eta^2 \d x
+ C  \varepsilon \int_{B_1} |\nabla u|^2 \Big( |x|^2 |\nabla(\eta^2)| + (|x| + \varepsilon |x|^2) \eta^2 \Big) \d x.
\end{split}
\end{equation}
Integrating by parts the divergence term on the left-hand side of \eqref{int:rad:2} yields the claim.
\end{proof}

\begin{remark}
\label{rad:lessreg}
In \eqref{x:der} we took a derivative of the equation in the $x$ direction
to get rid of the dependence on the nonlinearity.
Instead, we could have multiplied the equation 
by the test function $\xi = \div \left( x \, (x \cdot \nabla u)  \eta^2 \right)$
and integrated by parts.
Thanks to this, we avoid having to take any derivatives of $\vv$, 
since the term involving it can be bounded directly.
Notice also that we need $u$ to have three derivatives to be able to compute $L \cc$.
This is the same phenomenon as in the proof of Theorem~\ref{thm:sz};
see the discussion in Remark~\ref{sz:lessreg}.
\end{remark}

\begin{remark}
\label{remark:azero}
Since $|A(x) - A(0)| \leq C \varepsilon |x|$,
the inequality in Lemma~\ref{lemma:int:radial} also holds
if we replace $A(x)$ by the constant matrix $A(0)$
and we add an additional error term $C \varepsilon \int_{B_1} |\nabla u|^2 |x|^3 |\nabla \eta|^2 \d x$
on the right-hand side.
For future use, the final estimate involving $A(0)$ instead of $A(x)$ reads as
\begin{equation}
\label{radeasy}
\begin{split}
& \int_{B_1} |\nabla u|_{A(0)}^2 \Big((n-2)\eta^2 +  x\cdot \nabla (\eta^2)\Big) \d x \\
&+ \int_{B_1} \Big( -2 (x\cdot \nabla u) A(0) \nabla u \cdot \nabla (\eta^2)-|x \cdot \nabla u|^2 |\nabla \eta|_{A(0)}^2 \Big) \d x\\
& \quad \quad \quad  \quad \quad
\leq C \varepsilon \int_{B_1} |D^2 u| |\nabla u| |x|^2 \eta^2 \d x \\
& \quad \quad \quad  \quad \quad \quad \quad \quad + C \varepsilon \int_{B_1} |\nabla u|^2 \Big(|x|^{3} |\nabla \eta|^2  + (|x| + \varepsilon |x|^2) \eta^2 \Big) \d x,
\end{split}
\end{equation}
where
we have used that $|x|^2 |\nabla (\eta^2)| \leq |x|^3 |\nabla \eta|^2 + |x| \eta^2$.
\end{remark}

Next, we choose the singular test function $\eta = |x|_{A^{-1}(0)}^{-a/2} \zeta$ in Lemma~\ref{lemma:int:radial},
where the exponent $a \geq 0$ will satisfy $a \leq n-2$ when $n \leq 9$.
Recall our notation 
for the modulus of the position vector and the radial derivative
\[
r = |x| \quad \text{ and } \quad  u_r = \dfrac{x}{|x|} \cdot \nabla u.
\]

\begin{lemma}
\label{lemma:prerad}
Let $u \in C^{\infty}\big(\overline{B}_{1}\big)$ 
be a stable solution of 
$- L u = f(u)$ in $B_1$,
for some function $f \in C^1(\R)$.
Assume that 
\[
\| D A \|_{C^0(\overline{B}_1)} + \| \vv \|_{C^0(\overline{B}_1)} \leq \varepsilon
\]
for some $\varepsilon > 0$.

If
\begin{equation}
\label{range}
0 \leq a \leq \min\{10, n\} - 2,
\end{equation}
then
\begin{equation}
\label{rad:almost}
\begin{split}
&(n-2-a) \int_{B_{\rho}}  r^{-a} |\nabla u|^2 \d x + \frac{a(8-a)}{4} \int_{B_{\rho}} r^{-a} u_r^2 \d x\\
& \quad \quad \quad  \leq  C \int_{B_{2\rho} \setminus B_{\rho}} r^{-a} |\nabla u|^2 \d x  +C \varepsilon \int_{B_{2\rho}} r^{2-a} |D^2 u| |\nabla u|  \d x \\
& \quad \quad \quad \quad \quad \quad  
 + C \varepsilon \int_{B_{2\rho}} \big( r^{1-a} + \varepsilon r^{2-a}\big) |\nabla u|^2 \d x.
\end{split}
\end{equation}
for all $\rho \leq 1/2$,
where $C$ is a universal constant.
\end{lemma}
\begin{proof}
By approximation, 
the inequality in Lemma \ref{lemma:int:radial} is valid for Lipschitz test functions $\eta \in C^{0,1}_c(B_1)$.
Moreover, this inequality also holds for the singular test function
\[
\eta := |x|_{A^{-1}(0)}^{-a/2}\zeta,
\]
where $\zeta \in C^{0,1}_c(B_1)$ is a cut-off.
To see this, for $\delta > 0$ consider the $C^{0,1}_c$ approximation
\[
\eta_{\delta} = \min \{|x|_{A^{-1}(0)}^{-a/2}, \delta^{-a/2}\}\zeta
\]
and apply dominated convergence to take the limit as $\delta \to 0$.

By Remark~\ref{remark:azero},
it suffices to compute the left-hand side of the inequality in Lemma~\ref{lemma:int:radial}
with $A(0)$ in place of $A(x)$.
Since
\[
\nabla (\eta^2) = - a \zeta^2 
|x|_{A^{-1}(0)}^{-(a+2)} 
A^{-1}(0) x  + |x|_{A^{-1}(0)}^{-a} \nabla (\zeta^2),
\]
the first integrand in \eqref{radeasy} is equal to
\begin{equation}\label{int:decay:0}
\begin{split}
&|\nabla u|_{A(0)}^2 \Big( (n-2)\eta^2 + x \cdot \nabla(\eta^2) \Big) \\
&\quad = (n-2-a) |x|_{A^{-1}(0)}^{-a}|\nabla u|_{A(0)}^2 \zeta^2 + |x|_{A^{-1}(0)}^{-a} |\nabla u|_{A(0)}^2 (x \cdot \nabla (\zeta^2)).
\end{split}
\end{equation}
Moreover, since
\[
|\nabla \eta|^2_{A(0)} = 
\frac{a^2}{4} |x|_{A^{-1}(0)}^{-(a+2)} \zeta^2
- \frac{a}{2} |x|_{A^{-1}(0)}^{-(a+2)} \big(x \cdot \nabla(\zeta^2) \big)
+ |x|_{A^{-1}(0)}^{-a} |\nabla \zeta|^2_{A(0)},
\]
the second integrand is
\begin{equation}\label{int:decay:1}
\begin{split}
&-2(x\cdot \nabla u)
A(0) \nabla u \cdot \nabla(\eta^2)
-|x \cdot \nabla u|^2 |\nabla \eta|_{A(0)}^2 
\\
&\quad\quad= 
\frac{a(8-a)}{4}  |x|_{A^{-1}(0)}^{-(a+2)} |x\cdot \nabla u|^2 \zeta^2
-2|x|_{A^{-1}(0)}^{-a} (x\cdot \nabla u) A(0) \nabla u \cdot \nabla (\zeta^2)\\
&
\quad \quad \quad \quad 
-|x \cdot \nabla u|^2 |x|_{A^{-1}(0)}^{-a} |\nabla \zeta|_{A(0)}^2 + \frac{a}{2} |x \cdot \nabla u |^2 |x|_{A^{-1}(0)}^{-a-2} \big( x \cdot \nabla (\zeta^2) \big).
\end{split}
\end{equation}

From the identities \eqref{int:decay:0} and \eqref{int:decay:1},
by \eqref{radeasy},
it follows that
\begin{equation}\label{int:decay:6}
\begin{split}
& (n-2-a) \int_{B_1} |x|_{A^{-1}(0)}^{-a} |\nabla u|_{A(0)}^2 \zeta^2 \d x  + \dfrac{a(8-a)}{4} \int_{B_1} |x|_{A^{-1}(0)}^{-a-2} |x \cdot \nabla u |^2 \zeta^2 \d x\\
& \leq  C \int_{B_1} \big( r^{2-a}  |\nabla \zeta|^2 + r^{1-a} |\nabla(\zeta^2)| \big)|\nabla u|^2 \d x  +C \varepsilon \int_{B_1} r^{2-a} |D^2 u| |\nabla u| \zeta^2 \d x  \\
& \quad \quad 
+ C \varepsilon \int_{B_1} \big( r^{3-a} |\nabla \zeta|^2 + r^{2-a} |\nabla(\zeta^2)| \big)|\nabla u|^2 \d x
+ C \varepsilon \int_{B_1} \big(r^{1-a} + \varepsilon r^{2-a}\big)|\nabla u|^2 \zeta^2 \d x,
\end{split}
\end{equation}
for some universal constant $C$,
where we have controlled the remainder terms
thanks to the uniform ellipticity 
and the fact that $a$ is bounded by a dimensional constant.

For $ 0 < \rho \leq 1/2$
as in the statement,
we consider a Lipschitz function $\zeta$,
with $0 \leq \zeta \leq 1$, such that  $\zeta|_{B_{\rho}} = 1$,
$\supp \zeta \subset \overline{B}_{2\rho}$, and $|\nabla \zeta| \leq C/\rho$.
Plugging this cutoff function in \eqref{int:decay:6}, using that $r$ is comparable with $\rho$ inside $\supp \nabla \zeta \subset \overline{B}_{2\rho} \setminus B_{\rho}$,
we deduce that
\begin{equation}\label{int:decay:lul}
\begin{split}
& (n-2-a) \int_{B_1} |x|_{A^{-1}(0)}^{-a} |\nabla u|_{A(0)}^2 \zeta^2 \d x  + \dfrac{a(8-a)}{4} \int_{B_1} |x|_{A^{-1}(0)}^{-a-2} |x \cdot \nabla u |^2 \zeta^2 \d x\\
& \quad \quad \quad  \leq  C \int_{B_{2\rho} \setminus B_{\rho}} r^{-a} |\nabla u|^2 \d x  +C \varepsilon \int_{B_{2\rho}} r^{2-a} |D^2 u| |\nabla u|  \d x \\
& \quad \quad \quad \quad \quad \quad \quad + C \varepsilon \int_{B_{2\rho}} \big(r^{1-a} + \varepsilon r^{2-a}\big) |\nabla u|^2 \d x.
\end{split}
\end{equation}
Since $a$ is in the range \eqref{range},
the constants in the left-hand side of \eqref{int:decay:lul} are nonnegative.
Moreover,
by uniform ellipticity 
we have
$|x|_{A^{-1}(0)} \leq \elliptic^{-1/2} |x|$
and $|\nabla u|^2_{A(0)} \geq \elliptic |\nabla u|^2$,
hence, since $\zeta|_{B_{\rho}} = 1$ and $\zeta \geq 0$, it follows that
\begin{equation}
\label{int:lambda}
\begin{split}
&\lambda^{a/2 + 1}\left( (n-2-a)\int_{B_{\rho}} r^{-a} |\nabla u|^2 \d x + \frac{a(8-a)}{4} \int_{B_{\rho}} r^{-a} u_r^2 \d x\right)\\
&\leq
 (n-2-a) \int_{B_1} |x|_{A^{-1}(0)}^{-a} |\nabla u|_{A(0)}^2 \zeta^2 \d x  + \dfrac{a(8-a)}{4} \int_{B_1} |x|_{A^{-1}(0)}^{-a-2} |x \cdot \nabla u |^2 \zeta^2 \d x.
\end{split}
\end{equation}
Using \eqref{int:lambda} in \eqref{int:decay:lul} and multiplying by $\lambda^{-a/2-1}$ now yields the claim.
\end{proof}

We can finally prove Proposition~\ref{prop:rad}.
For this, we will apply Lemma~\ref{lemma:prerad} with the exponent $a = n-2$.
The key point in the proof will be to control the weighted $L^1$ norm of $|D^2 u| |\nabla u|$ in the right-hand side of \eqref{rad:almost}
by 
a weighted $L^2$ norm of the gradient.
We obtain this bound by writing the integral as an infinite sum on dyadic annuli and by using that 
the weight in each annulus can be pulled out of the integral.
This allows us to apply the non-weighted a priori estimate
for the ``Hessian times the gradient'' \eqref{ineq:hessgrad2} from Theorem~\ref{thm:sz}.

\begin{proof}[Proof of Proposition~\ref{prop:rad}]
Since $3 \leq n \leq 9$, we have that $\min\{10, n\}-2  = n-2$
and we may choose the exponent $a = n-2$ in Lemma~\ref{lemma:prerad}, which yields the inequality
\begin{equation}
\label{expo}
\begin{split}
&\frac{(n-2)(10-n)}{4} \int_{B_{\rho}} r^{2-n} u_r^2 \d x\\
& \quad \quad \quad  \leq  C \int_{B_{2\rho} \setminus B_{\rho}} r^{2-n} |\nabla u|^2 \d x  +C \varepsilon \int_{B_{2\rho}} r^{4-n} |D^2 u| |\nabla u|  \d x \\
& \quad \quad \quad \quad \quad \quad  
 + C \varepsilon \int_{B_{2\rho}} \big( r^{3-n} + \varepsilon r^{4-n}\big) |\nabla u|^2 \d x.
\end{split}
\end{equation}
To bound the Hessian term $\int_{B_{2\rho}} r^{4-n} |D^2 u| |\nabla u| \d x$ on the right-hand side of \eqref{expo},
we will apply the a priori estimate on annuli \eqref{ineq:hessgrad2} from Theorem~\ref{thm:sz} at different scales.

Let $r_{j} := 2^{-j}$ with $j \geq 0$.
As explained in Section~\ref{section:prelim}, the functions $u(r_{j} \cdot)$ are stable solutions to a semilinear equation with coefficients $A(r_{j} \cdot)$ and $r_{j} \vv( r_{j} \cdot)$.
In particular, since $\|D A\|_{C^0(\overline{B}_1)} + \|\vv\|_{C^0(\overline{B}_1)} \leq \varepsilon$, 
we also have
$\|D A(r_{j} \cdot )\|_{C^0(\overline{B}_1)} + \| r_{j}\vv(r_{j} \cdot )\|_{C^0(\overline{B}_1)} \leq \varepsilon r_{j} \leq \varepsilon$.
Hence, 
by \eqref{ineq:hessgrad2} in Theorem~\ref{thm:sz} applied to $u(r_{j} \cdot)$,
there is a universal $\varepsilon_0 > 0$ 
with the following property:
if $\varepsilon \leq \varepsilon_0$, then 
\begin{equation}
\label{bad}
\int_{B_{r_{j+1}} \setminus B_{r_{j+2}}} |D^2 u| |\nabla u| \d x \leq C r_j^{-1} \int_{B_{r_{j}} \setminus B_{r_{j+3}}} |\nabla u|^2 \d x \quad \text{ for all } j \geq 0,
\end{equation}
where $C$ is a universal constant.

Writing the weighted integral as an infinite sum on annuli, we have
\begin{equation}
\label{bad22}
\begin{split}
\int_{B_{1/2}} r^{4-n} |D^2 u| |\nabla u| \d x &
= \sum_{j = 0}^{\infty} \int_{B_{r_{j+1}} \setminus B_{r_{j+2}}} r^{4-n} |D^2 u| |\nabla u| \d x\\
& \leq C \sum_{j = 0}^{\infty} r_j^{4-n} \int_{B_{r_{j+1}} \setminus B_{r_{j+2}}} |D^2 u| |\nabla u| \d x,
\end{split}
\end{equation}
where in the last line we have used that $r^{4-n} \leq C r_{j}^{4-n}$ in $B_{r_{j+1}} \setminus B_{r_{j+2}}$, with $C$ universal.
Multiplying \eqref{bad} by $r_j^{4-n}$ and summing in $j$, the right-hand side in \eqref{bad22} can be bounded by
\begin{equation}
\label{bad23}
\begin{split}
\sum_{j = 0}^{\infty} r_{j}^{4-n} \int_{B_{r_{j+1}} \setminus B_{r_{j+2}}} |D^2 u| |\nabla u| \d x  &\leq C \sum_{j = 0}^{\infty} r_{j}^{3-n} \int_{B_{r_{j}} \setminus B_{r_{j+3}}} |\nabla u|^2 \d x\\
&\leq C \sum_{j = 0}^{\infty} \int_{B_{r_{j}} \setminus B_{r_{j+3}}} r^{3-n} |\nabla u|^2 \d x\\
& \leq C \int_{B_1} r^{3-n} |\nabla u| \d x.
\end{split}
\end{equation}
Combining \eqref{bad22} and \eqref{bad23}, we deduce
\begin{equation}
\label{bad2}
\int_{B_{1/2}} r^{4-n} |D^2 u | |\nabla u| \d x \leq C \int_{B_{1}} r^{3-n} |\nabla u|^2 \d x,
\end{equation}
where $C$ is universal.
Applying \eqref{bad2} to the stable solutions $u(4 \rho \cdot)$, 
there is a universal $\varepsilon_0 > 0$ with the following property:
if $\varepsilon \leq \varepsilon_0$, then
\begin{equation}
\label{bad3}
\int_{B_{2\rho}} r^{4-n} |D^2 u | |\nabla u| \d x \leq C \int_{B_{4\rho}} r^{3-n} |\nabla u|^2 \d x \quad \text{ for all } \rho \leq 1/4,
\end{equation}
where $C$ is a universal constant.

Applying \eqref{bad3} in \eqref{expo},
we deduce the key estimate
\begin{equation}
\label{bad4}
\begin{split}
&\frac{(n-2)(10-n)}{4} \int_{B_{\rho}} r^{2-n} u_r^2 \d x \\
& \quad \quad \quad \quad \leq C \int_{B_{2\rho} \setminus B_{\rho}} r^{2-n} |\nabla u|^2 \d x + C \varepsilon \int_{B_{4\rho}} r^{3-n} |\nabla u|^2 \d x,
\end{split}
\quad \text{ for } \rho \leq 1/4 \text{ and } \varepsilon \leq \varepsilon_0,
\end{equation}
where we have additionally bounded the last integrand in \eqref{expo} by $\varepsilon r^{4-n} \leq C r^{3-n}$.
Finally, since $(n-2)(10-n) > 0$, absorbing this constant on the right-hand side of \eqref{bad4} yields the claim.
\end{proof}

\begin{remark}
\label{remark:decay}
Our proof 
in Section~\ref{section:holder}
of the $C^{\alpha}$ bound~\eqref{eq:holder} from Theorem~\ref{thm:main}
controls the weighted integral $\int_{B_{\rho}} r^{2-n} |\nabla u|^2 \d x$.
It will require a delicate estimate proven in Section~\ref{section:radl1}.
As a consequence, we will also obtain a bound for the less singular error terms $\int_{B_{\rho}} r^{3-n} |\nabla u|^2 \d x$.
Here we point out that this last quantity can be estimated directly 
from our previous Lemma~\ref{lemma:prerad},
without the use of Section~\ref{section:radl1}.
This is done as follows.

Letting $a = n-3$ in Lemma~\ref{lemma:prerad},
since $a(8-a) = (n-3)(11-n) \geq 0$ for $3 \leq n \leq 11$,
we can drop the radial term in \eqref{rad:almost}
and the left-hand side becomes $\int_{B_{\rho}} r^{3-n} |\nabla u|^2 \d x$.
The right-hand side now includes an error term $\varepsilon \int_{B_{2\rho}} r^{5-n} |D^2 u| |\nabla u| \d x$,
which can be bounded by $\varepsilon \int_{B_{4\rho}} r^{4-n} |\nabla u|^2 \d x$
for $\varepsilon \leq \varepsilon_0$
as in the proof of Proposition~\ref{prop:rad}.
Hence, we obtain
\[
\int_{B_{\rho}} r^{3-n} |\nabla u|^2 \d x \leq C \int_{B_{4\rho} \setminus B_{\rho}} r^{3-n} |\nabla u|^2 \d x + C \varepsilon \int_{B_{\rho}} r^{4-n} |\nabla u|^2 \d x \quad \text{ for all } \rho \leq 1/4.
\]
Making $\varepsilon_0$ smaller if necessary, since $r^{4-n} \leq r^{3-n}$ in $B_{\rho}$, we can absorb the rightmost term into the left-hand side.
This leads to an expression which can be hole-filled,
and by a standard iteration argument it is easy to deduce the decay estimate
\begin{equation}
\label{olddec}
\int_{B_{\rho}} r^{3-n} |\nabla u|^2 \d x \leq C \rho^{2\alpha} \|\nabla u\|_{L^2(B_1)}^2 \quad \text{ for all } \rho \leq 1/4,
\end{equation}
where $\alpha > 0$ and $C$ are universal constants.
We note that,
adding superfluous variables, the decay \eqref{olddec} is satisfied when $n \leq 11$.

In general,
this strategy allows to prove decay 
of the weighted integrals $\int_{B_{\rho}} r^{-a} |\nabla u|^2 \d x$
for exponents 
$a < n - 2$ and $a \leq 8$.
As an application, 
we could extend the optimal regularity
result of Peng, Zhang, and Zhou~\cite{PengZhangZhouBMO} for stable solutions in dimensions $n \geq 10$ to our setting of variable coefficients.
The key point in \cite{PengZhangZhouBMO} is to find an a a priori estimate of the form $\rho^{-a_n} \int_{B_{\rho}} |\nabla u|^2\d x \leq C \|\nabla u\|^2_{L^2(B_1)}$,
where $a_n = 2 (1 + \sqrt{n-1})$ is a critical exponent.
When considering variable coefficients,
one has to deal with error terms 
$\varepsilon \int_{B_{\rho}} r^{1-a_n} |\nabla u|^2 \d x$.
Our method above gives the decay of these error terms whenever $10 \leq n \leq 13$.
For this, we choose 
$a = a_n - 1$ in \eqref{rad:almost},
and notice that $a_n - 1 < n -2$ and $a_n -1 \leq 8$ 
in this dimension range.
The case $n \geq 14$ can be treated similarly
using that $a = a_n -1 \geq 8$ and $u_r^2 \leq |\nabla u|^2$ in \eqref{rad:almost}.
\end{remark}

\begin{remark}
The key estimate in the proof of Proposition~\ref{prop:rad} is an inequality for weighted integrals,
\begin{equation}
\label{auxhess}
\int_{B_{2\rho}} r^{4-n} |D^2 u| |\nabla u| \d x \leq C \int_{B_{4\rho}} r^{3-n} |\nabla u|^2 \d x,
\end{equation}
which has been proved decomposing the integral in dyadic annuli.
There is a way to prove a weaker inequality than \eqref{auxhess}, 
namely,
\begin{equation}
\label{weighted:bad}
\begin{split}
\int_{B_{2\rho}} r^{4-n} |D^2 u| |\nabla u| \d x &\leq C \left(\int_{B_{4\rho}} r^{2-n} |\nabla u|^2 \d x\right)^{1/2}\left( \int_{B_{4\rho}} r^{4-n} |\nabla u|^2 \d x\right)^{1/2} \\
& \quad \quad \quad \quad \quad \quad + C \int_{B_{4\rho}} r^{3-n} |\nabla u|^2 \d x,
\end{split}
\end{equation}
which avoids the use of annuli
and which involves a very recent test function of
Cabr\'{e}~\cite{Cabre-StableHardy} and Peng, Zhang, and Zhou \cite{PengZhangZhouHolder}.

To obtain \eqref{weighted:bad},
one uses the inequality \eqref{ineq:hessgrad}
from the proof of Theorem~\ref{thm:sz}
with the singular test function 
$\eta = r^{-n/2} \zeta$, 
where $\zeta \in C^{\infty}_c(B_{1})$ is a cut-off.
It is worth noting
that this inequality relies on the Sternberg-Zumbrun estimate for $\int_{B_1} \anew^2 \eta^2 \d x$,
which comes from choosing the test function 
$\xi = |\nabla u| \eta = |\nabla u| r^{-n/2} \zeta$ 
in the integral stability inequality \eqref{stable:int}.
A function of this form had already appeared in \cite{Cabre-StableHardy},
where the author considered $\xi = |\nabla u| r^{-(n+ \varepsilon)/2} \zeta$ 
to prove the boundedness of stable solutions
for unsigned nonlinearities when $n \leq 4$.
Interestingly, 
our choice $\xi = |\nabla u| r^{-n/2}$ 
coincides with the 
test function used by the authors in \cite{PengZhangZhouHolder},
where they obtained an a priori H\"{o}lder estimate 
for stable solutions when $n \leq 5$.
While their strategy involves integrating by parts an expression that is already quadratic in the gradient,
expressing the new ``coupled'' Hessian errors $D^2 u \nabla u$ as the uncoupled expression $\Delta u \nabla u$ and an error term related to $\azero$,
we do not integrate by parts.

By Young's inequality, the coefficient in front of the highest order term $\int_{B_{4\rho}} r^{2-n} |\nabla u|^2 \d x$ in \eqref{weighted:bad} can be made arbitrarily small,
and this term can still be treated by our methods 
to yield the H\"{o}lder estimate from Theorem~\ref{thm:main}.
\end{remark}

\section{The radial derivative controls the function in $L^1$}
\label{section:radl1}

The goal of this section is to  prove Proposition~\ref{prop:l1rad},
where we control the $L^1$ norm of (generalized) superharmonic functions $L u \leq 0$ by the $L^1$ norm of the radial derivative on annuli.
By a comparison argument, it will suffice to derive the analogue $L^1$ estimate on spheres for harmonic functions $L v = 0$,
which are obtained by duality from the $L^{\infty}$ estimates of a Neumann problem.

Let $g \in C^{\infty}(\partial B_1)$,
and consider the Neumann problem in divergence form
\begin{equation}
\label{aux:neumann}
\left\{
\begin{array}{rl}
\div(A(x) \nabla \varphi) 
= 0 &\text{ in } B_1\\
A(x) \nabla \varphi \cdot \nu = g &\text{ on } \partial B_1,
\end{array}
\right.
\end{equation}
which
admits solutions if and only if $\int_{\partial B_1} g \d \mathcal{H}^{n-1} = 0$.
Recall that the solutions of \eqref{aux:neumann} are unique up to addition of constants.
We will derive an $L^{\infty}$ a priori estimate for the zero mean solutions of \eqref{aux:neumann}
in terms of the conormal derivative
\begin{equation}
\label{def:conormal}
N \varphi := 
A(x) \nabla \varphi \cdot \frac{x}{|x|}.
\end{equation}
This is achieved by a Moser iteration
based on the following Sobolev trace inequality
\begin{equation}
\label{key:ineq}
\|u\|_{L^{2^{\star}}(\partial B_1)}^2 \leq C \big(\| \nabla u \|_{L^2(B_1)}^2 +  \| u \|_{L^2(\partial B_1)}^2\big)
\end{equation}
for $u \in W^{1,2}(B_1)$,
where $C$ depends only on $n$ and
\begin{equation}
\label{def:sobolev:exp}
2^{\star} := \frac{n-1}{n-2}2
\end{equation}
is the Sobolev trace exponent.
We give a short proof of this inequality in Appendix~\ref{app:trace:sobolev}.

Our proof by Moser iteration is inspired by the one of Winkert in \cite{Winkert},
where 
he
obtains $L^{\infty}$ estimates for general quasilinear Neumann problems.
While the author employs certain technical interpolation and trace inequalities from the theory of Besov and Lizorkin-Triebel spaces,
we only need the basic trace inequality \eqref{key:ineq}, 
for which we give an elementary proof.

\begin{lemma}
\label{moser}
Let $\varphi$ be the unique solution of \eqref{aux:neumann} with $\int_{B_1} \varphi \d x = 0$.
Then
\[
\|\varphi \|_{L^{\infty}(B_1)} \leq C \| g\|_{L^{\infty}(\partial B_1)},
\]
where $C$ depends only on $n$ and $\elliptic$.
\end{lemma}
\begin{proof}
Dividing the solution by the norm $\|g \|_{L^{\infty}(\partial B_1)}$, we may assume that $\|g\|_{L^{\infty}(\partial B_1)} = 1$.
By the maximum principle, it suffices to bound the $L^{\infty}$ norm of $\varphi$ on the sphere $\partial B_1$.
In this proof, $C$ always denotes a generic constant depending only on $n$ and $\elliptic$.

First, we obtain a basic energy estimate.
Multiplying the equation \eqref{aux:neumann} by $\varphi$ and integrating by parts, we have
\begin{equation}
\label{en0}
\int_{B_1} |\nabla \varphi|^2_{A(x)} \d x = \int_{\partial B_1} g \varphi \d \mathcal{H}^{n-1}.
\end{equation}
Combining the standard trace inequality $\|u\|^2_{L^2(\partial B_1)} \leq C \big( \|\nabla u\|_{L^2(B_1)}^2 + \|u\|_{L^2(B_2)}^2 \big)$
with the Poincaré inequality in the ball (recall that $\int_{B_1} \varphi = 0$),
we also have
\begin{equation}
\label{en15}
\|\varphi\|_{L^{2}(\partial B_1)} 
\leq C \|\nabla \varphi\|_{L^2(B_1)}.
\end{equation}
Hence, applying Cauchy-Schwarz in \eqref{en0},
by \eqref{en15} we obtain
\begin{equation}
\label{en1}
\int_{B_1} |\nabla \varphi|^2_{A(x)} \d x 
\leq \|g\|_{L^2(\partial B_1)} \|\varphi\|_{L^2(\partial B_1)}
\leq C 
\|g\|_{L^2 (\partial B_1)}
 \|\nabla \varphi\|_{L^2(B_1)}.
\end{equation}
By uniform ellipticity and the bound $\|g\|_{L^{\infty}(B_1)} = 1$, from \eqref{en1} it follows that
\begin{equation}
\label{energy1}
\|\nabla \varphi\|_{L^2( B_1)} \leq C.
\end{equation}
Therefore, by \eqref{key:ineq}, \eqref{en15}, and \eqref{energy1} we deduce the a priori estimate
\begin{equation}
\label{energy:est}
\|\varphi\|_{L^{2^{\star}}(\partial B_1)} \leq C.
\end{equation}

Next, we derive an $L^{\infty}$ bound for the positive part of 
the solution
by Moser iteration.
Let $m \geq 2$.
Multiplying the equation by
the power $(\varphi^{+})^{m-1}$
and integrating by parts
\begin{equation}
\label{moser:1}
(m-1) \int_{B_1} (\varphi^{+})^{m-2}|\nabla \varphi^{+}|_{A(x)}^2 \d x = \int_{\partial B_1} (\varphi^{+})^{m-1}  g \d \mathcal{H}^{n-1}.
\end{equation}
Since
\[
|\nabla (\varphi^{+})^{\frac{m}{2}}|^2_{A(x)} = \frac{m^2}{4} (\varphi^{+})^{m-2} |\nabla (\varphi^{+})|^2_{A(x)},
\]
by \eqref{moser:1} and the uniform ellipticity, using that $\frac{m}{m-1} \leq 2$ for $m \geq 2$, we have
\begin{equation}
\label{moser:2}
\int_{B_1} |\nabla (\varphi^{+})^{\frac{m}{2}}|^2 \d x
\leq C m \int_{\partial B_1} (\varphi^{+})^{m-1} \d \mathcal{H}^{n-1}.
\end{equation}
Adding the integral $\int_{\partial B_1} (\varphi^{+})^m \d \mathcal{H}^{n-1}$ to both sides of \eqref{moser:2}, we have
\[
\| \nabla (\varphi^{+})^{\frac{m}{2}} \|_{L^2(B_1)}^2 + \| (\varphi^{+})^{\frac{m}{2}} \|_{L^2(\partial B_1)}^2
\leq C m 
\|(\varphi^{+})^{m-1}\|_{L^{1}(\partial B_1)}
+ \| (\varphi^{+})^{\frac{m}{2}} \|_{L^2(\partial B_1)}^2
\]
and applying the Sobolev trace inequality \eqref{key:ineq}  on the left-hand side yields
\begin{equation}
\label{moser:3}
\|\varphi^{+}\|^{m}_{L^{\frac{2^{\star}}{2}m}(\partial B_1)}
\leq C m 
\|(\varphi^{+})^{m-1}\|_{L^{1}(\partial B_1)}
+ 
C \|\varphi^{+}\|_{L^{m}(\partial B_1)}^{m}.
\end{equation}
By H\"{o}lder and since $m \geq 2$, 
the $L^{m-1}$ norm in \eqref{moser:3} can be bounded by
\[
\|(\varphi^{+})^{m-1}\|_{L^{1}(\partial B_1)}
\leq |\partial B_1|^{\frac{1}{m}} \|\varphi^{+}\|_{L^{m}(\partial B_1)}^{m-1} \leq C \|\varphi^{+}\|_{L^{m}(\partial B_1)}^{m-1},
\]
and hence
\begin{equation}
\label{moser:4}
\|\varphi^{+}\|^{m}_{L^{\frac{2^{\star}}{2}m}(\partial B_1)}
\leq C m 
\|\varphi^{+}\|_{L^{m}(\partial B_1)}^{m-1}
+ 
C \|\varphi^{+}\|_{L^{m}(\partial B_1)}^{m}.
\end{equation}
Since $\|\varphi^{+}\|^{m-1}_{L^{m}(\partial B_1)} \leq \max\{1,  \|\varphi^{+}\|_{L^{m}(\partial B_1)} \}^{m}$,
from \eqref{moser:4}
it follows that
\begin{equation}
\label{case:3}
\|\varphi^{+}\|_{L^{\frac{2^{\star}}{2}m}(\partial B_1)}
\leq C^{\frac{1}{m}} m^{\frac{1}{m}} \max \{1, \|\varphi^{+}\|_{L^{m}(\partial B_1)} \}.
\end{equation}

We wish to iterate \eqref{case:3}.
Let $m_0 := 2^{\star}$
and,
for $k \in \N$, let
\[
m_k := \left(\frac{2^{\star}}{2}\right)^k m_0.
\]
By \eqref{case:3} and the definition of $m_k$, we have
\begin{equation}
\label{case:4}
\begin{split}
\|\varphi^{+}\|_{L^{m_{k}}(\partial B_1)}
&\leq C^{\frac{1}{m_{k-1}}} m_{k-1}^{\frac{1}{m_{k-1}}} \max \{1, \|\varphi^{+}\|_{L^{m_{k-1}}(\partial B_1)} \} \\
& = (C m_0)^{\frac{1}{m_{0}} \left(\frac{2}{2^{\star}}\right)^{k-1} }  \left(\frac{2^{\star}}{2}\right)^{\frac{k-1}{m_{0}} \left(\frac{2}{2^{\star}}\right)^{k-1} } \max \{1, \|\varphi^{+}\|_{L^{m_{k-1}}(\partial B_1)} \}.
\end{split}
\end{equation}

We have exactly one of the following three cases:
\begin{itemize}
\item \underline{Case 1}:
\begin{equation}
\label{cas1}
\|\varphi^{+}\|_{L^{m_{k}}(\partial B_1)} \leq 1
\end{equation}
\item \underline{Case 2}:
there is an $l \in \{1, 2, \ldots, k - 1\}$ such that
\begin{equation}
\label{cas2}
\|\varphi^{+}\|_{L^{m_{k+1-i}}(\partial B_1)} > 1 \quad  \text{ for } 1\leq i \leq l, \quad \text{ and } \quad \|\varphi^{+}\|_{L^{m_{k-l}}(\partial B_1)} \leq 1.
\end{equation}
\item \underline{Case 3}:
\begin{equation}
\label{cas3}
\|\varphi^{+}\|_{L^{m_{k+1-i}}(\partial B_1)} > 1 \quad \text{ for } 1 \leq i \leq k.
\end{equation}
\end{itemize}

Case $1$ already yields a uniform bound for $\|\varphi^{+}\|_{L^{m_k}(B_1)}$.
If Case $2$ holds then, iterating \eqref{case:4} $l -1$ times, we arrive at
\begin{equation}
\label{cas2con}
\begin{split}
\|\varphi^{+}\|_{L^{m_{k}}(\partial B_1)} &\leq
(C m_0)^{\frac{1}{m_{0}} \sum_{j = k-l}^{k-1} \left(\frac{2}{2^{\star}}\right)^{j} }  \left(\frac{2^{\star}}{2}\right)^{\frac{1}{m_{0}} \sum_{j = k-l}^{k-1} j\left(\frac{2}{2^{\star}}\right)^{j} }.
\end{split}
\end{equation}
The right-hand side of \eqref{cas2con} is nondecreasing in $l$ for
(say) $C \geq 1$,
which we can always assume.
Finally, if Case 3 holds
then, iterating \eqref{case:4}, we obtain
\begin{equation}
\label{cas3con}
\begin{split}
\|\varphi^{+}\|_{L^{m_{k}}(\partial B_1)} &\leq
(C m_0)^{\frac{1}{m_{0}} \sum_{j = 0}^{k-1} \left(\frac{2}{2^{\star}}\right)^{j} }  \left(\frac{2^{\star}}{2}\right)^{\frac{1}{m_{0}} \sum_{j = 0}^{k-1} j\left(\frac{2}{2^{\star}}\right)^{j} } \|\varphi^{+}\|_{L^{m_{0}}(\partial B_1)}.
\end{split}
\end{equation}
By the monotonicity of \eqref{cas2con} in $l$
and using the a priori estimate \eqref{energy:est} for $\|\varphi\|_{L^{m_0}(\partial B_1)}$ in \eqref{cas3con},
we see that in all three cases above we have
\[
\|\varphi^{+}\|_{L^{m_{k}}(\partial B_1)} \leq
C (C m_0)^{\frac{1}{m_{0}} \sum_{j = 0}^{k-1} \left(\frac{2}{2^{\star}}\right)^{j} }  \left(\frac{2^{\star}}{2}\right)^{\frac{1}{m_{0}} \sum_{j = 0}^{k-1} j\left(\frac{2}{2^{\star}}\right)^{j} },
\] 
and since the exponent on the right-hand side is uniformly bounded, we deduce 
\begin{equation}
\label{there:u:go}
\|\varphi^{+}\|_{L^{m_{k}}(\partial B_1)} \leq C.
\end{equation}
Taking the limit as $k \to \infty$ in \eqref{there:u:go} now yields
\[
\|\varphi^{+}\|_{L^{\infty}(\partial B_1)} \leq C,
\]
which is the desired $L^{\infty}$ estimate for the positive part of the solutions.
The same argument gives an a priori estimate for the negative part $\varphi^{-}$ and yields the claim.
\end{proof}

By duality, from the $L^{\infty}$ estimate in Lemma~\ref{moser}
we deduce an $L^1$ bound for the elliptic problem with a source:

\begin{lemma}
\label{prop:l1:estimate}
Given $h \in C^{\infty}(\overline{B_1})$,
let $v \in C^{\infty}(\overline{B}_{1})$
satisfy
\[
\div(A(x) \nabla v) + h(x) = 0 \quad \text{ in } B_1.
\]

Then
\[
\|v - t\|_{L^{1}(\partial B_1)} \leq C \|N v\|_{L^{1}(\partial B_1)} + C \|h \|_{L^1(B_1)},
\]
where
$t := \inf \{ \overline{t} \, \colon \,
|\{v > \overline{t}\} \cap \partial B_1|
\leq |\partial B_1| /2
\}$
 and 
$C$ depends only on $n$ and $\elliptic$.
\end{lemma}
\begin{proof}
Replacing $v$ by $v-t$ we may assume that $t = 0$, therefore
$|\{v > 0\} \cap \partial B_1| \leq |\partial B_1|/2$
and 
$|\{v < 0\} \cap \partial B_1| \leq |\partial B_1|/2$.
The function ${\rm sgn}(v) = v/|v|$ in $v \neq 0$ can then be extended to $\{v = 0\} \cap \partial B_1$,
taking values $\pm 1$
and in such a way that $\int_{\partial B_1} {\rm sgn}(v) \d \mathcal{H}^{n-1} = 0$.
In particular, $|v| = v \,{\rm sgn}(v)$ on $\partial B_1$.

We define the convolutions on $\partial B_1$
\[
g_k := {\rm sgn}(v) \star \eta_k,
\]
where $\{\eta_k\}$ is a sequence of smooth mollifiers on $\partial B_1$.
We have $g \in C^{\infty}(\partial B_1)$, $|g_k| \leq 1$, and $\int_{\partial B_1} g_k \d \mathcal{H}^{n-1} = 0$ since ${\rm sgn}(v)$ has zero average on $\partial B_1$.
Moreover, it holds that
\begin{equation}
\label{g:limit}
\int_{\partial B_1} |v| \d \mathcal{H}^{n-1} = \lim_{k} \int_{\partial B_1} v g_k \d \mathcal{H}^{n-1}.
\end{equation}

Since $g_k$ has zero average on $\partial B_1$, we can uniquely solve the Neumann problem
\[
\left\{
\begin{array}{rl}
\div(A(x) \nabla \varphi_k) 
= 0 &\text{ in } B_1\\
N \varphi_k = g_k &\text{ on } \partial B_1
\end{array}
\right.
\]
imposing additionally that $\int_{B_1} \varphi_k \d x = 0$.
By Lemma \ref{moser}, we deduce 
\begin{equation}
\label{apriori:neum}
\|\varphi_k\|_{L^{\infty}(\partial B_1)} \leq C,
\end{equation}
where $C$ depends only on $n$ and $\elliptic$.
Notice that, integrating by parts, we have
\[
\begin{split}
\int_{\partial B_1} v g_k \d \mathcal{H}^{n-1} &= \int_{\partial B_1} v N \varphi_k \d \mathcal{H}^{n-1} 
= \int_{\partial B_1} (N v) \varphi_k \d \mathcal{H}^{n-1} - \int_{B_1} \div(A(x) \nabla v)  \varphi_k \d x\\
&= \int_{\partial B_1} (N v) \, \varphi_k \d \mathcal{H}^{n-1} + \int_{B_1} h \,  \varphi_k \d x, 
\end{split}
\]
where in the last equality we have used the equation satisfied by $v$.
Hence, by \eqref{apriori:neum}
\[
\left|\int_{\partial B_1} v g_k \d \mathcal{H}^{n-1}  \right| 
\leq C  \int_{\partial B_1} |N v| \d \mathcal{H}^{n-1} + C \int_{B_1} |h|\d x, 
\]
and the claim follows by \eqref{g:limit}.
\end{proof}

We now use the previous estimates on spheres
to obtain $L^1$ bounds on annuli for a divergence-form operator with drift.
The drift term will be treated as a source,
which will appear as an error in the right-hand side of the estimate.
If
the coefficient matrix $A(x)$ is close to the identity,
then the conormal derivative $N u$ 
is
close to the radial derivative $u_r$.
Hence, 
we will obtain Proposition~\ref{prop:l1rad} as a corollary of the following:
\begin{proposition}
\label{prop:con:rad}
Let $u \in C^{\infty}(\overline{B}_{1})$ be a supersolution 
$\lcal u \leq 0$ in $B_{1}$,
where $\lcal$ is the operator
$\lcal u = \div(A(x) \nabla u) + \dd(x) \cdot \nabla u$.
Assume that 
\[
\| \dd \|_{C^0(\overline{B}_1)} \leq \varepsilon 
\]
for some $\varepsilon > 0$.

Then
there exists a constant $t$, which depends on $u$,
such that
\[
\|u - t\|_{L^{1}( B_1 \setminus B_{1/8})} \leq C \| u_{r} \|_{L^1( B_1 \setminus B_{1/8})} + C\|N u\|_{L^1(B_1\setminus B_{1/8})} + C \varepsilon \|\nabla u \|_{L^{1}(B_1)},
\]
where $C$ is a constant depending only on $n$ and $\elliptic$.
\end{proposition}
\begin{proof}
Since $\|N u\|_{L^1(B_1 \setminus B_{1/8})} = \int_{1/8}^{1} \d r \int_{\partial B_{r}}|N u| \d \mathcal{H}^{n-1}$,
by the mean value theorem
\begin{equation}
\label{mvthm}
\|N u\|_{L^{1}(B_1\setminus B_{1/8})} = \frac{7}{8} \int_{\partial B_{\rho}} |N u| \d \mathcal{H}^{n-1}
\end{equation}
for some $\rho \in [1/8, 1]$.
Let $v$ be the unique solution of the boundary value problem
\[
\left\{
\begin{array}{rl}
\div(A(x) \nabla v) + \dd(x) \cdot \nabla u
= 0 &\text{ in } B_{\rho}\\
v= u &\text{ on } \partial B_{\rho}.
\end{array}
\right.
\]
Since $ \div(A(x) \nabla (u-v)) \leq 0$ in $B_{\rho}$,
by the comparison principle $u \geq v$ in $B_{\rho}$.
Moreover, using that $u = v$ on $\partial B_{\rho}$, we deduce that $N u \leq N v$ on $\partial B_{\rho}$.
In particular, this gives
\begin{equation}
\label{normal:comparison}
(N v)^{-} \leq (N u)^{-} \text{ on } \partial B_{\rho}.
\end{equation}
Notice also that, integrating the equation 
$\div(A(x) \nabla v) = - \dd(x) \cdot \nabla u$ in $B_{\rho}$, 
by the divergence theorem
we have
\[
\left|
\int_{\partial B_{\rho}} N v \d \mathcal{H}^{n-1} 
\right|
= \left|\int_{B_{\rho}} \div\left( A(x) \nabla v\right) \d x \right|
= \left|- \int_{B_{\rho}} \dd(x) \cdot \nabla u \d x\right| \leq \varepsilon \|\nabla u\|_{L^1(B_{\rho})}
\]
and since $\int_{\partial B_{\rho}} N v \d \mathcal{H}^{n-1} = \|(N v)^{+}\|_{L^1(B_{\rho})} - \|(N v)^{-}\|_{L^1(B_{\rho})}$, we deduce
\begin{equation}
\label{zero:conormal}
\|(N v)^{+}\|_{L^1(B_{\rho})} \leq \|(N v)^{-}\|_{L^1(B_{\rho})} + \varepsilon \|\nabla u\|_{L^1(B_{\rho})}.
\end{equation}
Using \eqref{zero:conormal} and \eqref{normal:comparison} we have
\begin{equation}
\label{normall1comp}
\begin{split}
\|N v\|_{L^1(\partial B_{\rho})} &= \|(N v)^{-}\|_{L^1(\partial B_{\rho})} + \|(N v)^{+}\|_{L^1(\partial B_{\rho})} \leq 2 \|(N v)^{-}\|_{L^1(\partial B_{\rho})} + \varepsilon \|\nabla u \|_{L^{1}(B_{\rho})}\\
&\leq 2 \| N u \|_{L^1(\partial B_{\rho})} + \varepsilon \|\nabla u \|_{L^{1}(B_{\rho})}.
\end{split}
\end{equation}

Applying Lemma \ref{prop:l1:estimate}
with coefficients $A(\rho \, \cdot )$ and source $h(x) = \rho^2 \dd (\rho x) \cdot \nabla u (\rho x)$
to the function $v(\rho \, \cdot )$ yields the estimate
\begin{equation}
\label{dod1}
\|v- t\|_{L^1(\partial B_{\rho})} \leq C \rho \|N v\|_{L^1(\partial B_{\rho})} + C \rho \varepsilon \|\nabla u\|_{L^1(B_{\rho})}.
\end{equation}
Since 
$u - t = v - t$ on $\partial B_{\rho}$,
combining \eqref{dod1} and \eqref{normall1comp}, we obtain
\[
\|u- t\|_{L^1(\partial B_{\rho})} \leq 
C \rho \|N u \|_{L^1(\partial B_{\rho})}
+ C \rho \varepsilon \|\nabla u\|_{L^1(B_{\rho})},
\]
and since $\rho \in [1/8, 1]$, by \eqref{mvthm}, we deduce that
\begin{equation}
\label{dod3}
\|u- t\|_{L^1(\partial B_{\rho})} \leq 
C  \|N u \|_{L^1(B_{1}\setminus B_{1/8})}
+ C \varepsilon \|\nabla u\|_{L^1(B_{1})}.
\end{equation}

To conclude the proof,
it suffices to show that 
\begin{equation}
\label{dude:2}
\|u - t\|_{L^{1}(B_1 \setminus B_{1/8})} \leq C \|u -t \|_{L^{1}(\partial B_{\rho})} + C \|u_{r}\|_{L^1(B_{1} \setminus B_{1/8})}.
\end{equation}
Since
$(u-t)(s\sigma)=(u-t)(\rho\sigma)-\int_s^\rho u_r (r\sigma) \,dr$
for every $s\in (1/8,1)$ and $\sigma\in \partial B_1$, we have
\[
s^{n-1}|(u-t)(s\sigma)|\leq 8^{n-1}\rho^{n-1} |(u-t)(\rho\sigma)|+8^{n-1}\int_{1/8}^1 r^{n-1} |u_r (r\sigma)| \,dr.
\]
Integrating in $\sigma\in \partial B_1$, and then in $s\in (1/8,1)$, we deduce \eqref{dude:2}.
Combining \eqref{dod3} and \eqref{dude:2} yields the claim.
\end{proof}

\begin{proof}[Proof of Proposition \ref{prop:l1rad}]
We consider the operator $\lcal$ with $\dd(x) = \bb(x)$ given by \eqref{def:b}, 
so that $\lcal u = L u = \div(A(x) \nabla u) + \bb(x) \cdot \nabla u$.
Since $A(0) = \id$, writing the conormal derivative \eqref{def:conormal} as 
$N u = u_r + (A(x) - \id) \nabla u \cdot \frac{x}{|x|}$
and by the mean value theorem,
we have
$|N u| \leq |u_r| + C \varepsilon |\nabla u|$ in $B_1$.
Applying Proposition~\ref{prop:con:rad} now,
the conormal term on the right-hand side of the estimate can be bounded by 
$C\|u_r\|_{L^1(B_{1} \setminus B_{1/8})} + C \varepsilon \|\nabla u\|_{L^1(B_1)}$,
hence the claim.
\end{proof}

\section{Proof of the $C^{\alpha}$ estimate}
\label{section:holder}

This section is devoted to proving the H\"{o}lder regularity estimate \eqref{eq:holder} in Theorem~\ref{thm:main}.
The main goal will be to show that the scale-invariant weighted integral $\int_{B_{\rho}} r^{2-n} |\nabla u|^2 $ decays like a power $\rho^{2\alpha}$,
since this gives a $C^{\alpha}$ bound.
We will show this property
under the additional assumption that the operator $L$ is close to the Laplacian,
i.e.,
assuming $A(0) = \id$ and $\|D A\|_{C^{0}(\overline{B}_1)}+ \|\vv\|_{C^0(\overline{B}_1)} \leq \varepsilon$ with $\varepsilon$ sufficiently small.
An affine transformation will then lead to an estimate that is valid for all operators, with bounds depending on the norms of the coefficients.

The key idea is
to write the weighted integral of the gradient as an infinite sum on dyadic annuli,
pulling out the weights,
and applying Propositions~\ref{prop:l2l1} and~\ref{prop:l1rad} in each annulus.
This allows to control the weighted $L^2$ norm of the gradient by a weighted $L^2$ norm of the radial derivative.
Once 
we have this bound,
Proposition~\ref{prop:rad} will lead directly to the decay by a standard iteration argument.
This will yield a bound of the $C^{\alpha}$ norm in terms of the $L^2$ norm of the gradient,
which can be controlled by the $L^1$ norm of the solutions thanks to Proposition~\ref{prop:l2l1}.

\begin{proof}[Proof of the H\"{o}lder estimate \eqref{eq:holder} in Theorem~\ref{thm:main}]

We may assume that $3 \leq n \leq 9$.
Indeed, when $n = 2$,
we recover the estimate by applying Theorem~\ref{thm:main} to the function $\widetilde{u}(x_1, x_2, x_3) := u(x_1, x_2)$,
which is a stable solution to the elliptic equation $L \widetilde{u} + \elliptic \widetilde{u}_{x_3x_3} = f(\widetilde{u})$ in $B_1\subset \R^{3}$.
Similarly, when $n = 1$, one considers the function $\widetilde{u}(x_1, x_2, x_3) := u(x_1)$.

Throughout the proof, $C$ denotes a generic universal constant unless stated otherwise.
The proof is divided in three steps.

\vspace{3mm}\noindent
\textbf{Step 1:}
{\it 
Under the assumption that
\[
A(0) = \id \quad \text{ and } \quad \|D A\|_{C^0(\overline{B}_1)} + \|\vv\|_{C^0(\overline{B}_1)} \leq \varepsilon,
\]
we prove that there is a universal $\varepsilon_0 >0$ with the following property: 
if $\varepsilon \leq \varepsilon_0$, then 
\begin{equation}
\label{thedecay}
\int_{B_{\rho}} r^{2-n} |\nabla u|^2 \d x \leq C \|\nabla u\|_{L^2(B_1)} \rho^{2\alpha} \quad \text{ for all } \rho \leq 1/8,
\end{equation}
where $\alpha > 0$ and $C$ are universal constants.
}

As explained before,
we will write the weighted Dirichlet integral as an infinite sum on dyadic annuli,
similarly to what we did for the weighted Hessian estimates in the proof of Proposition~\ref{prop:rad}.
We treat the case $\rho = 1/2$ first,
and then apply the scaling of the problem.

Let $r_j := 2^{-j}$ with $j \geq 0$. 
We have
\begin{equation}
\label{ann1}
\begin{split}
\int_{B_{1/2}} r^{2-n} |\nabla u|^2 \d x &
= \sum_{j = 0}^{\infty} \int_{B_{r_{j+1}} \setminus B_{r_{j+2}}} r^{2-n} |\nabla u|^2 \d x\\
& \leq C \sum_{j = 0}^{\infty} r_j^{2-n} \int_{B_{r_{j+1}} \setminus B_{r_{j+2}}} |\nabla u|^2 \d x.
\end{split}
\end{equation}
We want to apply Proposition~\ref{prop:l2l1} on annuli
to control the Dirichlet integrals in \eqref{ann1} by the $L^1$ norm of the solution,
and then Proposition~\ref{prop:l1rad} to obtain bounds in terms of the radial derivative.

We cover the annulus $B_{1/2} \setminus B_{1/4}$
by a finite number of balls $B_{d/2}(y_j)$,
where $d = d(n)$ is small enough so that
$B_{d}(y_j) \subset  B_{1} \setminus B_{1/8}$.
The number of balls depends only on $n$.
As explained in Section~\ref{section:prelim},
the functions $u(y_j + d \cdot)$ are stable solutions to a semilinear equation with coefficients $A(y_j + d \cdot)$ and $d \, \vv(y_j.+ d \cdot)$.
Applying Proposition~\ref{prop:l2l1} to each $u(y_j + d \cdot)$, 
there is a universal $\varepsilon_0 > 0$ such that,
for $\varepsilon \leq \varepsilon_0$, we have
\begin{equation}
\label{almann}
\|\nabla u\|_{L^2(B_{1/2} \setminus B_{1/4})}^2
\leq \sum_{j} \|\nabla u\|_{L^2(B_{d/2}(y_j))}^2
\leq C \sum_{j} \| u\|_{L^1(B_{2d}(y_j))}^2
 \leq C \|u\|_{L^1(B_1 \setminus B_{1/8})}^2.
\end{equation}
For each $t \in \R$, the function $u - t$ is a stable solution to $- L \widetilde{u} = f(\widetilde{u} + t)$ in $B_1$.
Hence,
by \eqref{almann}, 
it follows that
\begin{equation}
\label{almann2}
\|\nabla u\|_{L^2(B_{1/2} \setminus B_{1/4})}
 \leq C \|u - t \|_{L^1(B_1 \setminus B_{1/8})} \quad \text{ for all } t \in \R \text{ and } \varepsilon \leq \varepsilon_0.
\end{equation}
Since $A(0) = \id$,
we can choose $t$ in \eqref{almann2}  to be the constant in the conclusion of Proposition~\ref{prop:l1rad},
and by this result we deduce 
\begin{equation}
\label{almann3}
\|\nabla u\|_{L^2(B_{1/2} \setminus B_{1/4})}
 \leq C \|u_r \|_{L^1(B_1 \setminus B_{1/8})} + C \varepsilon \|\nabla u\|_{L^1(B_1)} \quad \text{ for all } \varepsilon \leq \varepsilon_0.
\end{equation}
Squaring \eqref{almann3} and by Cauchy-Schwarz, 
we also have the weaker
\begin{equation}
\label{almann4}
\|\nabla u\|_{L^2(B_{1/2} \setminus B_{1/4})}^2
 \leq C \|u_r \|_{L^2(B_1 \setminus B_{1/8})}^2 + C \varepsilon^2 \|\nabla u\|^2_{L^2(B_1)} \quad \text{ for all } \varepsilon \leq \varepsilon_0.
\end{equation}

Now we apply \eqref{almann4} to the rescaled functions $u(r_j \cdot )$,
which gives (see the comments in \ref{scaling} in Section~\ref{section:prelim})
\begin{equation}
\label{almann5}
\int_{B_{r_{j+1}}\setminus B_{r_{j+2}}} |\nabla u|^2 \d x 
\leq C \int_{B_{r_{j}}\setminus B_{r_{j+3}}} u_r^2 \d x + C \varepsilon^2 r_j^2 \int_{B_{r_{j}}} |\nabla u|^2 \d x
\quad \text{ for all } \varepsilon \leq \varepsilon_0.
\end{equation}
Hence, multiplying \eqref{almann5} by $r_{j}^{2-n}$ and summing in $j$
\begin{equation}
\label{almann65}
\begin{split}
&\sum_{j = 0}^{\infty} r_{j}^{2-n}\int_{B_{r_{j+1}}\setminus B_{r_{j+2}}} |\nabla u|^2 \d x \\
&\quad \leq C \sum_{j = 0}^{\infty} r_{j}^{2-n}\int_{B_{r_{j}}\setminus B_{r_{j+3}}} u_r^2 \d x + C \varepsilon^2 \sum_{j = 0}^{\infty} r_j^{4-n} \int_{B_{r_{j}}} |\nabla u|^2 \d x\\
&\quad  \leq C \sum_{j = 0}^{\infty} \int_{B_{r_{j}}\setminus B_{r_{j+3}}} r^{2-n} u_r^2 \d x + C \varepsilon^2 \sum_{j = 0}^{\infty} r_j \int_{B_{r_{j}}} r^{3-n}|\nabla u|^2 \d x \quad \text{ for all } \varepsilon \leq \varepsilon_0,
\end{split}
\end{equation}
where in the last line we have used that $r_j^{3-n} \leq r^{3-n}$ in $B_{r_j}$ for $n \geq 3$.
Since $r_j = 2^{-j}$, 
splitting the annuli into $B_{r_{j}}\setminus B_{r_{j+3}} = (B_{r_{j}}\setminus B_{r_{j+1}}) \cup (B_{r_{j+1}}\setminus B_{r_{j+2}}) \cup (B_{r_{j+2}}\setminus B_{r_{j+3}})$,
we see that the first integral in the right-hand side of \eqref{almann65} is bounded by
\[
\sum_{j = 0}^{\infty} \int_{B_{r_{j}}\setminus B_{r_{j+3}}} r^{2-n} u_r^2 \leq 3 \int_{B_{1}} r^{2-n} u_r^2,
\]
while the second can be bounded by
\[
\sum_{j = 0}^{\infty} r_j \int_{B_{r_j}} r^{3-n} |\nabla u|^2 \leq \bigg(\sum_{j = 0}^{\infty} r_j \bigg) \int_{B_{1}} r^{3-n} |\nabla u|^2 = 2 \int_{B_{1}} r^{3-n} |\nabla u|^2.
\]
From this,
it follows that 
\begin{equation}
\label{almann6}
\begin{split}
&\sum_{j = 0}^{\infty} r_{j}^{2-n}\int_{B_{r_{j+1}}\setminus B_{r_{j+2}}} |\nabla u|^2 \d x \\
&\quad  \leq C \int_{B_{1}} r^{2-n} u_r^2 \d x + C \varepsilon^2 \int_{B_{1}} r^{3-n}|\nabla u|^2 \d x \quad \text{ for all } \varepsilon \leq \varepsilon_0.
\end{split}
\end{equation}
Combining \eqref{ann1} and \eqref{almann6} now yields
\begin{equation}
\label{ann2}
\begin{split}
\int_{B_{1/2}} r^{2-n} |\nabla u|^2 \d x \leq C \int_{B_{1}} r^{2-n} u_r^2 \d x + C \varepsilon^2 \int_{B_{1}} r^{3-n}|\nabla u|^2 \d x \quad \text{ for all } \varepsilon \leq \varepsilon_0,
\end{split}
\end{equation}
and applying \eqref{ann2} to rescaled functions $u(2 \rho \cdot )$, we deduce
\begin{equation}
\label{ann3}
\begin{split}
\int_{B_{\rho}} r^{2-n} |\nabla u|^2 \d x & \leq C \int_{B_{2\rho}} r^{2-n} u_r^2 \d x + C \varepsilon^2 \rho \int_{B_{2\rho}} r^{3-n}|\nabla u|^2 \d x\\ 
& \quad \quad \quad \quad \quad \quad \quad \quad \quad \quad  \text{ for all } \rho \leq 1/2 \text{ and } \varepsilon \leq \varepsilon_0.
\end{split}
\end{equation}

Next, we apply the radial estimate
\eqref{ineq:rad:decay} from
Proposition~\ref{prop:rad} (with $2\rho$)
to bound the right-hand side of \eqref{ann3},
which gives
\begin{equation}
\label{ann4}
\begin{split}
\int_{B_{\rho}} r^{2-n} |\nabla u|^2 \d x &\leq C \int_{B_{4\rho} \setminus B_{2\rho}} r^{2-n} |\nabla u|^2 \d x + C \varepsilon(1 + \varepsilon \rho) \int_{B_{8\rho}} r^{3-n}|\nabla u|^2 \d x\\
& \quad \quad \quad \quad \quad \quad \quad \quad \quad \quad \quad \quad \quad \quad \quad \quad  \text{ for all } \rho \leq 1/8 \text{ and } \varepsilon \leq \varepsilon_0.
\end{split}
\end{equation}
Hence, using that the bounds $\rho \leq 1/8$ and $\varepsilon \leq \varepsilon_0$ are universal, 
splitting the last integral into $B_{8\rho} = (B_{8\rho} \setminus B_{\rho}) \cup B_{\rho}$,
and by $r^{3-n} \leq r^{2-n}$,
from \eqref{ann4} we deduce
\begin{equation}
\label{ann5}
\begin{split}
\int_{B_{\rho}} r^{2-n} |\nabla u|^2 \d x 
&\leq C \int_{B_{8\rho} \setminus B_{\rho}} r^{2-n} |\nabla u|^2 \d x + C \varepsilon \int_{B_{\rho}} r^{2-n}|\nabla u|^2 \d x\\
& \quad \quad \quad \quad \quad \quad \quad \quad \quad \quad \quad \quad \quad \text{ for all } \rho \leq 1/8 \text{ and } \varepsilon \leq \varepsilon_0.
\end{split}
\end{equation}
Taking $\varepsilon_0 > 0$ universal smaller if necessary, we can absorb the last integral into the left-hand side and obtain
\begin{equation}
\label{ann6}
\begin{split}
\int_{B_{\rho}} r^{2-n} |\nabla u|^2 \d x 
&\leq C \int_{B_{8\rho} \setminus B_{\rho}} r^{2-n} |\nabla u|^2 \d x \quad \text{ for all } \rho \leq 1/8 \text{ and } \varepsilon \leq \varepsilon_0.
\end{split}
\end{equation}
Hole-filling \eqref{ann6}, we also have
\begin{equation}
\label{ann7}
\begin{split}
\int_{B_{\rho}} r^{2-n} |\nabla u|^2 \d x 
&\leq \theta \int_{B_{8\rho}} r^{2-n} |\nabla u|^2 \d x \quad \text{ for all } \rho \leq 1/8 \text{ and } \varepsilon \leq \varepsilon_0,
\end{split}
\end{equation}
where $\theta = \frac{C}{1 + C} \in (0, 1)$ is universal.
Iterating \eqref{ann7}, for $8^{-(k+1)} < \rho \leq 8^{-k}$ we deduce
\[
\begin{split}
\int_{B_{\rho}} r^{2-n} |\nabla u|^2 \d x 
&\leq \theta^k \int_{B_{8^k\rho}} r^{2-n} |\nabla u|^2 \d x \leq \frac{1}{\theta} \rho^{2\alpha} \int_{B_1} r^{2-n} |\nabla u|^2 \d x,
\end{split}
\]
where $\alpha = -\frac{1}{2} \log_{8} \theta > 0$, and hence
\begin{equation}
\label{ann8}
\int_{B_{\rho}} r^{2-n} |\nabla u|^2 \d x \leq C \rho^{2 \alpha}\int_{B_1} r^{2-n} |\nabla u|^2 \d x \quad \text{ for all } \rho \leq 1/8 \text{ and } \varepsilon \leq \varepsilon_0.
\end{equation}
Finally, we can estimate the 
integral in the right-hand side of \eqref{ann8}
by splitting $B_1 = (B_1 \setminus B_{1/8}) \cup B_{1/8}$ and applying \eqref{ann6} with $\rho = 1/8$
to bound the term in the annulus.
This yields the claim.

\vspace{3mm}\noindent
\textbf{Step 2:}
{\it 
Assuming
\[
\|D A\|_{C^0(\overline{B}_1)} + \|\vv\|_{C^0(\overline{B}_1)} \leq \varepsilon,
\]
we prove that if $\varepsilon \leq \varepsilon_0$, then 
\begin{equation}
\label{badholder}
\|u\|_{C^{\alpha}(\overline{B}_{\theta})} \leq C \| u\|_{L^1(B_1)},
\end{equation}
where $\alpha > 0$, $\theta > 0$, $\varepsilon_0 > 0$, and $C$ are universal.
}

As explained in Section~\ref{section:prelim},
for each ball $B_{d}(y) \subset B_1$,
by uniform ellipticity,
the function
$u^{y, d}(x) := u(y + \frac{d}{\sqrt{\bounded}} A^{1/2}(y) \, x)$
is a stable solution of an equation in $B_1$ with coefficients
\[
\textstyle
A^{y, d}(x) := 
A^{-1/2}(y) A\big(y + \frac{d}{\sqrt{\bounded}} A^{1/2}(y) x \big)A^{-1/2}(y)
\]
and
\[
\textstyle
\vv^{y, d}(x) := 
\frac{d}{\sqrt{\bounded}}A^{-1/2}(y) \vv\big(y + \frac{d}{\sqrt{\bounded}} A^{1/2}(y) x \big).
\]
Notice that the matrix $A^{y, d}$ satisfies $A^{y, d}(0) = \id$ and the coefficients can be bounded by
$\|D A^{y, d}\|_{C^0(\overline{B}_1)} + \|\vv^{y, d}\|_{C^0(\overline{B}_1)} \leq C d \left(\|D A\|_{C^0(\overline{B}_1)} + \|\vv\|_{C^0(\overline{B}_1)}\right)$.
Choosing $d > 0$ universal sufficiently small so that $C d \leq 1$, we have
\[
\|D A^{y, d}\|_{C^0(\overline{B}_1)} + \|\vv^{y, d}\|_{C^0(\overline{B}_1)} \leq \varepsilon \quad \text{ for all } y \in B_{1-d}.
\]
Hence, for $\varepsilon \leq \varepsilon_0$ with the $\varepsilon_0 > 0$ from Step 1,
by \eqref{thedecay} we deduce
\[
\int_{B_{\rho}}r^{2-n}|\nabla u^{y, d}|^2 \d x \leq C \|\nabla u^{y, d}\|_{L^2(B_1)}^2 \rho^{2\alpha} \quad \text{ for } y \in B_{1-d} \, \text{ and } \, \rho \leq 1/8,
\]
and since $\int_{B_{\rho}}r^{2-n}|\nabla u^{y, d}|^2 \d x \geq \rho^{2-n} \int_{B_{\rho}}|\nabla u^{y, d}|^2 \d x$, we also have
\begin{equation}
\label{badalm}
\int_{B_{\rho}}|\nabla u^{y, d}|^2 \d x \leq C \|\nabla u^{y, d}\|_{L^2(B_1)}^2 \rho^{2\alpha+n-2} \quad \text{ for } \, y \in B_{1-d} \, \text{ and }\,  \rho \leq 1/8.
\end{equation}
For the remaining part of the proof of Step 2, we assume that $\varepsilon \leq \varepsilon_0$.

Now we express \eqref{badalm} in terms of the original function $u$.
By the change of variables $z = y + \frac{d}{\sqrt{\bounded}} A^{1/2}(y) x$
and by uniform ellipticity, 
using that $B_{\sqrt{\elliptic }\rho} \subset A^{1/2}(y) (B_{\rho})$,
we have
\[
\begin{split}
\int_{B_{\rho}} |\nabla u^{y, d}|^2 \d x &= \frac{d^{2-n}}{\bounded^{1-n/2}} \det(A(y))^{-1/2} \int_{y + \frac{d}{\sqrt{\bounded}}A^{1/2}(y) (B_{\rho})} |\nabla u|_{A(y)}^2 \d z\\
&\geq c \, d^{2-n} \int_{B_{d \sqrt{\frac{\elliptic}{\bounded}}\rho}(y)} |\nabla u|^2 \d z
\end{split}
\]
for some universal $c > 0$.
Similarly, we also have $\|\nabla u^{y, d}\| \leq C d^{2-n} \|\nabla u\|_{L^2(B_1)}$
and,
therefore, from \eqref{badalm} we deduce
\begin{equation}
\label{badalm2}
\int_{B_{d \sqrt{\frac{\elliptic}{\bounded}}\rho}(y)} |\nabla u|^2 \d z \leq C \|\nabla u\|_{L^2(B_1)}^2 \rho^{n-2+2\alpha} \quad \text{ for }  y \in B_{1-d} \, \text{ and } \,  \rho \leq 1/8.
\end{equation}

Dividing $\rho$ by $d \sqrt{\frac{\elliptic}{\bounded}}$ in \eqref{badalm2}
and letting $\theta := \frac{d}{16} \sqrt{\frac{\elliptic}{\bounded}}$, since $d$ is universal, we obtain
\[
\int_{B_{\rho}(y)} |\nabla u|^2 \d z \leq C \|\nabla u\|_{L^2(B_1)}^2 \rho^{n-2+2\alpha} 
\quad \text{ for }  y \in B_{1-d} \, \text{ and } \,  \rho \leq 2\theta,
\]
and 
by Cauchy-Schwarz we also have the weaker
\begin{equation}
\label{badalm4}
\int_{B_{\rho}(y)} |\nabla u| \d z \leq C \|\nabla u\|_{L^2(B_1)} \rho^{n-1+\alpha} 
\quad \text{ for }   y \in B_{1-d} \, \text{ and } \,  \rho \leq 2\theta.
\end{equation}

Taking $d$ smaller 
if necessary, we may assume that $B_{2\theta} \subset B_{1-d}$.
Hence, from \eqref{badalm4} it follows that 
\begin{equation}
\label{badalm5}
\int_{B_{\rho}(y)} |\nabla u| \d z \leq C \|\nabla u\|_{L^2(B_1)} \rho^{n-1+\alpha} 
\quad \text{ for all balls } B_{\rho}(y) \subset B_{2\theta}.
\end{equation}
Applying 
\cite{GilbargTrudinger}*{Theorem~7.19}
with $\Omega = B_{2\theta}$, we deduce the 
H\"{o}lder estimate
\begin{equation}
\label{holderlast}
\|u\|_{C^{\alpha}(B_{2\theta})} \leq C \|\nabla u\|_{L^2(B_1)}.
\end{equation}

To obtain the final bound \eqref{badholder} in terms of the $L^1$ norm,
apply \eqref{holderlast} to the rescaled function $u(\cdot/2)$ first,
and then Proposition~\ref{prop:l2l1} (taking $\varepsilon_0$ smaller if necessary).

\vspace{3mm}\noindent
\textbf{Step 3:}
{\it Conclusion. 
Scaling and covering argument.
}

We cover $\overline{B}_{1/2}$ by balls $B_{\theta \rho}(y_j)$,
where $\theta$ is the universal constant in Step 2 above
and $\rho$ is small so that $B_{\rho}(y_j) \subset B_1$.
The number of balls depends only on $n$, $\rho$, and $\theta = \theta(n, \elliptic, \bounded)$.
We choose $\rho$ 
smaller still
so that 
\begin{equation}
\label{assrho}
\left(\|D A\|_{C^0(\overline{B}_1)} + \|\vv\|_{C^0(\overline{B}_1)} \right) \rho \leq \varepsilon_0,
\end{equation}
with $\varepsilon_0 > 0$ the universal constant in Step~2.
Thus
$\rho = \rho(n, \elliptic, \|A\|_{C^1(\overline{B}_1)}, \|\vv\|_{C^0(\overline{B}_1)})$.
The functions $u(y_j + \rho \cdot)$
are stable solutions of an elliptic equation with coefficients
$A^{y_j, \rho} = A(y_j + \rho \cdot)$ and $\vv^{y_j, \rho} =  \rho \vv(y_j + \rho \cdot)$.
Since $B_{\rho}(y_j) \subset B_1$ and by \eqref{assrho},
the coefficients satisfy the bounds
\[
\begin{split}
\|D A^{y_j, \rho}\|_{C^0(\overline{B}_1)} + \|\vv^{y_j, \rho}\|_{C^0(\overline{B}_1)} & \leq \left(\|D A\|_{C^0(\overline{B}_{\rho}(y_j))} + \|\vv\|_{C^0(\overline{B}_{\rho}(y_j))}\right) \rho 
\leq \varepsilon_0,
\end{split}
\]
therefore, we can apply Step 2 to deduce
\[
\|u\|_{C^{\alpha}(\overline{B}_{1/2})} \leq \sum_{j} \|u\|_{C^{\alpha}(\overline{B}_{\theta \rho}(y_j))} 
\leq C \sum_{j} \|u\|_{L^{1}(B_{\rho}(y_j))} \leq C \|u\|_{L^1(B_1)},
\]
where $C = C(n, \elliptic, \|A\|_{C^1(\overline{B}_1)}, \|\vv\|_{C^0(\overline{B}_1)})$.
This concludes the proof of the theorem.
\end{proof}

\appendix

\medskip\medskip

\section{Stability is not equivalent to the integral inequality}
\label{app:equivalence}

Let $u \in C^{2}(\overline{\Omega})$ be a solution to $- L u = f(u)$ in $\Omega$ with $u = 0$ on $\partial \Omega$.
Recall that $u$ is a stable solution if
\begin{equation}
\label{stable:cond}
\jacobi \varphi = L \varphi + f'(u) \varphi \leq 0 \quad \text{ in } \Omega,
\end{equation}
for some function $\varphi \in C^{2}(\overline{\Omega})$ with $\varphi > 0$ in $\Omega$ and $\varphi = 0$ on $\partial \Omega$.
This is the stability condition \eqref{stable:point}
presented in the Introduction
and is equivalent to the nonnegativity of the first Dirichlet eigenvalue of $\jacobi$ (with the sign convention $\jacobi \varphi = - \mu \varphi$).
There, we also showed
that stable solutions satisfy the integral inequality \eqref{stable:int}, which reads
\begin{equation}
\label{integral:cond}
\int_{\Omega} f'(u) \xi^2 \d x \leq \int_{\Omega} |\nabla \xi - \textstyle\frac{1}{2} A^{-1}(x) \bb(x) \xi|^2_{A(x)} \quad \text{ for all } \xi \in C^{\infty}_c(\Omega).
\end{equation}

Our goal in this appendix is to show that the integral inequality \eqref{integral:cond} 
does not imply the stability condition \eqref{stable:cond} in general.
The main reason is that 
the problem is not variational, due to the drift in $L$.
We also give conditions under which the equivalence holds.
Namely, writing the operator in divergence form $L u= \div(A(x) \nabla u) + \bb(x) \cdot \nabla u$, we show that if $A^{-1}(x) \bb(x)$
is the gradient of a scalar function, then the problem is variational and the two conditions are equivalent. 

First we write the integrals in \eqref{integral:cond} as the quadratic form associated to a linear self-adjoint operator.
Integrating by parts, we have 
\begin{equation}
\label{2ndvar}
\int_{\Omega} \left(|\nabla \xi - \textstyle\frac{1}{2}\xi A^{-1}(x) \bb(x) |_{A(x)}^2 - f'(u) \xi^2 \right)\d x
= - \displaystyle \int_{\Omega} \xi \jacobinew \xi \d x,
\end{equation}
where $\jacobinew$ is the operator
\begin{equation}
\label{2ndjacobi}
\jacobinew \xi := 
\div(A(x) \nabla \xi) - 
\Big\{
\textstyle\frac{1}{2} \div(\bb(x)) + \textstyle\frac{1}{4} |\bb(x)|_{A^{-1}(x)}^2 
\Big\}
\xi 
+ f'(u) \xi.
\end{equation}
Hence, 
by the variational characterization of eigenvalues, 
\eqref{integral:cond} amounts to the nonnegativity of the principal eigenvalue of $\jacobinew$.

We can now state our example
of a solution satisfying \eqref{integral:cond} but not \eqref{stable:cond}:
\begin{example}
\label{counterexample}
Consider the operator
$L v = \Delta v + \bb(x) \cdot \nabla v$
with vector field
\[
\bb(x) 
= \frac{-x_2 e_1 + x_1 e_2}{\sqrt{x_1^2 + x_2^2}}.
\]
For each constant $c > 0$, we let $f(u) = (\lambda_1 + c) u + 1$,
where $\lambda_1$ denotes the least Dirichlet eigenvalue of the Laplacian in the unit ball $B_1$.

If 
$c > 0$ is sufficiently small, then the unique solution $u$ to the boundary value problem
\[
\left\{
\begin{array}{ll}
-L u = f(u) & \text{ in } B_1 \\
u = 0 & \text{ on } \partial B_1
\end{array}
\right.
\]
satisfies the integral stability condition \eqref{integral:cond} but is not a stable solution, i.e., the stability condition \eqref{stable:cond} does not hold.\footnote{The function $u$ can be given explicitly in terms of Bessel functions of the first kind $\mathcal{J}_{\alpha}$ as \[ u(x) = \frac{1}{(\lambda_1 + c) \mathcal{J}_{\frac{n-2}{2}}\left(\sqrt{\lambda_1 + c}\right)} |x|^{\frac{2-n}{2}} \mathcal{J}_{\frac{n-2}{2}}\left(\sqrt{\lambda_1 + c} \, |x|\right)  - \frac{1}{\lambda_1 + c}.\]}
\end{example}
 
\begin{proof}
The problem for $u$
is equivalent to 
solving
\begin{equation}
\label{upb}
\left\{
\begin{array}{ll}
- \Delta u - \bb(x) \cdot \nabla u - (\lambda_1 + c) u = 1 & \text{ in } B_1\\
u = 0 & \text{ on } \partial B_1.
\end{array}
\right.
\end{equation}
Notice that the drift 
$\bb \in L^{\infty}(B_1)$
has a weak derivative 
$D \bb \in L^p(B_1)$ for $1 \leq p < 2$,
and satisfies the identities $|\bb(x)| = 1$ and $\div \, \bb(x) = 0$ for a.e. $x \in B_1$.
Moreover, since $\bb$ is tangent to spheres, the derivative $\bb(x) \cdot \nabla $ vanishes on radial functions.
In particular,
the principal eigenfunction of the Laplacian
is also an eigenfunction of the adjoint operator 
$L^{T} = \Delta - \div\big( \bb(x) \cdot \big)= \Delta - \bb(x) \cdot \nabla$,
with eigenvalue $\lambda_1$.
Since the point spectrum of $L^{T}$ is discrete, for $c > 0$ small, we deduce that $\lambda_1 + c$ is not an eigenvalue of the adjoint operator.
The Fredholm alternative now gives that \eqref{upb} has a unique solution.

Let $\varphi_1$ and $\xi_1$ be positive principal eigenfunctions of $\jacobi$ and $\jacobinew$, respectively.
Since $\varphi_1$ and $\xi_1$ are positive in $B_1$, they 
must be radial.
It follows that 
\[
\jacobi \varphi_1 = \Delta \varphi_1 + (\lambda_1 + c) \varphi_1 = -\mu_1 \varphi_1 \quad \text{ and } \quad 
\jacobinew \xi_1 = \Delta \xi_1 + (\lambda_1 + c - 1/4) \xi_1 = - \widetilde{\mu_1} \xi_1,
\]
where $\mu_1$ and $\widetilde{\mu_1}$ are the least eigenvalues of each operator.
By uniqueness, the functions are multiples of the principal eigenfunction of the Laplacian.
Therefore, 
we have
$\mu_1 = -c < 0$ 
and $\widetilde{\mu_1} = 1/4- c > 0$ for $c$ sufficiently small.
This means that $u$ is not stable but \eqref{integral:cond} holds, which was the claim.
\end{proof}

Next we investigate the relation between the failure of the equivalence
and the form of the drift $\bb$.
Let $\varphi_1\in C^{2}(\overline{\Omega})$ be the unique positive principal eigenfunction of $\jacobi$ with $\int \varphi_1^2 \d x = 1$.
In particular, the function satisfies 
$\varphi_1 > 0$ in $\Omega$, $\varphi_1 = 0$ on $\partial \Omega$, and $\jacobi \varphi_1 = - \mu_1 \varphi_1$,
where $\mu_1 \in \R$ is the least eigenvalue of $\jacobi$.
Consider a test function $\xi \in C^{\infty}_c(\Omega)$.
Multiplying $\jacobi \varphi_1$ by $\xi^2/ \varphi_1$
and integrating by parts in $\Omega$, we have
\[
\begin{split}
- \mu_1 & = \int_{\Omega} (\jacobi \varphi_1) \frac{\xi^2}{\varphi_1} \d x = \int_{\Omega} \left( - A(x) \nabla \varphi_1 \cdot \nabla \left( \frac{\xi^2}{\varphi_1}\right) + \bb(x) \cdot \frac{\xi^2}{\varphi_1} \nabla \varphi_1 + f'(u) \xi^2  \right) \d x\\
& = \int_{\Omega} \left(
|\xi \nabla \log{\varphi_1} |_{A(x)}^2
- 2 A(x) \xi \, \nabla \log{\varphi_1} \cdot \nabla \xi
+ \xi \, \bb(x) \cdot \xi \, \nabla \log{\varphi_1} + f'(u) \xi^2
\right) \d x.
\end{split}
\]
Using that
\[
\begin{split}
&| \xi \nabla \log{\varphi_1}|_{A(x)}^2
-2 A(x) \xi \, \nabla \log{\varphi_1} \cdot \left( \nabla \xi-{\textstyle\frac{1}{2}} \xi A^{-1}(x) \bb(x)\right)\\
&\quad \quad = \left|\nabla \xi - \textstyle\frac{1}{2} \xi A^{-1}(x) \bb(x) - \xi \nabla \log{\varphi_1}\right|^2_{A(x)} -\left|\nabla \xi-{\textstyle\frac{1}{2}} \xi A^{-1}(x) \bb(x)\right|^2_{A(x)},
\end{split}
\]
in the integral above, by \eqref{2ndvar}
we obtain the identity
\begin{equation}
\label{equivalence}
\begin{split}
- \mu_1 = 
\int_{\Omega} \xi \jacobinew \xi \d x
+ \int_{\Omega} \left|\nabla \xi - \textstyle\frac{1}{2} \xi A^{-1}(x) \bb(x) - \xi \nabla \log{\varphi_1}\right|^2_{A(x)} \d x.
\end{split}
\end{equation}

Now, assuming the integral stability inequality \eqref{integral:cond}, we can minimize \eqref{2ndvar} among smooth functions $\xi$ with 
$\xi = 0$ on $\partial \Omega$
and
$\int_{\Omega} \xi^2 \d x = 1$.
The unique positive minimizer $\xi_1$ 
satisfies $\jacobinew \xi_1 = - \widetilde{\mu_1} \xi_1$,
where $\widetilde{\mu_1} \geq 0$ is the least eigenvalue of $\jacobinew$.
Letting $\xi = \xi_1$ in \eqref{equivalence} yields
\begin{equation}
\label{equivalence3}
\begin{split}
- \mu_1 = - \widetilde{\mu_1} + \int_{\Omega} \big|\nabla \big(\log{\xi_{1}} - \log{\varphi_1} \big) - \textstyle\frac{1}{2} A^{-1}(x) \bb(x)\big|^2_{A(x)} \xi_1^2 \d x,
\end{split}
\end{equation}
and from \eqref{equivalence3} we see that we always have $\mu_1 \leq \widetilde{\mu_1}$,
with equality if and only if 
\begin{equation}
\label{equality:case}
\nabla \log \left(\frac{\xi_1}{\varphi_1}\right) = \frac{1}{2} A^{-1}(x) \bb(x).
\end{equation}
This can only happen when the drift $\bb$ is of a special form.
Notice that the vector field from Example \ref{counterexample}
is the curl of $\sqrt{x_1^2 + x_2^2} \, e_3$
and so, by the Helmholtz decomposition, cannot be written as the gradient of a function.

Conversely,
assume that $\bb(x) = A(x) \nabla w(x)$ for some function $w \in C^{2}(\overline{\Omega})$.
In this case, the problem can be cast in variational form and conditions \eqref{stable:cond} and \eqref{integral:cond} are equivalent.
Indeed, the solutions of $- L u = f(u)$ in $\Omega$
are critical points of the functional 
$\mathcal{E}(u) = \int_{\Omega} e^{w(x)} \left( \textstyle \frac{1}{2}|\nabla u|_{A(x)}^2 - F(u)\right) \d x$,
where $F(u) = \int_{0}^{u} f(t) \d t$.
The integral stability inequality \eqref{integral:cond} amounts to
the nonnegativity of the second variation
\[
\frac{\d^2 }{\d^2 t} \Big|_{t = 0} \mathcal{E}\big(u + t \varphi \big) 
= \int_{\Omega} e^{w(x)}\Big( |\nabla \varphi|^2_{A(x)} - f'(u) \varphi^2 \Big) \d x = - \int_{\Omega} e^{w(x)} \varphi \jacobi \varphi \d x
\]
since, letting $\varphi = e^{-w/2} \xi$ in this expression, we have
\[
\frac{\d^2 }{\d^2 t} \Big|_{t = 0} \mathcal{E}\big(u + t e^{-w/2} \xi\big) 
= \int_{\Omega} \Big( |\nabla \xi - \textstyle\frac{1}{2} \xi \nabla w(x) |^2_{A(x)} - f'(u) \xi^2 \Big) \d x = \displaystyle- \int_{\Omega} \xi \jacobinew \xi \d x.
\]
In particular, 
since $- \int_{\Omega} \xi \jacobinew \xi \d x \geq \widetilde{\mu_1}\|\xi\|_{L^2(\Omega)}^2$ 
and taking $\varphi = \varphi_1$ to be the principal eigenfunction of $\jacobi$ above,
we have
\[
\mu_1 \int_{\Omega} e^{w(x)} \varphi_1^2 \d x = - \int_{\Omega} (e^{w/2} \varphi_1) \jacobinew (e^{w/2} \varphi_1) \d x \geq \widetilde{\mu_1}
\int_{\Omega} e^{w(x)} \varphi_1^2 \d x
\]
and we obtain the reverse inequality $\mu_1 \geq \widetilde{\mu_1}$.

\section{A trace inequality}
\label{app:trace:sobolev}

First we prove a simple lemma to control the $L^p$ norm in the ball by the $L^p$ norms of the trace and the gradient:
\begin{lemma}
\label{lemma:obs}
For $p \geq 1$ and $u\in W^{1,p}(B_1)$, we have
\[
\|u\|_{L^{p}(B_1)}^p \leq 2^{p-1} \left(  \| u\|_{L^{p}(\partial B_1)}^p + \|\nabla u\|_{L^p(B_1)}^p \right).
\]
\end{lemma}

\begin{proof}
By approximation, we may assume that $u \in C^{\infty}(\overline{B}_1)$.
For $r \in (0,1)$ and $\sigma \in \partial B_1$, we have
$u(r \sigma) = u(\sigma) - \int_{r}^{1} \sigma \cdot \nabla u(t \sigma) \d t$
and hence
\begin{equation}
\label{obs2}
\begin{split}
r^{n-1} |u(r \sigma)|^p &\leq 2^{p-1} r^{n-1}|u(\sigma)|^p + 2^{p-1} r^{n-1} \int_{r}^{1} |\nabla u(t \sigma)|^p \d t\\
& \leq 2^{p-1} |u(\sigma)|^p + 2^{p-1} \int_{0}^{1} t^{n-1} |\nabla u(t \sigma)|^p \d t.
\end{split}
\end{equation}
Integrating \eqref{obs2} in $\int_{0}^{1} \d r \int_{\partial B_1} \d \mathcal{H}^{n-1}(\sigma)$ now yields the claim.
\end{proof}

We prove a Sobolev trace inequality with best exponent:

\begin{proposition}
\label{prop:trace}
For $1 < p < n$,
let $p^{\star} := \frac{n-1}{n-p} p$.
Then 
\[
\|u\|_{L^{p^{\star}}(\partial B_1)}^p \leq C \left(\|u\|_{L^{p}(\partial B_1)}^p + \|\nabla u\|_{L^{p}(B_1)}^p \right) 
\]
for all $u \in W^{1, p}(B_1)$,
where $C$ is a constant depending only on $n$ and $p$.
\end{proposition}
\begin{proof}
By approximation, we may assume that $u \in C^{\infty}(\overline{B}_{1})$.
Recall
the standard Sobolev inequality
\begin{equation}
\label{std:sobolev}
\|u\|_{L^{p_{S}}(B_1)}^{p} \leq C \big(\|u\|_{L^p(B_1)}^p + \|\nabla u\|_{L^p(B_1)}^p \big),
\end{equation}
where $p_{S} := \frac{n}{n-p}p$
is the Sobolev exponent
and $C$ depends only on $n$ and $p$.

By the divergence theorem we have
\[
\int_{\partial B_1} |u|^{p^{\star}} \d \mathcal{H}^{n-1} = \int_{B_1} \div(x |u|^{p^{\star}}) \d x = n \int_{B_1} |u|^{p^{\star}} \d x + p^{\star} \int_{B_1} |u|^{p^{\star}-2} u(x \cdot \nabla u) \d x,
\]
whence
\begin{equation}
\label{trace:1}
\int_{\partial B_1} |u|^{p^{\star}} \d \mathcal{H}^{n-1} \leq n \int_{B_1} |u|^{p^{\star}} \d x + p^{\star} \int_{B_1} |u|^{p^{\star}-1} |\nabla u| \d x.
\end{equation}
The last term in \eqref{trace:1} can be bounded by the H\"{o}lder inequality as
\[
\int_{B_1} |u|^{p^{\star}-1} |\nabla u| \d x \leq \left(\int_{B_1} |u|^{(p^{\star}-1) \frac{p}{p-1} } \d x \right)^{\frac{p-1}{p}} \|\nabla u\|_{L^p(B_1)},
\]
and noticing that $(p^{\star}-1)\frac{p}{p-1} = p_{S}$
we deduce
\begin{equation}
\label{trace:3}
 \|u\|_{L^{p^{\star}}(\partial B_1)}^{p^{\star}} \leq 
n \|u\|_{L^{p^{\star}}(B_1)}^{p^{\star}} + 
p^{\star} \|\nabla u\|_{L^p(B_1)}\|u\|_{L^{p_{S}}(B_1)}^{p^{\star}-1}.
\end{equation}
Since $p^{\star} < p_{S}$,
by H\"{o}lder we have $\|u\|_{L^{p^{\star}}(B_1)} \leq C \|u\|_{L^{p_S}(B_1)}$,
and applying the Sobolev inequality \eqref{std:sobolev} 
in \eqref{trace:3}, we obtain the trace Sobolev inequality
\begin{equation}
\label{trace:sob}
\|u\|_{L^{p^{\star}}(\partial B_1)}^p\leq
C \big(\|u\|_{L^p(B_1)}^p + \|\nabla u\|_{L^p(B_1)}^p \big),
\end{equation}
where $C$ depends only on $n$ and $p$.
Applying Lemma~\ref{lemma:obs} in \eqref{trace:sob} now yields the claim.
\end{proof}

\section{Two interpolation inequalities}
\label{app:interpolation}

We recall two interpolation inequalities in cubes by Cabr\'{e}~\cite{CabreQuant}.
The first one states that the $L^2$ norm of the gradient 
can be bounded in terms of the $L^1$ norm of 
the ``Hessian times the gradient'' and the $L^2$ norm of the function.
In the second one, the $L^2$ norm of the function is controlled by the $L^2$ norm of the gradient and the $L^1$ norm of the function.

\begin{proposition}[X. Cabr\'{e} \cite{CabreQuant}]
\label{interpol}
Let $Q = (0,1)^n \subset \R^{n}$ and $u \in C^2(\overline{Q})$.

Then, for every $\delta \in (0,1)$,
\[
\| \nabla u \|_{L^2(Q)}^2 \leq 2 n \delta \| \,|\nabla u| \, D^2 u \, \|_{L^1(Q)} + n \left(\frac{18}{\delta}\right)^2 \|u\|_{L^2(Q)}^2.
\]
\end{proposition}

\begin{proposition}[X. Cabr\'{e} \cite{CabreQuant}]
\label{nash}
Let $Q = (0,1)^n \subset \R^{n}$ and $u \in C^2(\overline{Q})$.

Then, for every $\widetilde{\delta} \in (0,1)$,
\[
\| u\|_{L^2(Q)}^2 \leq 2 n^2 \widetilde{\delta}^2 \|\nabla u\|_{L^2(Q)}^2 + 2^{n+1} \widetilde{\delta}^{-n} \|u\|_{L^1(Q)}^2.
\]
\end{proposition}

\section{Absorbing errors in larger balls}
\label{app:simon}

We recall the following well-known lemma of Simon~\cite{Simon},
which allows to absorb errors in larger balls when controlling quantities in smaller balls:

 \begin{lemma}[L. Simon \cite{Simon}]
\label{lemma:simon}
Let $\beta\geq 0$ and $C_0>0$. 
Let  $\mathcal{B}$ be the class of all open balls $B$ contained in the unit ball $B_1$ of $\R^n$ and let 
$\sigma \colon \mathcal{B} \rightarrow [0,+\infty)$ satisfy the following subadditivity property:
\[
\sigma(B)\leq \sum_{j=1}^N \sigma(B^j) \quad \mbox{ whenever }  N\in \mathbb{Z}^+, \{B^j\}_{j=1}^N \subset \mathcal{B}, \text{ and } B \subset \bigcup_{j=1}^N B^j. 
\]

It follows that there exists a constant $\delta>0$, which depends only on $n$ and $\beta$, such that if
\[
\rho^\beta \sigma\left(B_{\rho/2}(y)\right) \leq \delta \rho^\beta \sigma\left(B_\rho(y)\right)+ C_0\quad \mbox{whenever }B_\rho(y)\subset B_1,
\]
then
\[ 
\sigma(B_{1/2}) \leq C C_0
\]
for some constant $C$ which depends only on $n$ and $\beta$.
\end{lemma} 

\section*{Acknowledgements}

The author thanks Xavier Cabr\'{e} for his patient guidance and useful discussions on the topic of this article.


\end{document}